\documentclass[a4paper]{article}

\usepackage{amsmath,amssymb}
\usepackage{amscd}
\usepackage{mathrsfs}
\usepackage{amsthm}
\usepackage{color}

\theoremstyle{plain}

\newtheorem{theorem}{\bf Theorem}[section]
\newtheorem{lemma}[theorem]{\bf Lemma}
\newtheorem{proposition}[theorem]{\bf Proposition}
\newtheorem{corollary}[theorem]{\bf Corollary}

\theoremstyle{definition}
\newtheorem{definition}[theorem]{\bf Definition}
\newtheorem{example}[theorem]{\bf Example}
\newtheorem{remark}[theorem]{\bf Remark}

\newtheorem{question}[theorem]{\bf {Question}}

\newcommand{\eqa}[1]{
\begin{align*}
#1
\end{align*}}

\newcommand{\nai}[2]{\langle #1,#2\rangle}
\newcommand{\dom}[1]{{{\rm{dom}}{(#1)}}}
\newcommand{\mattwo}[4]{\begin{pmatrix}#1 & #2\\ #3 & #4\end{pmatrix}}

\title{Ultraproducts of von Neumann Algebras}
\author{Hiroshi ANDO\\
\and Uffe HAAGERUP}

\begin{document}
\maketitle

\begin{abstract}
We study several notions of ultraproducts of von Neumann algebras from a unifying viewpoint. In particular, we show that for a $\sigma$-finite von Neumann algebra $M$, the ultraproduct $M^{\omega}$ introduced by Ocneanu is a corner of the ultraproduct $\prod^{\omega}M$ introduced by Groh and Raynaud. Using this connection, we show that the ultraproduct action of the modular automorphism group of $\varphi$ on the Ocneanu ultraproduct is the modular automorphism group of the ultrapower state $\varphi^{\omega}$, i.e., $(\sigma_t^{\varphi})^{\omega}=\sigma_t^{\varphi^{\omega}}\ (t\in \mathbb{R})$ holds on $M^{\omega}$. Moreover, we show that Golodets' asymptotic algebra $C_M^{\omega}$ is isomorphic to $M'\cap M^{\omega}$, and his auxiliary algebra $\mathscr{R}$ is isomorphic to $M^{\omega}$. 
From these results, we show the following consequences. 
(1) Ueda's question asking whether a full factor $M$ has trivial central sequence algebra, i.e, whether $M_{\omega}=\mathbb{C}$ implies $M'\cap M^{\omega}=\mathbb{C}$, has an affirmative answer when $M$ has separable predual. Here, $M_{\omega}$ denotes Connes' asymptotic centralizer.
(2) Golodets' Theorem (rephrased in terms of the Ocneanu ultrapower): in an analogy to McDuff property, for a $\sigma$-finite type III factor $M$, the restriction $\dot{\varphi}$ of the ultrapower state $\varphi^{\omega}$ onto the central sequence algebra $M'\cap M^{\omega}$ is independent of the choice of a normal faithful state $\varphi$, and for $0<\lambda<1$, $M$ absorbs Powers factor $R_{\lambda}$ tensorially $M\overline{\otimes}R_{\lambda}\cong M$, if and only if $\lambda$ is an eigenvalue of $\Delta_{\dot{\varphi}}$.
(3) The Ocneanu ultrapower $M^{\omega}$ of type III$_{\lambda}$ factor $(0<\lambda\le 1)$ is again a type III$_{\lambda}$ factor. 
The same conclusion holds for the Groh-Raynaud ultrapower $\prod^{\omega}M$ too. However, $M^{\omega}$ can not be a factor if $M$ is a type III$_0$ factor. Moreover, $\prod^{\omega}M$ has a semifinite type direct summand in that case. 
We also show that although the ultrapower $R^{\omega}$ of the hyperfinite type II$_1$ factor $R$ is a type II$_1$ factor, $\prod^{\omega}R$ is not semifinite (and is not a factor).
(4) For a $\sigma$-finite type III$_1$ factor $M$, the Ocneanu ultrapower $M^{\omega}$ has strictly homogeneous state space, i.e.,  any two normal faithful states $\varphi, \psi$ on $M^{\omega}$ are unitarily equivalent. 
\end{abstract}


\medskip
\tableofcontents
\section{Introduction}
The purpose of this paper is to study several notions of ultraproducts and central sequence algebras of von Neumann algebras which are not necessarily of finite type. Since it does not seem to be well-known that there are various notions of ultraproducts, let us start from an overview of the history. 
   
The notions of central sequences and ultraproducts play a central role in the study of operator algebras and their automorphisms. The importance of central sequences was already recognized as early as in Murray-von Neumann's work \cite[$\S$6]{MurrayVonNeumann} on rings of operators. After establishing the uniqueness of the hyperfinite type II$_1$ factor $R$, they tried to prove the existence of non-isomorphic type II$_1$ factors. In $(R,\tau)=\bigotimes_{\mathbb{N}} (M_2(\mathbb{C}),\frac{1}{2}\text{Tr})$ (Tr denotes the usual trace on $M_2(\mathbb{C})$), consider a sequence 
\[u_n=1^{\otimes n}\otimes \mattwo{0}{1}{1}{0}\otimes 1\otimes \cdots ,\ n\in \mathbb{N}.\]
$\{u_n\}_{n
=1}^{\infty}$ satisfies 
\begin{itemize}
\item[(1)] $\sup_n\|u_n\|<\infty$. 
\item[(2)] $u_na-au_n\to 0$ strongly for any $a\in R$. 
\item[(3)] $\tau(u_n)=0$ for all $n\in \mathbb{N}$ and $\|u_n\|_2\not\to 0$.
\end{itemize}
A sequence of operators $\{x_n\}_{n=1}^{\infty}$ in a finite von Neumann algebra $M$ is called a {\it central sequence} if it satisfies (1) and (2), and it is called {\it nontrivial} if in addition it satisfies (3). A type II$_1$ factor with non-trivial central sequence is said to have {\it property Gamma}. Using the so-called $14\varepsilon$ argument, they showed  that the group von Neumann algebra $L(\mathbb{F}_2)$ of the free group $\mathbb{F}_2$ on two generators does not have property Gamma while $R$ does, whence $R\not\cong L(\mathbb{F}_2)$. Central sequences were then used to show the existence of uncountably many type II$_1$ (type II$_{\infty}$) factors \cite{McDuff,Sakai2}. Variants of the property Gamma, such as property L of Pukanszky \cite{Pukanszky} were also studied to provide examples of type III factors without non-trivial central sequences. 
On the other hand, the study of the quotient of a finite (A)W$^*$-algebra by its maximal ideals gave rise to the concept of tracial ultraproducts. The study of such quotient algebras was carried out by Wright. He showed \cite[Theorems 4.1 and 5.1]{Wright} that the quotient of a AW$^*$-algebra of type II with a trace by its maximal ideal is a AW$^*$-factor of type II, and quotient of finite AW$^*$-algebra of type I by its maximal ideals are generically AW$^*$-factors of type II$_1$. Sakai \cite[Theorem 7.1]{Sakai} showed that the quotient of a finite W$^*$-algebra $M$ by a maximal ideal $\mathcal{I}_{\omega}=\{x\in M; (x^*x)^{\natural}(\omega)=0\}$ is a finite W$^*$-factor. Here, $\natural\colon M\to \mathcal{Z}(M)\cong C(\Omega)$ is the center valued trace and $\omega$ is a point in the Gelfand spectrum $\Omega$ of the center $\mathcal{Z}(M)$. When $M=\bigoplus_{\mathbb{N}}M_n(\mathbb{C})$, we have $\Omega=\beta \mathbb{N}$ and $M/\mathcal{I}_{\omega}$ is what is now called the tracial ultraproduct of $\{M_n(\mathbb{C})\}_{n=1}^{\infty}$.   
More generally, the {\it tracial ultraproduct} $(M_n,\tau_n)^{\omega}$ of a sequence of finite von Neumann algebras with faithful tracial states $\{M_n,\tau_n\}_{n=1}^{\infty}$ along a free ultrafilter $\omega \in \beta \mathbb{N}\setminus \mathbb{N}$ is defined as the quotient algebra $(M_n,\tau_n)^{\omega}:=\ell^{\infty}(\mathbb{N},M_n)/\mathcal{I}_{\omega}(\mathbb{N},M_n)$, where $\ell^{\infty}(\mathbb{N},M_n)$ is the C$^*$-algebra of all bounded sequences of $\prod_{\mathbb{N}}M_n$, and $\mathcal{I}_{\omega}(\mathbb{N},M_n)$ is the ideal of $\ell^{\infty}(\mathbb{N},M_n)$ consisting of those sequences $(x_n)_n$ which satisfies $\tau_n(x_n^*x_n)\to 0$ along $\omega$. For the case of constant sequence $M_n=M, \tau_n=\tau$, $(M_n,\tau_n)^{\omega}$ is written as $M^{\omega}$ and called the {\it ultrapower} of $M$. Few years later after Sakai's work, McDuff \cite{McDuff2} revealed the importance of the tracial ultrapower and central sequences. Viewing $M$ as a subalgebra of $M^{\omega}$ by diagonal embedding, central sequences form a von Neumann subalgebra $M_{\omega}=M'\cap M^{\omega}$. Among other things, she proved that for a type II$_1$ factor $M$, $M_{\omega}$ is either abelian or type II$_1$, and the latter case occurs if and only if $M$ absorbs $R$ tensorially: $M\cong M\overline{\otimes}R$ (such a factor $M$ is now called {\it McDuff}).

The analysis on the tracial ultrapower was the crucial step in Connes' celebrated work on the classification of injective factors \cite{Connes2}. Since then, ultraproduct techniques has been an important ingredient on the classification of group actions on factors. Furthermore, the conjecture he casually posed in \cite{Connes2} that every II$_1$ factor with separable predual embeds into $R^{\omega}$ has drawn much attention today, and many interesting equivalent conditions are found \cite{Kirchberg}.  Nowadays tracial ultraproducts are also studied from model theoretical viewpoint (see e.g., \cite{GeHa},[FHS1-3]). For the case of type II$_1$ factors, a related construction using nonstandard analysis is given by \cite{OzawaHinokuma}.

The definition of the central sequence algebra $M_{\omega}$ is generalized for arbitrary von Neumann algebras by Connes \cite{Connes3}. It is defined as $M_{\omega}:=\mathcal{M}_{\omega}(\mathbb{N},M)/\mathcal{I}_{\omega}(\mathbb{N},M)$, where $\mathcal{M}_{\omega}(\mathbb{N},M)$ is the set of all $(x_n)_n\in \ell^{\infty}(\mathbb{N},M)$ satisfying $\|x_n\psi-\psi x_n\|\to 0$ along $\omega$ for all $\psi \in M_*$ (here in $\mathcal{I}_{\omega}(\mathbb{N},M)$, convergence is with respect to strong* topology). $M_{\omega}$ is called the {\it asymptotic centralizer} of $M$. On the other hand, the generalization of $M^{\omega}$ is more involved. If $M$ is not of finite type, then $\mathcal{I}_{\omega}(\mathbb{N},M)$ is not an ideal of $\ell^{\infty}(\mathbb{N},M)$. Therefore one has to modify the definition of $M^{\omega}$ for infinite type von Neumann algebras. The right definition of $M^{\omega}$ was given by Ocneanu \cite{Ocneanu} in order to generalize Connes' automorphism analysis approach for general injective von Neumann algebras. 
It is defined as $M^{\omega}:=\mathcal{M}^{\omega}(\mathbb{N},M)/\mathcal{I}_{\omega}(\mathbb{N},M)$, where $\mathcal{M}^{\omega}(\mathbb{N},M)$ is the two-sided normalizer of $\mathcal{I}_{\omega}(\mathbb{N},M)$. That is, $\mathcal{M}^{\omega}(\mathbb{N},M)$ consists of those $(x_n)_n\in \ell^{\infty}(\mathbb{N},M)$ which satisfies 
$(x_n)_n\mathcal{I}_{\omega}(\mathbb{N},M)\subset \mathcal{I}_{\omega}(\mathbb{N},M)$ and $\mathcal{I}_{\omega}(\mathbb{N},M)(x_n)_n\subset \mathcal{I}_{\omega}(\mathbb{N},M)$. 
We call $M^{\omega}$ the {\it Ocneanu ultrapower} of $M$. As same as tracial ultraproducts, any projection $p$ (resp. unitary $u$) in $M^{\omega}$ is represented by a sequence of projections $(p_n)_n$ (resp. unitaries $(u_n)_n$) of $M$. 
A decade before Ocneanu's definition of $M^{\omega}$, another generalization of $M'\cap M^{\omega}$ for a general factor $M$ with separable predual was proposed by Golodets \cite{Golodets}. It is defined as follows:  let $\varphi$ be a normal faithful state on $M$. Consider the GNS representation of $M$ associated with $\varphi$, so that 
$\varphi=\nai{\ \cdot\ \xi_{\varphi}}{\xi_{\varphi}}$ with a cyclic and separating vector $\xi_{\varphi}$ on a Hilbert space $H$. 
Consider the following (non-normal) state $\overline{\varphi}$ on $\ell^{\infty}=\ell^{\infty}(\mathbb{N},M)$:
\[\overline{\varphi}((x_n)_n):=\lim_{n\to \omega}\varphi(x_n),\ \ \ \ (x_n)_n\in \ell^{\infty}(\mathbb{N},M).\]   
Let $\pi_{\rm{Gol}}\colon \ell^{\infty}\to \mathbb{B}(H_{\rm{Gol}})$ be the GNS representation of $\overline{\varphi}$ with a cyclic vector $\overline{\xi}$ satisfying $\overline{\varphi}=\nai{\ \cdot\ \overline{\xi}}{\overline{\xi}}$. Let $e_{\omega}$ be the projection of $H_{\rm{Gol}}$ onto $\overline{\pi_{\rm{Gol}}(\ell^{\infty})'\overline{\xi}}$. Define 
\[\mathscr{R}:=e_{\omega}\pi_{\rm{Gol}}(\ell^{\infty})''e_{\omega}\subset \mathbb{B}(e_{\omega}H_{\rm{Gol}}).\]
Let $\overline{M}_d$ be the subspace of $\ell^{\infty}(\mathbb{N},M)$ consisting of constant sequences $(x,x,\cdots )_n,\ \ x\in M$. Then the {\it asymptotic algebra} $C_M^{\omega}$ of $M$ is defined by 
\[C_M^{\omega}:=\mathscr{R}\cap \pi_{\rm{Gol}}(\overline{M}_d)'\subset \mathbb{B}(e_{\omega}H_{\rm{Gol}}).\]
Moreover, $\varphi$ induces a normal faithful state $\tilde{\varphi}$ on $\mathscr{R}$, whence a state  $\dot{\varphi}:=\tilde{\varphi}|_{C_M^{\omega}}$ on $C_M^{\omega}$.  He then proved the following interesting property: let 
\[\overline{N}:=\{\overline{x}\in \ell^{\infty}(\mathbb{N},M); \pi_{\rm{Gol}}(\overline{x})e_{\omega}, \pi_{\rm{Gol}}(\overline{x}^*)e_{\omega}\in \mathscr{R}.\}\] 
Then \cite[Lemma 2.3.3]{Golodets} $\mathscr{R}=\pi_{\rm{Gol}}(\overline{N})e_{\omega}$, and we have \cite[Lemma 2.3.5]{Golodets}
\[\sigma_t^{\tilde{\varphi}}(\pi_{\rm{Gol}}(\overline{x})e_{\omega})=\pi_{\rm{Gol}}((\sigma_t^{\varphi}(x_n)))e_{\omega},\ \ \ \ \overline{x}=(x_n)_n\in \overline{N},\ \ t\in \mathbb{R}.\]
Based on the above, he proved that both (the isomorphism class of) $C_M^{\omega}$ and $\dot{\varphi}$  were independent of the choice of $\varphi$, and its point spectra characterized Araki's \cite{Araki2} property $L_{\lambda}'\colon$ $M\overline{\otimes}R_{\lambda}\cong M$ \cite[Theorem 2.5.2]{Golodets}.
Moreover, Golodets and Nessonov proved \cite[Lemma 1.1]{GoNe} that its centralizer $(C_M^{\omega})_{\dot{\varphi}}$ is isomorphic to $M_{\omega}$.
  It seems that these works have not been widely recognized, possibly because most of their works were written in Russian. 
It is not clear from his definition if $\mathscr{R}$ or $C_M^{\omega}$ is related to Ocneanu's constructions. We show in $\S$\ref{subsec: Golodets' asymptotic algebra} that Golodets' construction is equivalent to Ocneanu's one.

On the other hand, the development of non-commutative integration theory for von Neumann algebras suggests to seek for a notion of ``ultarproduct $\widetilde{M}^{\omega}$" of $M$ so that the Banach space ultraproduct $(L^p(M))_{\omega}$ of non-commutative $L^p$-space (in the sense of \cite{Haagerup4}) for $M$ is isometrically isomorphic to $L^p(\widetilde{M}^{\omega})\  (1\le p<\infty)$. In that viewpoint, it is not the Ocneanu ultraproduct $M^{\omega}$ that plays the role. For example, if one uses the Ocneanu ultraproduct, $\mathbb{B}(H)^{\omega}=\mathbb{B}(H)$ holds (e.g., \cite{MasudaTomatsu}), while $L^1(\mathbb{B}(H))_{\omega}=(\mathbb{B}(H)_*)_{\omega}$ is much larger than $\mathbb{B}(H)_*$ if $\dim(H)=\infty$. 
The right definition of the ultraproduct $\widetilde{M}^{\omega}$ in this context was given by Groh and Raynaud. More precisely, Groh \cite{Groh} showed that the ultraproduct of the predual $M_*$ of a von Neumann algebra $M$ can be regarded as the predual of some huge von Neumann algebra $\widetilde{M}^{\omega}$: consider the Banach space ultrapower $(M_*)_{\omega}$ (resp. $(M)_{\omega}$) of the predual $M_*$ (resp. $M$), and define a map $J_G\colon (M_*)_{\omega}\to ((M)_{\omega})^*$ by
\[\nai{x}{J_G(\hat{\psi})}:=\lim_{n\to \omega}\psi_n(x_n),\ \ \ \ x=(x_n)_{\omega}\in (M)_{\omega},\ \ \hat{\psi}=(\psi_n)_{\omega}\in (M_*)_{\omega}.\]
Then it holds that  \cite[Proposition 2.2 (b)]{Groh} $J_G$ is an isometric embedding and its range $J_G((M_*)_{\omega})$ is a translation-invariant subspace of $((M)_{\omega})^*$, whence there exists a central projection $z\in ((M)_{\omega})^{**}$ such that $J_G((M_*)_{\omega})=((M)_{\omega})^*z$. Therefore $J_G((M_*)_{\omega})$ can be regarded as the predual of the W$^*$-algebra $((M)_{\omega})^{**}z$.  
Then almost two decades later, a more handy construction was given by Raynaud \cite{Raynaud}: fix a representation $\pi$ of $M$ on a Hilbert space $H$ so that each $\varphi \in M_*^+$ is represented as a vector functional. Consider the Banach space ultrapower $(M)_{\omega}$ and regard it as a C$^*$-subalgebra of $\mathbb{B}(H)_{\omega}$. Define $J_R\colon \mathbb{B}(H)_{\omega}\to \mathbb{B}(H_{\omega})$ ($H_{\omega}$ is the ultrapower Hilbert space of $H$) by 
\[J_R(x)\xi:=(x_n\xi_n)_{\omega}, \ \ \ x=(x_n)_{\omega}\in \mathbb{B}(H)_{\omega},\ \ \xi=(\xi_n)_{\omega}\in H_{\omega}.\]
Then it holds \cite[Theorem 1.1]{Raynaud} that $(M_*)_{\omega}$ is isometrically isomorphic to the predual of the von Neumann algebra $\widetilde{M}^{\omega}$ generated by $J_R((M)_{\omega})$.    
We write $\widetilde{M}^{\omega}$ as $\prod^{\omega}M$ (where we choose the standard representation) and call it the {\it Groh-Raynaud ultrapower} of $M$.
Raynaud also showed that $\prod^{\omega}M$ has such nice behaviors as $L^p(M)_{\omega}\cong L^p(\prod^{\omega}M)$  completely isometrically, and $\prod^{\omega}M'=(\prod^{\omega}M)'$. The Groh-Raynaud ultrapower was effectively used in e.g., Junge's work on Fubini Theorem \cite{Junge}. On the other hand the Groh-Raynaud ultrapower has drawbacks too. In general, even if $M$ has separable predual, $\prod^{\omega}M$ is not even $\sigma$-finite (there is no faithful normal state), while $M^{\omega}$ is always $\sigma$-finite when $M$ is. Moreover, the center of $\prod^{\omega}M$ can be much larger than $M^{\omega}$: for example, Raynaud \cite{Raynaud} showed that $\prod^{\mathcal{U}}\mathbb{B}(H)\ (\text{dim}(H)=\infty)$, is not semifinite for a free ultrafilter $\mathcal{U}$ on a suitable index set $I$. It seems that there has been no attempts to consider the relationships among the Ocneanu ultraproducts, the Groh-Raynaud ultraproducts and Golodets' asymptotic algebras. 

We show that all these ultraproducts are closely related, and the study of one helps that of the other in an essential way. Using the connection, we show some interesting phenomena of the Ocneanu ultraproducts of type III factors which do not appear in the tracial case.\\

The outline of the paper is as follows.

In $\S$\ref{sec: General Backgrounds and Notations} we fix notations and recall basic facts about Tomita-Takesaki theory and Arveson-Connes' spectral theory for automorphism groups.  

In $\S$\ref{sec: Ultraproduct of von Neumann Algebras}, we define the Ocneanu and the Groh-Raynaud ultraproduct. Here we use a more general definitions than the original ones. 
Namely, we use a sequence of von Neumann algebras $\{M_n\}_{n=1}^{\infty}$ (and a sequence $\{\varphi_n\}_{n=1}^{\infty}$ of normal faithful states for the Ocneanu case) to construct $\prod^{\omega}M_n$ (and $(M_n,\varphi_n)^{\omega}$). We will see in later sections that this is not a formal generalization, but has actual implications. 
We show that each $(M_n,\varphi_n)^{\omega}$ corresponds to a corner of $\prod^{\omega}M_n$ (Theorem \ref{UW2.1.3}, Proposition \ref{prop: normalizer iff commutes with p}). 
Then generalizing the result of Raynaud, we prove in Theorem \ref{UW2.3.2} that the Groh-Raynaud ultraproduct preserves the standard form of von Neumann algebras. As part of the proof, we show (Lemma \ref{UW2.3.1}) that the condition $JxJ=x^*, x\in \mathcal{Z}(M)$, which is one of four conditions in the axioms of standard form, automatically follows from the other three conditions. This will be important for the proof of Theorem \ref{UW2.3.2}, for $\mathcal{Z}(\prod^{\omega}M_n)$ can be strictly larger than $\prod^{\omega}\mathcal{Z}(M_n)$\ (here, $\mathcal{Z}(M)$ is the center of $M$). In $\S$\ref{subsec: Golodets' asymptotic algebra}, we show that Golodets' asymptotic algebra is equivalent to Ocneanu's central sequence algebra. More precisely, we have $\mathscr{R}\cong M^{\omega}$ and under this *-isomorphism, $C_M^{\omega}$ is mapped to $M'\cap M^{\omega}$ (Theorem \ref{thm: identification of Golodets algebra}). Thus we have clarified the relationships among all the existing notions of ultraproducts of von Neumann algebras. 

In $\S$\ref{sec: Tomita-Takesaki Theory for Ultraproducts} we show how modular structure of the original algebra $M$ reflects the structure of the Ocneanu ultrapower $M^{\omega}$. We first show (Theorem\ref{thm: ultrapower of modular automorphism}) that the ultraproduct action of the sequence of modular automorphism groups $\{\sigma_t^{\varphi_n}\}_{n=1}^{\infty}$ on the Ocneanu ultraproduct $(M_n,\varphi_n)^{\omega}$ is equal to the modular automorphism group of $\varphi^{\omega}=(\varphi_n)^{\omega}\colon \ (\sigma_t^{\varphi_n})^{\omega}=\sigma_t^{\varphi^{\omega}}\ (t\in \mathbb{R})$. This in particular means that $t\mapsto (\sigma_t^{\varphi_n})^{\omega}$ is a continuous flow. This is somewhat surprising, because the ultrapower of continuous flows tend to be highly discontinuous on $M^{\omega}$ (see Masuda-Tomatsu's recent work on Rokhlin flows \cite{MasudaTomatsu}). The corresponding result in the case of constant sequence of algebras was obtained by Golodets \cite{Golodets} for his auxiliary algebra $\mathscr{R}$,  by Raynaud for the corner of $\prod^{\omega}M$ which corresponds to the Ocneanu ultrapower. 
Theorem \ref{thm: ultrapower of modular automorphism} will be the key tool for our analysis that follows. For instance, it follows that a sequence $(x_n)_n\in \ell^{\infty}(\mathbb{N},M_n)$ belongs to $\mathcal{M}^{\omega}(M_n,\varphi_n)$ if and only if $(x_n)_n$ can be approximated by sequences $(y_n)_n$ which belongs to a compact spectral subspace for $\sigma^{\varphi}$ (Proposition \ref{prop: characterization of the normalizer}).  
$\Delta_{\varphi^{\omega}}$ behaves like the ``ultrapower of $\Delta_{\varphi}$", although $(\Delta_{\varphi})^{\omega}$ is not a well-defined object. For example, it has such properties as $\sigma(\Delta_{\varphi^{\omega}})=\sigma(\Delta_{\varphi})$ and $\Delta_{\varphi^{\omega}}^{\frac{
1}{2}}(x_n\xi_{\varphi})_{\omega}=(\Delta_{\varphi}^{\frac{1}{2}}x_n\xi_{\xi_{\varphi}})_{\omega}$ for $(x_n)^{\omega}\in M^{\omega}$ (Corollary \ref{cor: some useful conseuences of main theorem}. Similar result to the latter is obtained by Golodets \cite{Golodets}). On the other hand, unlike the case of the ultrapower of bounded operators, the point spectra  $\sigma_p(\Delta_{\varphi^{\omega}})$ can be strictly smaller than $\sigma(\Delta_{\varphi})\setminus \{0\}$ (Proposition \ref{prop: point spectrum is smaller than the spectrum for UP state}).  In $\S$\ref{subsec: Strict Homogeneity of State Spaces}, 
we show that the Ocneanu ultrapower $M^{\omega}$ of $\sigma$-finite type III$_1$ factor $M$ has {\it strictly homogeneous state space}. That is, any two normal faithful states are unitarily equivalent (Theorem \ref{UW3.1.4}).  
We also show that no type III$_1$ factor with separable predual has strictly homogeneous state space (Proposition \ref{prop: no separable M with strictly homogeneous state space}), and if a $\sigma$-finite type III$_1$ factor $M$ has this property, the centralizer of any normal faithful state on $M$ is a factor of type II$_1$ (Proposition \ref{prop: centralizer of a factor with strictly homogeneous state space}). 
In $\S$\ref{subsec: Ultarpower of a Normal Faithful Semifinite Weight}, we extend Theorem \ref{thm: ultrapower of modular automorphism} to the ultrapower $\varphi^{\omega}$ of a normal faithful semifinite weight on a $\sigma$-finite von Neumann algebra (Lemma \ref{lem: modular automorphism of the ultrapower of weight}), and prove that if $\varphi$ is lacunary (that is, 1 is isolated in $\sigma(\Delta_{\varphi})$), then $(M^{\omega})_{\varphi^{\omega}}\cong (M_{\varphi})^{\omega}$ holds (Proposition \ref{prop: centralizer of lacunary weight}). In $\S$\ref{subsec: Golodets state}, we reinterpret the main results of Golodets work \cite{Golodets} mentioned above from our viewpoint. In particular, we introduce the {\it Golodets state} $\dot{\varphi}_{\omega}=\varphi^{\omega}|_{M'\cap M^{\omega}}$ and show that $\dot{\varphi}_{\omega}$ does not depend on the choice of $\varphi$ if $M$ is a factor (Proposition \ref{prop: phi_dot does not depend on phi}), and in that case $M$ has Araki's property $L_{\lambda}' (0<\lambda<1)$ if and only if $\lambda$ is the eigenvalue of $\Delta_{\dot{\varphi}_{\omega}}$ (Theorem \ref{thm: point spectrum and property L'}). Moreover, we show that the centralizer $(M'\cap M^{\omega})_{\dot{\varphi}_{\omega}}$ of the Golodets state $\dot{\varphi}_{\omega}$ is precisely Connes' asymptotic centralizer $M_{\omega}$ (Proposition \ref{prop: centralizer of ultrapower state}), which will play a key role in the next section.  
 
In $\S$\ref{sec: Ueda's Question}, we consider the subtle difference between $M_{\omega}$ and $M'\cap M^{\omega}$ for $\sigma$-finite type III factors. In general, $M_{\omega}\subset M'\cap M^{\omega}$ holds. However, while $M_{\omega}$ is a finite von Neumann algebra, $M'\cap M^{\omega}$ can be of type III (Example \ref{ex: Takesaki's example?}). However, based on the analysis of free product type III factors, Ueda \cite[$\S$5.2]{Ueda1} asked the following:
\begin{question}
Does $M_{\omega}=\mathbb{C}$ imply $M'\cap M^{\omega}=\mathbb{C}$?
\end{question}
We give (Theorem \ref{thm: solution to Ueda problem}) an affirmative answer to the question for separable predual case.   
Moreover, we show that for a $\sigma$-finite type III$_0$ factor $M$, $M_{\omega}=M'\cap M^{\omega}$ holds (Proposition \ref{prop: no difference between M_omega and M'cap Momega}).
   
In $\S$\ref{sec: Factoriality and Type of Ultraproducts}, we consider the following questions:
\begin{question} Let $(\varphi_n)_n$ be a sequence of normal faithful states on a $\sigma$-finite factor $M$.\\
(1) Are $M^{\omega}$ and $\prod^{\omega}M$ factor too?\ If so, what are their types?\\
(2) Does $(M,\varphi_n)^{\omega}$ depend on the choice of $(\varphi_n)_n$?\\
(3) Is $(M,\varphi_n)^{\omega}$ (semi-) finite if $M$ is (semi-) finite?\\
(4) Is $(M,\varphi_n)^{\omega}$ of type III if $M$ is of type III? 
\end{question}
For (1), if $M$ is of finite type, it is well-known that $M^{\omega}$ is also a finite type factor. Also, it is known that $M^{\omega}$ is a type I$_{\infty}$ (resp. type II$_{\infty}$) factor if so is $M$ (Proposition \ref{prop: Masuda-Tomatsu for semifinite factor}). However, the situation for the Groh-Raynaud ultrapower is different: we show that  $\prod^{\omega}R$ is not semifinite (and not a factor), where $R$ is the hyperfinite type II$_1$ factor (Theorem \ref{thm: prod^{omega}R is not a factor}). Type III case is more interesting: we show that if $M$ is a $\sigma$-finite type III$_{\lambda}\ (0<\lambda\le 1)$ factor, then both $M^{\omega}$ and $\prod^{\omega}M$ are type III$_{\lambda}$ factors (Theorem \ref{thm: UP of type III nonzero factor is a factor}). On the other hand, if $M$ is of type III$_0$, then $M^{\omega}$ is not a factor (Theorem \ref{thm: non-factoriality of UP(III_0)}). Moreover, $\prod^{\omega}M$ has a semifinite component and is not a factor (Remark \ref{rem: nonfactoriality of Raynaud UP of III_0}).  
As for (2), we show that if $M$ is of type III$_{\lambda}\ (0<\lambda \le 1)$, then $(M,\varphi_n)^{\omega}\cong M^{\omega}$ and therefore $(M,\varphi_n)^{\omega}$ does not depend on $(\varphi_n)_n$ (Theorem \ref{thm: UP of type III nonzero factor is a factor}). However, regarding (3)(4), there exists $(\varphi_n)_n$ such that $(R,\varphi_n)^{\omega}$ is not semifinite (Proposition \ref{prop: (R,varphi_n)^{omega} not semifinite}). Also, if $M$ is of type III$_0$, then there exists $(\varphi_n)_n$ such that $(M,\varphi_n)^{\omega}\cong (M_{\varphi_n})^{\omega}$ is of finite type (Theorem \ref{thm: UP of III_0 can be of finite type}). Finally, let us remark that our ultraproduct analysis has been used for the recent study of QWEP von Neumann algebras and Effros-Mar\'echal topology on the space of von Neumann algebras by the authors and Winsl\o w \cite{AndoHaagerupWinslow}. 

\section{General Backgrounds and Notations}\label{sec: General Backgrounds and Notations}
First we fix a notation and recall basics facts about ultraproducts. Throughout the paper, $\omega$ denotes the fixed free ultrafilter on $\mathbb{N}$ (in fact, many of the results and proofs in this paper are the same for free ultrafilters on any set). For a von Neumann algebra $M$ on a Hilbert space $H$, $\mathcal{Z}(M)$ denotes the center of $M$, and $S_\text{n}(M)$ (resp. $S_{\text{nf}}(M)$) denotes the space of normal (resp. normal faithful) states on $M$. As usual, we define two seminorms $\|\cdot \|_{\varphi},\ \|\cdot \|_{\varphi}^{\sharp}$ for $\varphi \in S_{\text{n}}(M)$ by
\[\|x\|_{\varphi}:=\varphi(x^*x)^{\frac{1}
{2}},\ \|x\|_{\varphi}^{\sharp}:=\varphi \left (x^*x+xx^*\right )^{\frac{1}{2}},\ x\in M.\]
If $M$ is $\sigma$-finite and $\varphi$ is faithful, $\|\cdot \|_{\varphi}$ (resp. $\|\cdot\|_{\varphi}^{\sharp}$) defines the strong (resp. strong*) topology on the unit ball of $M$. The support projection of a normal state $\varphi$ is written as $\text{supp}(\varphi)$. For a projection $p\in M$, $z_M(p)$ denotes the central support of $p$ in $M$. $\mathcal{U}(M)$ is the group of unitaries in $M$. vN$(H)$ denotes the space of all von Neumann algebras acting on $H$. Ball$(M)$ is the closed unit ball of $M$. For a self-adjoint operator $A$ on $H$, $\dom{A}$ is the domain of definition of $A$, $\sigma(A)$ (resp. $\sigma_p(A)$) denotes the spectra (resp. point spectra) of $A$. The range (resp. the domain) of $A$ is written as ${\rm{ran}}(A)$ (resp. ${\rm{dom}}(A)$). $G(A)=\{(\xi,A\xi);\xi \in \dom{A}\}$ is the graph of $A$. We denote the sequence of elements of a set like $\{a_n\}_{n=1}^{\infty}$. However, we also use the notation $(a_n)_n$ when we think of the sequence as an element in an algebra such as $\ell^{\infty}(\mathbb{N}, \mathbb{B}(H))$. For a unit vector $\xi\in H$, the corresponding vector state is denoted as $\omega_{\xi}$. 
\subsection{Ultraproduct of Banach Spaces}
Let $(E_n)_n$ be a sequence of Banach spaces, and let $\ell^{\infty}(\mathbb{N},E_n)$ be the Banach space of all sequences $(x_n)_n\in \prod_{n=1}^{\infty}E_n$ with $\sup_{n\ge 1}\|x_n\|<\infty$ with the norm $\|(a_n)_n\|=\sup_{n\ge 1}\|a_n\|,\ \ (a_n)_n\in \ell^{\infty}(\mathbb{N},E_n)$. 
The {\it Banach space ultraproduct} $(E_n)_{\omega}$ is defined as the quotient $\ell^{\infty}(\mathbb{N},E_n)_n/\mathcal{J}_{\omega}$, where $\mathcal{J}_{\omega}$ is the closed subspace of all $(x_n)_n\in \ell^{\infty}(\mathbb{N},E_n)$ which satisfies $\lim_{n\to \omega}\|x_n\|=0$. An element of $(E_n)_{\omega}$ represented by $(x_n)_n\in \ell^{\infty}(\mathbb{N},E)$ is written as $(x_n)_{\omega}$. One has $\|(x_n)_{\omega}\|=\lim_{n\to \omega}\|x_n\|,\ (x_n)_{\omega}\in (E_n)_{\omega}$. If $(H_n)_n$ is a sequence of Hilbert spaces, then $(H_n)_{\omega}$ is again a Hilbert space with the inner product given by 
\[\nai{(\xi_n)_{\omega}}{(\eta_n)_{\omega}}=\lim_{n\to \omega}\nai{\xi_n}{\eta_n},\ \ \ \ (\xi_n)_{\omega},\ (\eta_n)_{\omega}\in (H_n)_{\omega}.\]
For a sequence $(A_n)_n$ of C$^*$-algebras, $(A_n)_{\omega}$ is again a C$^*$-algebra when equipped with the pointwise multiplication and involution of sequences \cite[$\S$3.1]{Heinrich}. However, the Banach space ultraproduct of von Neumann algebras is not a von Neumann algebra in general (cf. Remark \ref{rem: UP of B(H) is not W*}). 
\subsection{Modular Theory}
We make a brief summary of modular theory needed for our purpose. We refer to vol. II of \cite{TakBook} for details. In particular we omit the modular theory for weights/Hilbert algebras, which will be used only for Proposition \ref{prop: centralizer of lacunary weight} and Theorem \ref{thm: non-factoriality of UP(III_0)}. Let $M$ be a $\sigma$-finite von Neumann algebra, and let $\varphi \in S_{\rm{nf}}(M)$. Using GNS representation $(M,\pi_{\varphi},H,\xi_{\varphi})$, $\varphi$ is represented as a vector state $\omega_{\xi_{\varphi}}$ and $\xi_{\varphi}\in H$ is a cyclic and separating vector for $M$ (we identify $x\in M$ with $\pi_{\varphi}(x)$). Then the following operator $S_{\varphi}^0$,
\[\dom{S_{\varphi}^0}:=M\xi_{\varphi},\ \ \ \ S_{\varphi}^0x\xi_{\varphi}:=x^*\xi_{\varphi},\  x\in M,\]
is a densely defined anti-linear operator on $H$. Since $S_{\varphi}^0$ is closable, we may consider the polar decomposition $S_{\varphi}=J_{\varphi}\Delta_{\varphi}^{\frac{1}{2}}$ of its closure. It can be shown that $J_{\varphi}$ is an anti-linear involution and $\Delta_{\varphi}$ is a positive, invertible self-adjoint operator on $H$. Furthermore, $J_{\varphi}\xi_{\varphi}=\Delta_{\varphi}\xi_{\varphi}=\xi_{\varphi}$ and $J_{\varphi}\Delta_{\varphi}J_{\varphi}=\Delta_{\varphi}^{-1}$ hold. 
$J_{\varphi}$ (resp. $\Delta_{\varphi}$) is called the {\it modular conjugation operator} (resp. {\it modular operator}) of $\varphi$. 
Tomita's fundamental Theorem \cite{Takesaki1} states that 
\[J_{\varphi}MJ_{\varphi}=M',\ \Delta_{\varphi}^{it}M\Delta_{\varphi}^{-it}=M \text{ \ \ for all }t\in \mathbb{R}.\]   
Therefore 
\[\sigma_t^{\varphi}(x):=\Delta_{\varphi}^{it}x\Delta_{\varphi}^{-it},\ \ \ \ x\in M,\ t\in \mathbb{R},\]
defines a one parameter automorphism group of $M$, called the {\it modular automorphism group} of $\varphi$.

Next we recall Arveson-Connes' spectral theory for automorphism groups. Since we apply the theory only to modular automorphism group, we present the case of one-parameter automorphism group only. In the sequel we identify the dual group $\widehat{\mathbb{R}}$ of the additive group $\mathbb{R}$ with itself. For $f\in L^1(\mathbb{R})$, we define the Fourier transform $\hat{f}$  by 
\[\hat{f}(\lambda):=\int_{\mathbb{R}}e^{it\lambda}f(t)dt,\ \ \ \ \ \lambda \in \widehat{\mathbb{R}}=\mathbb{R}.\]
We also define $\sigma_f^{\varphi}(x):=\int_{\mathbb{R}}f(t)\sigma_t^{\varphi}(x)dt\ (x\in M)$.
\begin{itemize}
\item[(1)] 
For $x\in M$, $\text{Sp}_{\sigma^{\varphi}}(x)$ is defined by
\[\left \{\lambda \in \widehat{\mathbb{R}};\ \hat{f}(\lambda)=0 \text{  for all } f\in L^1(\mathbb{R}) \text{  with  }\sigma^{\varphi}_f(x)=0\right \}.\] 
\item[(2)] The {\it Arveson spectrum} of $\sigma^{\varphi}$, denoted by $\text{Sp}(\sigma^{\varphi})$ is the set
\[\left \{\lambda \in \widehat{\mathbb{R}};\ \hat{f}(\lambda)=0 \text{  for all } f\in L^1(\mathbb{R}) \text{  with  }\sigma^{\varphi}_f=0\right \}.\]
It is shown that ${\rm{Sp}}(\sigma^{\varphi})=\log (\sigma(\Delta_{\varphi})\setminus \{0\})$. 
\item[(3)] For a subset $E$ of $\widehat{\mathbb{R}}$, the {\it spectral subspace} of $\sigma^{\varphi}$ corresponding to $E$ is given by 
\[M(\sigma^{\varphi},E):=\{x\in M;\ \text{Sp}_{\sigma^{\varphi}}(x)\subset E\}.\]
The fixed point subalgebra $M(\sigma^{\varphi},\{0\})$ is called the {\it centralizer} of $\varphi$, and is written as $M_{\varphi}$. It is known that $M_{\varphi}=\{x\in M;\varphi (xy)=\varphi(yx),\ y\in M\}$, and it is always a finite von Neumann algebra with a normal faithful trace $\varphi|_{M_{\varphi}}$. The spectral subspaces have the following properties:
\begin{list}{}{}
\item[(i)] $M(\sigma^{\varphi},E)^*=M(\sigma^{\varphi},-E)$.
\item[(ii)] $M(\sigma^{\varphi},E)M(\sigma^{\varphi},F)\subset M(\sigma^{\varphi},\overline{E+F})$.
\item[(iii)] $\lambda \in \text{Sp}(\sigma^{\varphi})$ if and only if $M(\sigma^{\varphi},E)\neq \{0\}$ for any closed neighborhood $E$ of $\lambda$.
\end{list}
\item[(4)] The {\it Connes spectrum} of $\sigma^\varphi$, denoted by $\Gamma (\sigma^{\varphi})$, is given by
\[\Gamma(\sigma^{\varphi})=\bigcap_{e\in \text{Proj}(M_{\varphi})}\text{Sp}(\sigma^{\varphi_e}).\]
Here, for $e\in \text{Proj}(M_{\varphi})$, $\sigma^{\varphi_e}$ is the restricted action of $\sigma^{\varphi}$ to the reduced algebra $M_e$, which coincides with the modular automorphism group of $\varphi|_{M_e}$. It holds that \[\Gamma(\sigma^{\varphi})=\bigcap_{0\neq e\in \text{Proj}(\mathcal{Z}(M_{\varphi}))}\text{Sp}(\sigma^{\varphi_e}),\]
 whence $\Gamma(\sigma^{\varphi})=\text{Sp}(\sigma^{\varphi})$ if $M_{\varphi}$ is a factor.
\item[(5)] Let $M$ be a $\sigma$-finite factor. The Connes' {\it S-invariant} is defined by 
\[S(M)=\bigcap_{\varphi \in S_{\rm{nf}}(M)}\sigma(\Delta_{\varphi}),\]
It is shown that $S(M)\setminus \{0\}$ is a closed multiplicative subgroup of $\mathbb{R}_+^*=(0,\infty)$, and $\Gamma(\sigma^{\varphi})=\log (S(M)\setminus \{0\})$. A $\sigma$-finite type III factor $M$ is called of
\begin{list}{}{}
\item[(i)] type III$_0$ if $S(M)=\{0,1\}$.
\item[(ii)] type III$_{\lambda}$ if $S(M)=\{\lambda^n; n\in \mathbb{Z}\}\cup \{0\}\ \ (0<\lambda<1)$.
\item[(iii)] type III$_1$ if $S(M)=[0,\infty)$.
\end{list}
For general factors, one needs to use normal faithful semifinite weights to define the $S$-invariant. However, the above classification of type III factors will not be affected by this change.  
\end{itemize}

\section{Ultraproduct of von Neumann Algebras}\label{sec: Ultraproduct of von Neumann Algebras}
\subsection{The Ocneanu Ultraproduct}\label{subsec: The Ocneanu Ultraproduct}
We first introduce the Ocneanu ultraproduct of a family of von Neumann algebras along $\omega$, with respect to a sequence of their states. This is a slight generalization of the construction of Ocneanu \cite[$\S$5]{Ocneanu} for a single algebra with a single state, and of the construction in \cite[$\S$5]{HW1} for tracial states; both generalize classical notions studied by Sakai \cite{Sakai} and McDuff \cite{McDuff}.

Specifically, let $(M_n)_n$ be a sequence of $\sigma$-finite von Neumann algebras, and let $\varphi_n$ be a normal faithful state on $M_n$ for each $n\in \mathbb{N}$. With a slight abuse of notation, put
\[\ell^{\infty}(\mathbb{N}, M_n):=\left \{(x_n)_n\in \prod_{n\in \mathbb{N}}M_n;\ \sup_{n\in \mathbb{N}}\|x_n\|<\infty\right \},\]
\[\mathcal{I}_{\omega}(M_n,\varphi_n):=\left \{(x_n)_n\in \ell^{\infty}(\mathbb{N},M_n); \ \lim_{n\to \omega}\|x_n\|_{\varphi_n}^{\sharp}=0\right \},\]
and also, with the abbreviated notation $\mathcal{I}_{\omega}$ for $\mathcal{I}_{\omega}(M_n,\varphi_n)$, let
\[\mathcal{M}^{\omega}(M_n,\varphi_n):=\left \{(x_n)_n\in \ell^{\infty}(\mathbb{N},M_n);\ (x_n)_n\mathcal{I}_{\omega}\subset \mathcal{I}_{\omega}, \text{ and }\mathcal{I}_{\omega}(x_n)_n\subset \mathcal{I}_{\omega}\right \}.\] 
It is then apparent that $\mathcal{M}^{\omega}(M_n,\varphi_n)$ is a C$^*$-algebra (with pointwise operations and supermum norm) in which $\mathcal{I}_{\omega}(M_n,\varphi_n)$ is a closed ideal. We then define 
\[(M_n,\varphi_n)^{\omega}:=\mathcal{M}^{\omega}(M_n,\varphi_n)/\mathcal{I}_{\omega}(M_n,\varphi_n)\]
(the quotient C$^*$-algebra). By the same proof as in \cite[$\S$5.1]{Ocneanu}, one gets:
\begin{proposition}\label{UW1.2.1}With the above notations, $(M_n,\varphi_n)^{\omega}$ is a {\rm{W}}$^*$-algebra. 
\end{proposition}
We remark that  Proposition \ref{prop: normalizer iff commutes with p} below gives an alternative proof of Proposition \ref{UW1.2.1}. 

We denote the image of $(x_n)_n\in \mathcal{M}^{\omega}(M_n,\varphi_n)$ in $(M_n,\varphi_n)^{\omega}$ as $(x_n)^{\omega}$. Then we have the following straightforward generalization of \cite[$S$5.1]{Ocneanu}:
\begin{proposition}\label{UW1.2.2} The following defines a normal faithful state $(\varphi_n)^{\omega}$ on $(M_n,\varphi_n)^{\omega}$:
\[(\varphi_n)^{\omega}((x_n)^{\omega}):=\lim_{n\to \omega}\varphi_n(x_n),\ \ \ (x_n)^{\omega}\in (M_n,\varphi_n)^{\omega}.\]
\end{proposition}
The special case considered by Ocneanu is the following:  all $M_n$ are equal to a fixed von Neumann algebra $M$, and all $\varphi_n$ are equal to a fixed normal faithful state $\varphi$ on $M$. In this case, we denote $(M_n,\varphi_n)^{\omega}$ by $M^{\omega}$, since the latter algebra does not depend on $\varphi$ (in fact, $\mathcal{I}_{\omega}(M_n,\varphi_n)$ determines the same set of bounded sequences for different state $\varphi$); we also denote $(\varphi_n)^{\omega}$ by $\varphi^{\omega}$. 

\subsection{The Groh-Raynaud Ultraproduct}\label{sec:The Groh-Raynaud Ultraproduct} 
In this section, we define Groh \cite{Groh}-Raynaud \cite{Raynaud}'s ultraproduct of a sequence of von Neumann algebras, which is in a rather direct way related to the ultraproduct of C$^*$-algebras and Hilbert spaces. 

Let $(H_n)_n$ be a sequence of Hilbert spaces, and let $H_{\omega}:=(H_n)_{\omega}$. Let $(\mathbb{B}(H_n))_{\omega}$ be the Banach space ultraproduct of $(\mathbb{B}(H_n))_n$. 
\begin{definition}
Define $\pi_{\omega}\colon (\mathbb{B}(H_n))_{\omega}\to \mathbb{B}(H_{\omega})$ by
\[\pi_{\omega}((a_n)_{\omega})(\xi_n)_{\omega}:=(a_n\xi_n)_{\omega},\ \ \ (a_n)_n\in \ell^{\infty}(\mathbb{N},\mathbb{B}(H_n)),\ \ (\xi_n)_{\omega}\in H_{\omega}.\]
\end{definition}
It is easy to check that $\pi_{\omega}((a_n)_{\omega})$ is a well-defined *-homomorphism, and since 
\[\|\pi_{\omega}((a_n)_{\omega})\|=\lim_{n\to \omega}\|a_n\|=\|(a_n)_{\omega}\|,\ \ \ \ (a_n)_{\omega}\in (\mathbb{B}(H_n))_{\omega},\]
$\pi_{\omega}$ is injective.

\begin{lemma}\label{UW2.1.1} $\pi_{\omega}((\mathbb{B}(H_n))_{\omega})$ is strongly dense in $\mathbb{B}(H_{\omega})$.
\end{lemma}
\begin{proof}
Let $\xi=(\xi_n)_{\omega}\in H_{\omega}$ and let $p_n\in \mathbb{B}(H_n)$ be the projection onto span$(\xi_n)\subset H_n$. Then $p:=\pi_{\omega}((p_n)_{\omega})$ is the projection onto span$(\xi)\subset H_{\omega}$, as for any $\eta=(\eta_n)_{\omega}\in H_{\omega}$ and $\zeta=(\zeta_n)_{\omega}\in H_{\omega}$, we have:
\eqa{
\nai{p\eta}{\zeta}&=\lim_{n\to \omega}\nai{(p\eta)_n}{\zeta_n}=\lim_{n\to \omega}\nai{p_n\eta_n}{\zeta_n}\\
&=\lim_{n\to \omega}\nai{\eta_n}{\xi_n}\nai{\xi_n}{\zeta_n}=\nai{\eta}{\xi}\nai{\xi}{\zeta}\\
&=\nai{\nai{\eta}{\xi}\xi}{\zeta}.
}
This shows that any rank one projection in $\mathbb{B}(H_{\omega})$ is contained in the subalgebra $\pi_{\omega}(\mathbb{B}(H_n)_{\omega})$. Therefore $\pi_{\omega}((\mathbb{B}(H_n))_{\omega})$ generates $\mathbb{B}(H_{\omega})$ as a von Neumann algebra. 
\end{proof}
\begin{definition}\label{UW2.1.2} Let $(M_n)_n$ be a sequence of W$^*$-algebras.  
\begin{list}{}{}
\item[(1)] Let $M_n\subset \mathbb{B}(H_n)$ be a fixed faithful representation of $M_n$ on a Hilbert space $H_n$. The {\it abstract ultraproduct} of the sequence $(M_n,H_n)_n$ is defined 
as the strong operator closure of $\pi_{\omega}((M_n)_{\omega})$ in $\mathbb{B}(H_{\omega})$, and is denoted as $\prod^{\omega}(M_n,H_n)$. 
\item[(2)] The {\it Groh-Raynaud ultraproduct} of $(M_n)_n$, denoted simply as $\prod^{\omega}M_n$ is defined as $\prod^{\omega}M_n:=\prod^{\omega}(M_n,H_n)$, where we choose the standard representation of $M_n$.  
\end{list}
\end{definition}
From Lemma \ref{UW2.1.1}, it follows that 
\[\prod^{\omega}(\mathbb{B}(H_n),H_n)=\mathbb{B}(H_{\omega}).\]
However, note that the Groh-Raynaud ultraproduct $\prod^{\omega}\mathbb{B}(H_n)$ is not equal to $\mathbb{B}(H_{\omega})$.
\begin{remark}\label{rem: UP of B(H) is not W*}
Let $H$ be a separable infinite-dimensional Hilbert space. We remark that although  $\pi_{\omega} (\mathbb{B}(H)_{\omega})$ is srongly dense in $\mathbb{B}(H_{\omega})$, $\pi_{\omega}$ is not surjective.\\
To see this, using the weak compactness of the unit ball of $H_{\omega}$, define $P\in \mathbb{B}(H_{\omega})$ by 
\[P(\xi_n)_{\omega}:=(\xi)_{\omega},\ \ \ \xi:=\text{weak-}\lim_{n\to \omega}\xi_n,\ \ \ \ (\xi_n)_{\omega}\in H_{\omega}.\]
$P$ is well-defined and is bounded, because for each $n\in \mathbb{N}$ we have
\eqa{
|\nai{P(\xi_n)_{\omega}}{P(\xi_n)_{\omega}}|&=|\nai{\xi}{\xi}|=\lim_{k\to \omega}\lim_{n\to \omega}|\nai{\xi_k}{\xi_n}|\\
&\le \|(\xi_n)_{\omega}\|^2.}
Therefore $P\in \mathbb{B}(H_{\omega})$. It is easy to see that $P^2=P$ holds. We show that $P\notin \pi (\mathbb{B}(H)_{\omega})$. Assume by contradiction that there is a bounded sequence $(p_n)_n\in \ell^{\infty}(\mathbb{N},\mathbb{B}(H))$ such that $\pi_{\omega}((p_n)_{\omega})=P$ holds. This means that if a bounded sequence $(\xi_n)_n$ in $H$ converges weakly to $\xi\in H$, then $\|p_n\xi_n-\xi\|\to 0 \ (n\to \omega)$. In particular, $p_n\to 1\ (n\to \omega)$ strongly.  
\\ \\
\textbf{Step 1}. We first show that $P=P^*$, hence $P$ is a projection (onto the closed subspace $H$ of $H_{\omega}$). 
Let $(\xi_n)_n,\ (\eta_n)_n\in \ell^{\infty}(\mathbb{N},H)$ and let $\xi=\text{weak-}\lim_{n\to \omega}\xi_n, \eta=\text{weak-}\lim_{n\to \omega}\eta_n$. We have
\eqa{
\nai{P(\xi_n)_{\omega}}{(\eta_n)_{\omega}}&=\lim_{n\to \omega}\nai{\xi}{\eta_n}=\nai{\xi}{\eta}\\
&=\lim_{n\to \omega}\nai{\xi_n}{\eta}=\nai{(\xi_n)_{\omega}}{P(\eta_n)_{\omega}},
}
whence $P=P^*$ holds.\\ \\
\textbf{Step 2}. 
There exists a sequence $(\eta_n)_n$ of unit vectors in $H$ such that $\{n\in \mathbb{N};\ \eta_n\in \text{ran}(p_n)\}\in \omega$ and $\eta_n\to 0\ (n\to \omega)$ weakly.\\
To see this, fix an orthonormal base $(e_n)_n$ of $H$. Since $\lim_{n\to \omega}\|p_ne_1-e_1\|=0$, we have
\[\left \{n\in \mathbb{N};\ \|p_ne_1-e_1\|<\frac{1}{3}\right \}\subset \left \{n\in \mathbb{N};\ \|p_ne_1\|\ge \frac{1}{2}\right \}=:I_1\in \omega.\]
Define $l_n:=\max \{1\le j\le n;\ \|p_ne_j\|\ge \frac{1}{2}\}$ for each $n\in I_1$. 
We then define $(\eta_n)_n$ by 
\[\eta_n:=\begin{cases}
\dfrac{p_ne_{l_n}}{\|p_ne_{l_n}\|} & (n\in I_1)\\
\ \ \ e_1 & (n\notin I_1).
\end{cases}\]
Next, suppose $i\ge 1$ and $\varepsilon>0$ are given. Let 
\[I_2:=\{n\in \mathbb{N};\ \|p_ne_i-e_i\|<\varepsilon/2\}\in \omega.\]
Since 
$\lim_{n\to \omega}p_n=1$ strongly,  the set $I_3$ defined by 
\[I_3:=\{n\in \mathbb{N};\ l_n>i\}\]
 belongs to $\omega$ as well. Then for each $n\in I:=I_1\cap I_2\cap I_3\in \omega$, we have 
\eqa{
|\nai{\eta_n}{e_i}|&\le \frac{1}{\|p_ne_{l_n}\|}\{|\nai{e_{l_n}}{p_ne_i-e_i}|+|\nai{e_{l_n}}{e_i}|\}\\
&\le 2\|p_ne_i-e_i\|<\varepsilon. 
}
This shows that $\lim_{n\to \omega}\nai{\eta_n}{e_i}=0$. Since $i\in \mathbb{N}$ is arbitrary, we obtain the claim.\\ \\
\textbf{Step 3}. We get a contradiction.\\
Since $P=\pi((p_n)_{\omega})$ is a projection, we may choose $p_n$ to be a projection for all $n\in \mathbb{N}$ (see e.g., \cite[Proposition 2.1 (4)]{GeHa}).  
By Step 2, there  exists a sequence of unit vectors $(\eta_n)_n$ such that $J:=\{n\in \mathbb{N}; \eta_n\in \text{ran}(p_n)\}\in \omega$ and $\text{weak-}\lim_{n\to \omega}\eta_n=0$. Then, by definition, we have $P(\eta_n)_{\omega}=0$. However, for $n\in J$, we have $\|p_n\eta_n\|=\|\eta_n\|=1$, hence $p_n\eta_n$ does not tend to 0 along $\omega$. This is a contradiction. Hence $P$ is not in the range of $\pi$.
\end{remark} 
\subsection{Relationship between $(M_n,\varphi_n)^{\omega}$ and $\prod^{\omega}M_n$}\label{subsec: Relation ship between Ocneanu and Rayanud}
As we have seen, there are two notions of ultraproducts for von Neumann algebras. The following theorem explains the relation between the Ocneanu ultraproduct and the Groh-Raynaud ultraproduct:
\begin{theorem}\label{UW2.1.3} Let $(M_n)_n$ be a sequence of $\sigma$-finite von Neumann algebras and let a normal faithful state $\varphi_n$ on $M_n$ be given for each $n\in \mathbb{N}$. 
Assume that each $M_n$ acts standardly on $H_n=L^2(M_n,\varphi_n)$, so that $\prod^{\omega}M_n\subset \mathbb{B}((H_n)_{\omega})$. Also let $M^{\omega}=(M_n,\varphi_n)^{\omega}, \varphi^{\omega}=(\varphi_n)^{\omega}$, and define $w\colon L^2(M^{\omega},\varphi^{\omega})\to (H_n)_{\omega}$ by 
\[w(x_n)^{\omega}\xi_{\varphi^{\omega}}:=(x_n\xi_{\varphi_n})_{\omega},\ \ \ (x_n)^{\omega}\in M^{\omega}.\]
Then $w$ is an isometry, and $w^*(\prod^{\omega}M_n)w=M^{\omega}$.
\end{theorem}
To ease notation, let $N=\prod^{\omega}M_n$ in the sequel. That $w$ is indeed an isometry is seen by direct calculation. To show the identity $w^*Nw=M^{\omega}$, we need to study the following subsets of $\prod_{n\in \mathbb{N}}M_n$ (for which we use the indicated short notations):
\eqa{
\ell^{\infty}&:=\ell^{\infty}(\mathbb{N},M_n)\\
\mathcal{L}_{\omega}&:=\left \{(x_n)_n\in \ell^{\infty};\ \lim_{n\to \omega}\varphi_n(x_n^*x_n)=0\right \},\ \mathcal{L}_{\omega}^*:=\{(x_n^*)_n;\ (x_n)_n\in \mathcal{L}_{\omega}\}\\
\mathcal{M}^{\omega}&:=\mathcal{M}^{\omega}(M_n,\varphi_n),\ \ \mathcal{I}_{\omega}:=\mathcal{I}_{\omega}(M_n,\varphi_n)\\
}
\begin{lemma}\label{lem: easy lemma}
$\mathcal{L}_{\omega}$ is a closed left ideal of $\ell^{\infty}$, and $\mathcal{I}_{\omega}=\mathcal{L}_{\omega}\cap \mathcal{L}_{\omega}^*$.
\end{lemma}
\begin{proof}
It is easy to see that $\mathcal{L}_{\omega}$ is a closed subspace of $\ell^{\infty}$. 
Let $(x_n)_n\in \mathcal{L}_{\omega}$ and $(a_n)_n\in \ell^{\infty}$. Then we have $\varphi(x_n^*a_n^*a_nx_n)\le \|a_n\|^2\varphi(x_n^*x_n)\stackrel{n\to \omega}{\to}0$. Therefore $(a_nx_n)_n\in \mathcal{L}_{\omega}$ and $\mathcal{L}_{\omega}$ is a closed left ideal of $\ell^{\infty}$. The last claim is obvious. 
\end{proof}
Before going further, we prove a result about hereditary C$^*$-subalgebras. Recall the following
\begin{theorem}{\rm{\cite[Theorem 3.2.1]{Murphy}}}\label{thm: hereditary subalgebras}
Let $A$ be a {\rm{C}}$^*$-algebra. If $L$ is a closed left ideal in $A$, then $L\cap L^*$ is a hereditary {\rm{C}}$^*$-subalgebra of $A$. The map $L\mapsto B(L):=L\cap L^*$ is a bijection from the set of closed left ideals of $A$ onto the set of hereditary {\rm{C}}$^*$-subalgebras of $A$. The inverse of the map is given by $B\mapsto L(B)$, where $B$ is a hereditary {\rm{C}}$^*$-subalgebra of $A$ and 
\[L(B):=\{a\in A;\ a^*a\in B\}.\]
\end{theorem}

\begin{lemma}\label{lem: Mcap (L+L^*)=B}
Let $A$ be a {\rm{C}}$^*$-algebra, and let $L$ be a closed left ideal of $A$. Let $B=L\cap L^*$ be the corresponding hereditary {\rm{C}}$^*$-subalgebra of $A$, and let $M$ be the two-sided multiplier of $B$:
\[M:=\{a\in A; aB\subset B,\ Ba\subset B\}.\]
Then we have 
\begin{list}{}{}
\item[{\rm{(1)}}] $LM\subset L, ML^*\subset L^*$.
\item[{\rm{(2)}}]
$M\cap (L+L^*)=B$.
\end{list}
\end{lemma}
\begin{proof} It is easy to see that $M$ is a C$^*$-subalgebra of $A$.\\
(1) Let $a\in L$ and $x\in M$. Then $a^*a\in L^*L\subset L\cap L^*=B$. Therefore $x^*a^*ax\in B$. By Theorem \ref{thm: hereditary subalgebras}, $L=L(B)$ implies that $ax\in L$. Therefore $LM\subset L$. Taking the adjoint, we obtain $ML^*\subset L^*$.\\   
(2) We show the claim in two steps. \\ \\
\textbf{Step 1}. $M\cap L=M\cap L^*=B$.\\
Since $M$ and $B$ are self-adjoint, it suffices to show that $M\cap L=B$. Since $B$ is a C$^*$-algebra, it is clear that $B\subset M\cap L$. Conversely, suppose $x\in M\cap L$. Then $x^*\in L^*$ holds, and hence $x^*x\in L^*L\subset L\cap L^*=B$. On the other hand, as $x\in M$, we have $x(x^*x)x^*\in B$, which implies that $xx^*=\{x(x^*x)x^*\}^{\frac{1}{2}}\in B$. Then by Theorem \ref{thm: hereditary subalgebras}, again, $x^*\in L=L(B)\Leftrightarrow x\in L^*$ holds. Hence $x\in L\cap L^*=B$.\\ \\
\textbf{Step 2}. $M\cap (L+L^*)=B$.\\
By Step 1, it suffices to show that $M\cap (L+L^*)=(M\cap L)+(M\cap L^*)$. It is clear that $(M\cap L)+(M\cap L^*)\subset M\cap (L+L^*)$. Conversely, suppose $x\in M\cap (L+L^*)$. Then there is $y\in L, z\in L^*$ such that $x=y+z$ holds. We show that $y,z\in M$. Let $b\in B$. Then $yb\in L$. Furthermore, $yb=xb-zb$ is in $L^*$, because $b\in B,x\in M$ implies that $xb\in B=L\cap L^*$ and $z\in L^*$. Therefore $yB\subset B$. On the other hand, $by\in L\cap L^*$ (since $y\in L$, $b\in L^*$) holds. Therefore $By\subset B$. This shows that $y\in M$. Similarly, we have $z\in M$. Therefore $M\cap (L+L^*)=(M\cap L)+(M\cap L^*)$ holds. This finishes the proof. 
\end{proof}

\begin{corollary}\label{cor: M^{omega} cap (L+L^*)=I} We have 
\begin{list}{}{}
\item[{\rm{(1)}}] $\mathcal{L}_{\omega}\mathcal{M}^{\omega}\subset \mathcal{L}_{\omega}$, $\mathcal{M}^{\omega}\mathcal{L}_{\omega}^*\subset \mathcal{L}_{\omega}^*$.
\item[{\rm{(2)}}] $\mathcal{M}^{\omega}\cap (\mathcal{L}_{\omega}+ \mathcal{L}_{\omega}^*)=\mathcal{I}_{\omega}$. 
\end{list}
\end{corollary}

\begin{proof}
By Lemma \ref{lem: easy lemma}, we can apply Lemma \ref{lem: Mcap (L+L^*)=B} to $A=\ell^{\infty}, L=\mathcal{L}_{\omega}, M=\mathcal{M}^{\omega}, B=\mathcal{I}_{\omega}$. 
\end{proof} 

Now, let $\xi_{\omega}:=(\xi_{\varphi_n})_{\omega}\in (H_n)_{\omega}$ and let
\[\varphi_{\omega}(x):=\nai{x\xi_{\omega}}{\xi_{\omega}},\ \ \ x\in N.\]
Then $\varphi_{\omega}$ is a normal state on $N$.
\begin{definition}\label{def: support projection p}
We denote by $p$ the support projection of $\varphi_{\omega}$, which is the projection onto $\overline{(\prod^{\omega}M_n)'\xi_{\omega}}$. 
\end{definition}
For simplicity, we shall mostly write $\pi_{\omega}(x)$ as just $x$ in the following (for $x\in (M_n)_{\omega}$).\\ \\
Next lemma is similar to \cite[Lemma 2.2.2]{Golodets} and \cite[Proposition 2.1]{RaynaudXu}.
\begin{lemma}\label{UW2.1.5} For all $x\in N$, there is $(x_n)_n\in \ell^{\infty}$ such that
\begin{itemize}
\item[{\rm{(1)}}] $x\xi_{\omega}=(x_n)_{\omega}\xi_{\omega}$ and $x^*\xi=(x_n^*)_{\omega}\xi_{\omega}$.
\item[{\rm{(2)}}] $x-p^{\perp}xp^{\perp}=(x_n)_{\omega}-p^{\perp}(x_n)_{\omega}p^{\perp}$.
\end{itemize}
\end{lemma}
\begin{proof}
Consider the following subset of $(H_n)_{\omega}\oplus (H_n)_{\omega}$:
\[E:=\left \{((x_n)_{\omega}\xi_{\omega}, (x_n^*)_{\omega}\xi_{\omega});\ (x_n)_n\in \ell^{\infty},\ \sup_{n\ge 1}\|x_n\|\le 1\right \}.\]
We claim that $E$ is a closed subset of $(H_n)_{\omega}\oplus (H_n)_{\omega}$. Indeed, let $(\eta,\zeta)$ be in the closure of $E$, and choose a sequence $\{(x_n^k)_n\}_{k=1}^{\infty}\subset \ell^{\infty}$ such that $\sup_{n\ge 1}\|x_n^k\|\le 1$ for all $k\in \mathbb{N}$, and such that
\[\|(x_n^k)_{\omega}\xi_{\omega}-\eta \|\le 2^{-k-1} \text{ and } \|(x_n^k)^*_{\omega}\xi_{\omega}-\zeta\|\le 2^{-k-1},\]
for all $k\in \mathbb{N}$. Then in particular we have, for all $k\in \mathbb{N}$:
\[\|(x_n^{k+1})_{\omega}\xi_{\omega}-(x_n^k)_{\omega}\xi_{\omega}\|\le 2^{-k} \text{ and } \|(x_n^{k+1})^*_{\omega}\xi_{\omega}-(x_n^k)^*_{\omega}\xi_{\omega}\|\le 2^{-k},\]
so that if we define
\[F_k:=\{n\in \mathbb{N};\ \|x_n^{k+1}\xi_{\varphi_n}-x_n^k\xi_{\varphi_n}\|\le 2^{-k} \text{ and } \|(x_n^{k+1})^*\xi_{\varphi_n}-(x_n^k)^*\xi_{\varphi_n}\|\le 2^{-k}\},\]
then we have $F_k\in \omega$ for all $k\in \mathbb{N}$. Hence with
\[G_k:=\{k,k+1,\cdots \}\cap \bigcap_{j=1}^kF_j,\]
we have $G_k\in \omega$ for all $k\in \mathbb{N}$ because $\omega$ is free, and $(G_k)_k$ is a decreasing sequence with empty intersection. In particular, 
\[\mathbb{N}=(\mathbb{N}\setminus G_1)\sqcup \bigsqcup_{j=1}^{\infty} (G_j\setminus G_{j+1})\]
(disjoint union). Now, define a sequence $(x_n)_n\in \ell^{\infty}$ by
\[
x_n:=\begin{cases}
x_n^1 & (n\in \mathbb{N}\setminus G_1),\\
x_n^j & (n\in G_j\setminus G_{j+1}).
\end{cases}
\]
Then $\sup_{n\ge 1}\|x_n\|\le 1$. Fix $k\in \mathbb{N}$. If $n\in G_k$, then as $G_k=\bigsqcup_{j=k}^{\infty}(G_j\setminus G_{j+1})$, we may choose $j\ge k$ such that $n\in G_j\setminus G_{j+1}$, so that $x_n=x_n^j$ and as $n\in G_j\subset G_m\subset F_m$ for every $m\le j$, we therefore have 
\eqa{
\|x_n\xi_{\varphi_n}-x_n^k\xi_{\varphi_n}\|
&=\|x_n^j\xi_{\varphi_n}-x_n^k\xi_{\varphi_n}\|\\
&\le \sum_{m=k}^{j-1}\|x_n^{m+1}\xi_{\varphi_n}-x_n^m\xi_{\varphi_n}\|\\
&\le \sum_{m=k}^{j-1}2^{-m}\le 2^{-k+1},
}
for every $n\in G_k$. It follows that $\|(x_n)_{\omega}\xi_{\omega}-(x_n^k)_{\omega}\xi_{\omega}\|\le 2^{-k+1}$, so that 
\[\|(x_n)_{\omega}\xi_{\omega}-\eta\|\le \|(x_n)_{\omega}\xi_{\omega}-(x_n^k)_{\omega}\xi\|+\|(x_n^k)_{\omega}\xi_{\omega}-\eta\|\le 2^{-k+1}+2^{-k+1}.\]
As $k\in \mathbb{N}$ may be chosen to be arbitrarily big, we conclude that $(x_n)_{\omega}\xi_{\omega}=\eta$. The proof that $(x_n^*)_{\omega}\xi_{\omega}=\zeta$ is similar. Hence $E$ is closed, as claimed.

We are now ready to prove (1). It clearly suffices to consider $x\in N$ with $\|x\|\le 1$. By the definition of the Groh-Raynaud ultraproduct, and Kaplansky's Theorem, we may choose a net $\{(x_n^{\alpha})_n\}_{\alpha}\subset \ell^{\infty}$ such that $\sup_{n\ge 1}\|x_n^{\alpha}\|\le 1$ for every $\alpha$ and such that $\lim_{\alpha}(x_n^{\alpha})_{\omega}=x$ in the strong$^*$-topology on $N$. But then $(x\xi_{\omega},x^*\xi_{\omega})$ is in the closure of $E$, hence in $E$ by the previous paragraph, and (1) follows. 

Finally, (2) follows from (1): with $x$ and $(x_n)_n$ from there, we have for all $y\in N'$:
\[xy\xi_{\omega}=yx\xi_{\omega}=y(x_n)_{\omega}\xi_{\omega}=(x_n)_{\omega}y\xi_{\omega},\]
so $xp=(x_n)_{\omega}p$, and similarly $x^*p=(x_n^*)_{\omega}p$. Conjugating the latter identity, $px=p(x_n)_{\omega}$ holds. Now (2) follows easily.
\end{proof}

\begin{proposition}\label{UW2.1.6}
There is a vector space isomorphism 
\[\rho\colon \ell^{\infty}/\mathcal{I}_{\omega}\to pNp\oplus pNp^{\perp}\oplus p^{\perp}Np\]
such that $\rho^{-1}(pNp)=\mathcal{M}^{\omega}/\mathcal{I}_{\omega}$,\ $\rho^{-1}(pNp^{\perp})=\mathcal{L}_{\omega}/\mathcal{I}_{\omega}$,\ and $\rho^{-1}(p^{\perp}Np)=\mathcal{L}^*_{\omega}/\mathcal{I}_{\omega}$. In particular, we have
\[\ell^{\infty}=\mathcal{M}^{\omega}+\mathcal{L}_{\omega}+\mathcal{L}^*_{\omega}.\]
\end{proposition}
\begin{proof}
Observe first that for $(x_n)_n\in \ell^{\infty}$, we have
\[(x_n)_n\in \mathcal{L}_{\omega}\Leftrightarrow \|(x_n)_{\omega}\xi_{\omega}\|=0\Leftrightarrow (x_n)_{\omega}\in Np^{\perp}.\]
Hence $(x_n)_n\in \mathcal{L}_{\omega}^*\Leftrightarrow (x_n)_{\omega}\in p^{\perp}N$, so by Lemma \ref{lem: easy lemma}, 
\[(x_n)_n\in \mathcal{I}_{\omega}=\mathcal{L}_{\omega}\cap \mathcal{L}_{\omega}^*\Leftrightarrow (x_n)_{\omega}\in p^{\perp}Np^{\perp}.\]
Hence by letting
\[\rho((x_n)_n/\mathcal{I}_{\omega}):=(x_n)_{\omega}-p^{\perp}(x_n)_{\omega}p^{\perp},\]
we obtain a well-defined injective linear map from $\ell^{\infty}/\mathcal{I}_{\omega}$ into $pNp\oplus pNp^{\perp}\oplus p^{\perp}Np$, and it is in fact surjective by Lemma \ref{UW2.1.5}.

By definition and the above, we have for all $(x_n)_n\in \ell^{\infty}$:
\eqa{
\rho((x_n)_n/\mathcal{I}_{\omega})\in pNp^{\perp}&\Leftrightarrow (x_n)_{\omega}-p^{\perp}(x_n)_{\omega}p^{\perp}=p(x_n)_{\omega}p^{\perp}\\
&\Leftrightarrow (x_n)_{\omega}\in Np^{\perp}\\
&\Leftrightarrow (x_n)_n\in \mathcal{L}_{\omega},
}
and from this, 
\[\rho((x_n)_n/\mathcal{I}_{\omega})\in p^{\perp}Np\Leftrightarrow (x_n^*)_n\in \mathcal{L}_{\omega}\Leftrightarrow (x_n)_n\in \mathcal{L}_{\omega}^*.\]
Therefore we have
\begin{equation}
\rho^{-1}(pNp^{\perp})=\mathcal{L}_{\omega}/\mathcal{I}_{\omega},\ \ \ \ \ \rho^{-1}(p^{\perp}Np)=\mathcal{L}_{\omega}^*/\mathcal{I}_{\omega}.\label{eq: rho^{-1}(pNp^perp)}
\end{equation}
Finally, if $\rho((x_n)_n/\mathcal{I}_{\omega})\in pNp$, and $(y_n)_n\in \mathcal{I}_{\omega}$, we have $(y_n)_{\omega}\in p^{\perp}Np^{\perp}$, and so 
\eqa{
\rho((x_ny_n)_n/\mathcal{I}_{\omega})&=(x_n)_{\omega}(y_n)_{\omega}-p^{\perp}(x_n)_{\omega}(y_n)_{\omega}p^{\perp}\\
&=((x_n)_{\omega}-p^{\perp}(x_n)_{\omega}p^{\perp})(y_n)_{\omega}\\
&=\rho((x_n)_n/\mathcal{I}_{\omega})(y_n)_{\omega}=0,
}
and therefore $(x_ny_n)_n\in \mathcal{I}_{\omega}$. Similarly $(y_nx_n)_n\in \mathcal{I}_{\omega}$, so $(x_n)_n\in \mathcal{M}^{\omega}$. This shows $\rho^{-1}(pNp)\subset \mathcal{M}^{\omega}/\mathcal{I}_{\omega}$.
On the other hand, by Eq. (\ref{eq: rho^{-1}(pNp^perp)}) and by Corollary \ref{cor: M^{omega} cap (L+L^*)=I} (2), we have
\eqa{
\mathcal{M}^{\omega}/\mathcal{I}_{\omega}\cap \rho^{-1}(pNp^{\perp}\oplus p^{\perp}Np)&=[\mathcal{M}^{\omega}\cap (\mathcal{L}_{\omega}+\mathcal{L}_{\omega}^*)]/\mathcal{I}_{\omega}\\
&=\{0\},
}
whence we have $\rho^{-1}(pNp)=\mathcal{M}^{\omega}/\mathcal{I}_{\omega}$. 

In particular,  $\ell^{\infty}/\mathcal{I}_{\omega}=\mathcal{M}^{\omega}/\mathcal{I}_{\omega}+\mathcal{L}_{\omega}/\mathcal{I}_{\omega}+\mathcal{L}_{\omega}^*/\mathcal{I}_{\omega}$, and the last claim is then obvious.
\end{proof}
\begin{proposition}\label{prop: normalizer iff commutes with p}
Let $(x_n)_n\in \ell^{\infty}$. Then $(x_n)_n\in \mathcal{M}^{\omega}$ if and only if $p(x_n)_{\omega}=(x_n)_{\omega}p$ holds. Moreover, $\rho|_{M^{\omega}}\colon M^{\omega}\to pNp, (x_n)_n/\mathcal{I}_{\omega}\mapsto (x_n)_{\omega}p$ is a *-isomorphism. Therefore, the Ocneanu ultraproduct is isomorphic to a reduction of the Groh-Raynaud ultraproduct by the support projection $p$ of $\varphi_{\omega}\in N_*$.  
\end{proposition}
\begin{proof}
By Proposition \ref{UW2.1.6}, $(x_n)_n/\mathcal{I}_{\omega}\in \mathcal{M}^{\omega}/\mathcal{I}_{\omega}$ holds if and only if  $\rho((x_n)_n/\mathcal{I}_{\omega})=p(x_n)_{\omega}p+p(x_n)_{\omega}p^{\perp}+p^{\perp}(x_n)_{\omega}p\in pNp$, if and only if $p(x_n)_{\omega}p^{\perp}=p^{\perp}(x_n)_{\omega}p=0$. The last condition is equivalent to $(x_n)_{\omega}p=p(x_n)_{\omega}$. Since $\rho|_{M^{\omega}}\colon M^{\omega}\to pNp$ is linear and bijective, to prove the last assertion it is enough to show that $\rho|_{M^{\omega}}$ is a *-homomorphism. Let $(a_n)_n,(b_n)_n\in \mathcal{M}^{\omega}$. Then as $(a_n)_{\omega}, (b_n)_{\omega}$ commute with $p$, we have
\eqa{
\rho((a_nb_n)_n/\mathcal{I}_{\omega})&=(a_nb_n)_{\omega}p=(a_n)_{\omega}p(b_n)_{\omega}p\\
&=\rho((a_n)_n/\mathcal{I}_{\omega})\rho((b_n)_n/\mathcal{I}_{\omega}),\\
\rho((a_n^*)_n/\mathcal{I}_{\omega})&=(a_n^*)_{\omega}p=(p(a_n)_{\omega})^*=((a_n)_{\omega}p)^*\\
&=\rho((a_n)_n/\mathcal{I}_{\omega})^*,
}
whence $\rho|_{M^{\omega}}$ is a *-isomorphism. 
\end{proof}

\begin{proof}[Proof of Theorem \ref{UW2.1.3}]
First, observe that for $(x_n)_n\in \mathcal{M}^{\omega}$ and $(y_n)^{\omega}\in M^{\omega}$, we have 
\[\pi_{\omega}((x_n)_{\omega})w(y_n)^{\omega}\xi_{\varphi^{\omega}}=(x_ny_n\xi_{\varphi_n})_{\omega}=w(x_ny_n)^{\omega}\xi_{\varphi^{\omega}}=w(x_n)^{\omega}(y_n)^{\omega}\xi_{\varphi^{\omega}},\]
so $\pi_{\omega}((x_n)_{\omega})w=w(x_n)^{\omega}$. Hence $M^{\omega}\subset w^*Nw$. To prove $w^*Nw\subset M^{\omega}$, it is enough to show that $w^*\pi_{\omega}((x_n)_{\omega})w\in M^{\omega}$ for $(x_n)_n\in \ell^{\infty}$. Let $(x_n)_n\in \ell^{\infty}$. By Proposition \ref{UW2.1.6}, we have that $(x_n)_n\in \mathcal{M}^{\omega}+\mathcal{L}_{\omega}+\mathcal{L}_{\omega}^*$. Furthermore, by the above, $w^*\pi_{\omega}((x_n)_{\omega})w\in M^{\omega}$ if $(x_n)_n\in \mathcal{M}^{\omega}$. Therefore it suffices to show that $w^*\pi_{\omega}((x_n)_{\omega})w\in M^{\omega}$ when $(x_n)_n\in \mathcal{L}_{\omega}$. But if $(x_n)_n\in \mathcal{L}_{\omega}$ and $(y_n)_n\in \mathcal{M}^{\omega}$, we have $(x_ny_n)_n\in \mathcal{L}_{\omega}$ by Lemma \ref{cor: M^{omega} cap (L+L^*)=I} (1), and so 
\[\pi_{\omega}((x_n)_{\omega})w(y_n)^{\omega}\xi_{\varphi^{\omega}}=(x_ny_n\xi_{\varphi_n})_{\omega}=0,\]
so $w^*\pi_{\omega}((x_n)_{\omega})w=w^*\cdot 0=0\in M^{\omega}$.
\end{proof}

In the next section, we will show (Proposition \ref{prop: wM^{omega}w^*=qNq and ww^*=q}) that $ww^*=q$, where $q=pJ_{\omega}pJ_{\omega}$, and $wM^{\omega}w^*=q(\prod^{\omega}M_n)q$. Here, $J_{\omega}$ is the ultraproduct of $(J_{\varphi_n})_n$.
The following result will be used in $\S$\ref{subsec: Golodets state}, Lemma \ref{lem: characterization of the centralizer1}.
\begin{corollary}\label{cor: decomposition of ell^infty}
For any $(a_n)_n\in \ell^{\infty}$, there exists $(b_n)_n\in \mathcal{M}^{\omega}$, $(c_n)_n\in \mathcal{L}_{\omega}$, and $(d_n)_n\in \mathcal{L}_{\omega}^*$ such that 
\begin{list}{}{}
\item[{\rm{(1)}}] $a_n=b_n+c_n+d_n$ for $n\in \mathbb{N}$.
\item[{\rm{(2)}}] $\displaystyle \|(b_n)^{\omega}\|\le \lim_{n\to \omega}\|a_n\|$.
\end{list}
\end{corollary}

\begin{proof}
Since $(a_n)_n\in \ell^{\infty}$, by Proposition \ref{UW2.1.6}, there exists $(b_n)_n\in \mathcal{M}^{\omega}$, $(c_n)_n\in \mathcal{L}_{\omega}$, and $(d_n)_n\in \mathcal{L}_{\omega}^*$ such that $a_n=b_n+c_n+d_n$. $(b_n)_n$ is unique modulo $\mathcal{I}_{\omega}$, and since $\rho|_{M^{\omega}}\colon M^{\omega}\to pNp$ is a *-isomorphism (Proposition \ref{prop: normalizer iff commutes with p}), we have
\[\|(b_n)^{\omega}\|=\|\rho^{-1}(p(a_n)_{\omega}p)\|
=\|p(a_n)_{\omega}p\|\le \lim_{n\to \omega}\|a_n\|.\]
\end{proof}

\subsection{Standard Forms}\label{subsec: standard form} Our next step is to show (Theorem \ref{UW2.3.2} below) that the Groh-Raynaud ultraproduct of a sequence of standard von Neumann algebras is again standard, in such a way that the standard form of the ultraproduct algebra is obtained as an ultraproduct of the standard forms of the sequence. This result was first obtained by Raynaud \cite[Corollary 3.7]{Raynaud} in the case of constant sequence of algebras.  
We give a different proof, since this plays a crucial role in the proof of Theorem \ref{thm: ultrapower of modular automorphism}. 
For the convenience of the reader, recall the definition of a standard form \cite{Haagerup1}. 
\begin{definition}\label{def: standard form}
Let $(M,H,J,P)$ be a quadruple, where $M$ is a von Neumann algebra, $H$ is a Hilbert space on which $M$ acts, $J$ is an antilinear isometry on $H$ with $J^2=1$, and $P\subset H$ is a closed convex cone which is self-dual, i.e., $P=P^0$, where
\[P^0:=\{\xi \in H;\ \nai{\xi}{\eta}\ge 0,\ \eta\in P\}.\]
Then $(M,H,J,P)$ is called a {\it standard form} if the following conditions are satisfied:
\begin{itemize}
\item[1.] $JMJ=M'$.
\item[2.] $J\xi=\xi,\ \xi\in P$.
\item[3.] $xJxJ(P)\subset P,\ x\in M$.
\item[4.] $JxJ=x^*,\ x\in \mathcal{Z}(M)$.
\end{itemize}
\end{definition}

\begin{theorem}\label{UW2.3.2}Let $(M_n,H_n,J_n,P_n)_n$ be a sequence of standard forms. Let $H_{\omega}:=(H_n)_{\omega}$, let $J_{\omega}$ be defined on $H_{\omega}$ by
\[J_{\omega}(\xi_n)_{\omega}:=(J_n\xi_n)_{\omega},\ \ \ (\xi_n)_{\omega}\in H_{\omega},\]
and let
\[P_{\omega}:=\{(\xi_n)_{\omega}\in H_{\omega};\ \xi_n\in P_n \text{ for all }n\in \mathbb{N}\}.\]
Then the quadruple 
\[\left (\prod^{\omega}M_n,H_{\omega},J_{\omega},P_{\omega}\right )\]
is again a standard form.
\end{theorem}
Conditions 2. and 3. can be easily verified. For 1., we have to show the Raynaud's Theorem that $(\prod^{\omega}M_n)'=\prod^{\omega}M_n'$ (Theorem \ref{UW2.2.2} below). It might look obvious that 4. holds. However, we will see that $\mathcal{Z}(\prod^{\omega}M_n)$ is different from $\prod^{\omega}\mathcal{Z}(M_n)$ in general. Therefore it is not obvious that the equality $J_{\omega}xJ_{\omega}=x^*$ holds for $x\in \mathcal{Z}(\prod^{\omega}M_n)$. However, this can be fixed by showing that condition 4. is redundant (this has general and independent interest):
\begin{lemma}\label{UW2.3.1} Let $(M,H,J,P)$ be a quadruple satisfying, conditions 1.-3. in Definition \ref{def: standard form}. Then $(M,H,J,P)$ satisfies condition 4, whence it is a standard form.
\end{lemma}
We use the following Araki's characterization of the modular conjugation operator.
\begin{theorem}{\rm{\cite[Theorem 1]{Araki}}}\label{thm: characterization of J}
Let $\xi$ be a cyclic and separating vector for a von Neumann algebra $M$ on a Hilbert space $H$. Then a conjugate-linear involution $J$ is the modular conjugation operator associated with the state $\omega_{\xi}=\nai{ \cdot\ \xi}{\xi}$ if and only if $J$ satisfies the following conditions.
\begin{itemize}
\item[{\rm{(1)}}] $JMJ=M'$.
\item[{\rm{(2)}}] $J\xi=\xi$.
\item[{\rm{(3)}}] $\nai{\xi}{aJaJ\xi}\ge 0$ for all $a\in M$, and equality holds if and only if $a=0$.
\end{itemize}
\end{theorem} 
\begin{proof}[Proof of Lemma \ref{UW2.3.1}]The proof is in three steps. Throughout, conditions 1.-3. in Definition \ref{def: standard form} are assumed to hold.

\textbf{Step 1}: Assume first that $M$ has a cyclic and separating vector $\xi \in P$. Then by Theorem \ref{thm: characterization of J}, $J$ is the modular involution associated with $\xi$. But then 4. is immediate from Tomita-Takesaki theory \cite{Takesaki1}.

\textbf{Step 2}: Assume now the slightly more general situation where we have $\xi \in P$ such that $\overline{MM'\xi}=H$. Let $e$ and $e'$ be the projections onto $\overline{M'\xi}$ and $\overline{M\xi}$, respectively. If $f$ is a central projection in $M$, and $f\ge e$, then we have, for all $x\in M$ and $x'\in M'$:
\[fxx'\xi=xfx'\xi=xx'\xi,\]
and so $f=1$; it follows that the central support of $e$ is 1, and similarly, it follows that the central support of $e'$ is 1. Moreover, as
\[JM'\xi=JM'J\xi=M\xi,\]
we have that $JeJ=e'$.

Now, let $f:=ee'$. Then $JfJ=JeJeJ^2=e'e=ee'=f$. By the proof of \cite[Lemma 2.6]{Haagerup1}, it follows that $(fMf,f(H),J|_{f(H)},f(P))$ does also satisfy the conditions 1.-3. But as 
\[\overline{fMf\xi}=\overline{ee'M\xi}=ee'(H)=f(H),\]
and similarly $\overline{fM'f\xi}=f(H)$, we see that $\xi$ is a separating and cyclic vector for $fMf$, acting on $f(H)$. Hence by Step 1, we have
\[J|_{f(H)}dJ|_{f(H)}=d^*\]
for all central element $d$ of $fMf$. But as $e$ and $e'$ have central support 1, the map $c\mapsto fcf$ is a *-isomorphism from the center of $M$ onto the center of $fMf$. We now prove that 4. holds in the case under consideration: let $c\in M\cap M'$, then $JcJ-c^*\in M\cap M'$, so as $JfJ=f$, we get from the above:
\eqa{
0&=J|_{f(H)}fcfJ|_{f(H)}-(fcf)^*\\
&=(JfcfJ-fc^*f)|_{f(H)}\\
&=f(JcJ-c^*)f|_{f(H)},
}
hence $JcJ=c^*$ holds by the injectivity of $c\mapsto fcf$.

\textbf{Step 3}: We now consider the general case. Let $(\xi_{\alpha})_{\alpha}\subset P\setminus \{0\}$ be a maximal family with respect to the property that $(\overline{MM'\xi_{\alpha}})_{\alpha}$ forms an orthogonal family of subspaces of $H$. Let $q_{\alpha}$ be a projection onto $\overline{MM'\xi_{\alpha}}$. The projections  $(q_{\alpha})_{\alpha}$ are clearly central, and as 
\[JMM'\xi_{\alpha}=(JMJ)(JM'J)J\xi_{\alpha}=M'M\xi_{\alpha}=MM'\xi_{\alpha},\]
one has also $Jq_{\alpha}=q_{\alpha}J$ for all $\alpha$. Hence with $p:=1-\sum_{\alpha}q_{\alpha}$, we have $JpJ=p$. Now, assume that $p\neq 0$. As $P$ spans $H$, we may then choose $\eta\in P$ such that $p\eta\neq 0$. Let $\xi=p\eta$. Then 
\[\xi=p\eta=p^2\eta=pJpJ\eta\in P,\]
and as $\xi\perp \overline{MM'\xi_{\alpha}}$ for all $\alpha$, it is easy to see that $\overline{MM'\xi}\perp \overline{MM'\xi_{\alpha}}$ for all $\alpha$. But this contradicts the maximality of $(\xi_{\alpha})_{\alpha}$, so that $p=0$ and hence $\sum_{\alpha}q_{\alpha}=1$. Now, each of the quadruples
\[(q_{\alpha}M,q_{\alpha}(H),J|_{q_{\alpha}(H)},q_{\alpha}(P))\]
satisfy 1.-3. and the condition considered in Step 2, since $q_{\alpha}(H)=\overline{MM'\xi_{\alpha}}$; hence 4. holds for the above quadruple, i.e.,
\[J|_{q_{\alpha}(H)}c_{\alpha}J|_{q_{\alpha}(H)}=c_{\alpha}^*|_{q_{\alpha}(H)}\]
whenever $c_{\alpha}$ is a central element of $q_{\alpha}M$.

Now, let $c\in M\cap M'$. Then $q_{\alpha}c$ is a central element of $q_{\alpha}M$, and so 
\eqa{
JcJ&=J\left (\sum_{\alpha}q_{\alpha}cq_{\alpha}\right )J=\sum_{\alpha}Jq_{\alpha}cq_{\alpha}Jq_{\alpha}\\
&=\sum_{\alpha}J|_{q_{\alpha}(H)}cq_{\alpha}J|_{q_{\alpha}(H)}q_{\alpha}=\sum_{\alpha}c^*q_{\alpha}\\
&=c^*.
}
\end{proof}
Next we show that the Groh-Raynaud ultraproduct preserves commutant. This result was obtained by Raynaud \cite[Theorem 1.8]{Raynaud} in the case of a constant sequence of algebras. 
\begin{lemma}\label{UW2.2.1} Let $(H_n)_n$ be a sequence of Hilbert spaces, and let $M_n\in {\rm{vN}}(H_n)$ for each $n\in \mathbb{N}$. Let $H_{\omega}=(H_n)_{\omega}$, and $M=\prod^{\omega}(M_n,H_n)$ and $N=\prod^{\omega}(M_n',H_n)$. For any $\xi \in H_{\omega}$ and $a'\in M'$, there exists $a\in N$ such that $a\xi=a'\xi$ and $\|a\|\le \|a'\|$.
\end{lemma}
\begin{proof} Let $\xi=(\xi_n)_{\omega}\in H_{\omega}$ and let $a'\in M'$; to prove the lemma, we may and do assume that $\|a'\|=1$. Let $\eta=a'\xi=(\eta_n)_{\omega}$ and put
\[\varepsilon_n:=\sup \{\nai{x\eta_n}{\eta_n}-\nai{x\xi_n}{\xi_n};\ x\in M_n,\ 0\le x\le 1\},\ \ n\in \mathbb{N}.\]
Then $\varepsilon_n\ge 0\ (n\in \mathbb{N})$, and by weak-compactness of Ball$(M_n)\cap M_n^+$, there is $(x_n)_n\in \prod_{n\in \mathbb{N}}M_n$ such that $0\le x_n\le 1$ and 
\[\varepsilon_n=\nai{x_n\eta_n}{\eta_n}-\nai{x_n\xi_n}{\xi_n},\]
for all $n\in \mathbb{N}$. In particular, $\pi_{\omega}(x)\in M$, where $x=(x_n)_{\omega}$. Also, e.g., by \cite[Lemma 3.1]{HW1}, we have 
\[\lim_{n\to \omega}\varepsilon_n=\nai{x\eta}{\eta}-\nai{x\xi}{\xi}=\nai{xa'\xi}{a'\xi}-\nai{x\xi}{\xi}\le 0,\]
and hence $\lim_{n\to \omega}\varepsilon_n=0$. Moreover, by definition of $(\varepsilon_n)_n$, we have 
\[\omega_{\eta_n}(x)\le \omega_{\xi_n}(x)+\varepsilon_n,\ \ \ x\in M_n,\ \ \ 0\le x\le 1,\ \ \ n\in \mathbb{N},\]
so 
\[\omega_{\eta_n}\left (\frac{x}{\|x\|}\right )\le \omega_{\xi_n}\left (\frac{x}{\|x\|}\right )+\varepsilon_n,\ \ \ x\in M_n^+\setminus \{0\},\ \ \ n\in \mathbb{N},\]
and hence by \cite[Lemma 3.2]{HW1}, there exists $(\eta_n')_n\in \prod_{n\in \mathbb{N}}H_n$ such that $\|\eta_n-\eta_n'\|\le \varepsilon_n^{\frac{1}{2}}$ and $\omega_{\eta_n'}\le \omega_{\xi_n}$ for each $n\in \mathbb{N}$. In particular, $(\eta_n)_{\omega}=(\eta_n')_{\omega}$, since $\lim_{n\to \omega}\varepsilon_n=0$. By \cite[Lemma 3.1]{HW1}, we then get $(a_n')_n\in \prod_{n\in \mathbb{N}}M_n'$ such that $\|a_n'\|\le 1$ and $a_n'\xi_n=\eta_n'\ (n\in \mathbb{N})$. Let $a:=\pi_{\omega}((a_n')_{\omega})$. Then $a\in N$, and
\[a\xi=(a_n'\xi_n)_{\omega}=(\eta_n')_{\omega}=(\eta_n)_{\omega}=\eta=a'\xi.\]
Also
\[\|a\|=\lim_{n\to \omega}\|a_n'\|\le 1=\|a'\|.\]
\end{proof} 

\begin{theorem}{\rm{\cite[Theorem 1.8]{Raynaud}}}\label{UW2.2.2}
Let $(M_n,H_n,J_n,P_n)_n$ be as in Theorem \ref{UW2.3.2}. Then one has
\[\left (\prod^{\omega}M_n\right )'=\prod^{\omega}M_n'.\]
\end{theorem}
\begin{proof}
Let $A=(M_n)_{\omega}$ and $B=(M_n')_{\omega}$ and identity these with their images under $\pi_{\omega}$. Then $B\subset A'$ is clear, so it suffices to prove that $A'\subset B''$. Let $a'\in A'$ and $\xi_1,...,\xi_m\in H_{\omega}\ (m\in \mathbb{N})$. Let $F$ be the type I$_m$-factor, acting on $K=\mathbb{C}^m$. Using the matrix picture of $M_n\otimes F$, it is clear that 
\[(M_n\otimes F)_{\omega}=A\otimes F \text{ on } H_{\omega}\otimes K,\]
(as *-algebras) and hence
\[((M_n\otimes F)_{\omega})'=A'\otimes \mathbb{C}1,\]
as von Neumann algebras. Thus $a'\otimes 1\in (\prod^{\omega}(M_n\otimes F, H_n\otimes K))'$, and so by Lemma \ref{UW2.2.1}, there is 
\[a\otimes 1\in \prod^{\omega}((M_n\otimes F)', H_n\otimes K)=\prod^{\omega}M_n'\otimes \mathbb{C}1\]
with $\|a\|\le \|a'\|$ and 
\[(a\otimes 1)(\xi_1,...,\xi_m)=(a'\otimes 1)(\xi_1,...,\xi_m),\]
hence $a\xi_j=a'\xi_j\ (j=1,\cdots ,m)$. This means that $B$ meets any so-neighborhood of $a'$. As $a'\in A'$ was arbitrary, we conclude that $A'\subset B''$, as desired.
\end{proof}
\begin{remark}
It does not follow from the above theorem that the ultraproduct of a sequence of factors is again a factor, because for a standard von Neumann algebra $(M,H)$, the center $\mathcal{Z}(M)$ is not always standardly represented in $H$. We shall in fact prove that this is typically not the case (see Proposition \ref{thm: prod^{omega}R is not a factor}, and Remark \ref{rem: nonfactoriality of Raynaud UP of III_0} below. See also \cite{Raynaud}, Proposition 1.14). 
\end{remark}
Now we prove the main result of this subsection. 
\begin{proof}[Proof of Theorem \ref{UW2.3.2}]
It is clear that $P_{\omega}$ is a closed convex cone. We prove self-duality as follows: assume that $\xi=(\xi_n)_{\omega}\in P_{\omega}^0$. For each $n\in \mathbb{N}$, there is $\eta_n^+,\eta_n^-,\zeta_n^+,\zeta_n^-\in P_n$ such that $\eta_n^+\perp \eta_n^-$, $\zeta_n^+\perp \zeta_n^-$ and $\xi_n=\eta_n^+-\eta_n^-+i(\zeta_n^+-\zeta_n^-)$. 
Then by $\xi\in P_{\omega}^0$, we have 
\eqa{
\lim_{n\to \omega}\nai{\xi_n}{\eta_n^-}&=-\lim_{n\to \omega}\|\eta_n^-\|^2+i\lim_{n\to \omega}\nai{\zeta_n^+-\zeta_n^-}{\eta_n^-}\\
&\ge 0.
}
Therefore $(\eta_n^-)_{\omega}=0$. We also have
\eqa{
\lim_{n\to \omega}\nai{\xi_n}{\zeta_n^{\pm}}&=\lim_{n\to \omega}\left (\nai{\eta_n^+}{\zeta_n^{\pm}}\pm i\|\zeta_n^{\pm}\|^2\right )\\
&\ge 0.
} 
Therefore $(\zeta_n^{\pm})_{\omega}=0$, and $(\xi_n)_{\omega}=(\eta_n^+)_{\omega}\in P_{\omega}$, so $P_{\omega}^0=P_{\omega}$.

By Theorem \ref{UW2.2.2}, it follows that 1. in Definition \ref{def: standard form} holds for the quadruple $(\prod^{\omega}M_n,H_{\omega},J_{\omega},P_{\omega})$, and the properties 2.-3. in Definition \ref{def: standard form} are easily checked. By Lemma \ref{UW2.3.1}, the claim follows.
\end{proof}

\begin{theorem}{\rm{\cite[Theorem 1.1]{Raynaud}}}\label{UW2.3.3} Let $(M_n)_n$ be a sequence of standard von Neumann algebras. Then $(\prod^{\omega}M_n)_*$ is Banach space isomorphic to the Banach space ultraproduct $((M_n)_*)_{\omega}$, in such a way that a normal functional on $\prod^{\omega}M_n$ is implemented by the ultraproduct vectors corresponding to the isomorphic image in $((M_n)_*)_{\omega}$.
\end{theorem}
\begin{proof}
Let $(\varphi_n)_{\omega}\in ((M_n)_*)_{\omega}$. 
As each $M_n$ is standard, we have sequences $(\xi_n)_n,\ (\eta_n)_n\in \prod_{n\in \mathbb{N}}H_n$ such that 
\[\varphi_n(x)=\nai{x\xi_n}{\eta_n},\ \ \ x\in M_n,\ n\in \mathbb{N}.\]
and $\|\varphi_n\|=\|\xi_n\|^2=\|\eta_n\|^2\ (n\in \mathbb{N})$. In particular, both $(\xi_n)_n$ and $(\eta_n)$ are bounded. Define $\xi_{\omega}:=(\xi_n)_{\omega}$ and $\eta_{\omega}:=(\eta_n)_{\omega}$ in $(H_n)_{\omega}$. Then define $\varphi_{\omega}\in (\prod^{\omega}M_n)_*$ by
\[\varphi_{\omega}(x)=\nai{x\xi_{\omega}}{\eta_{\omega}},\ \ \ x\in \prod^{\omega}M_n,\]
and $\|\varphi_{\omega}\|=\lim_{n\to \omega}\|\varphi_n\|$. Hence $\Phi\colon ((M_n)_*)_{\omega}\to (\prod^{\omega}M_n)_*$ defined by $\Phi((\varphi_n)_{\omega}):=\varphi_{\omega}$ is isometric. Note also that for $(x_n)_n\in \ell^{\infty}(\mathbb{N},M_n)$, we have 
\begin{align}
\Phi((\varphi_n)_{\omega})(\pi_{\omega}((x_n)_{\omega}))&=\lim_{n\to \omega}\nai{x_n\xi_n}{\eta_n}\notag \\
&=\lim_{n\to \omega}\varphi_n(x_n).\ \label{eq: Phi is linear}
\end{align}
Since $\pi_{\omega}(\ell^{\infty}(\mathbb{N},M_n))$ is strongly dense in $\prod^{\omega}M_n$, $\Phi((\varphi_n)_{\omega})$ is uniquely determined by Eq. (\ref{eq: Phi is linear}) and is independent of the choice of $(\xi_n)_n,\ (\eta_n)_n$. 
It is clear that $\Phi(\lambda (\varphi_n)_{\omega})=\lambda \Phi((\varphi_n)_{\omega})$ for $\lambda \in \mathbb{C}$ and $(\varphi_n)_{\omega}\in ((M_n)_*)_{\omega}$.
Note also that if $(\varphi_n)_{\omega}, (\psi_n)_{\omega}\in ((M_n)_*)_{\omega}$, 
then by Eq. (\ref{eq: Phi is linear}), for $(x_n)_n\in \ell^{\infty}(\mathbb{N},M_n)$ we have 
\eqa{
\Phi((\varphi_n+\psi_n)_{\omega})(\pi_{\omega}((x_n)_{\omega})&=\lim_{n\to \omega}(\varphi_n+\psi_n)(x_n)\\
&=\lim_{n\to \omega}\varphi_n(x_n)+\lim_{n\to \omega}\psi_n(x_n)\\
&=[\Phi((\varphi_n)_{\omega})+\Phi((\psi_n)_{\omega})](\pi_{\omega}((x_n)_{\omega}).
}
Therefore by the strong density of $\pi_{\omega}(\ell^{\infty}(\mathbb{N},M_n))$, $\Phi((\varphi_n+\psi_n)_{\omega})=\Phi((\varphi_n)_{\omega}+\Phi((\psi_n)_{\omega})$ holds.  
Hence $\Phi$ is linear. 
Surjectivity of $\Phi$ follows from Theorem \ref{UW2.3.2} by reversing the above argument.
Therefore $\Phi$ is an isometric isomorphism. 
\end{proof}
In the following, $(M_n)_n$ is a sequence of standard von Neumann algebras, and we identify $(\varphi_n)_{\omega}\in ((M_n)_*)_{\omega}$ with its image $\varphi_{\omega}$ in $(\prod^{\omega}M_n)_*$. 
\begin{corollary}\label{UW2.3.4}Let $\varphi$ be a normal state on $\prod^{\omega}M_n$. Then there are normal states $\varphi_n\in (M_n)_*$ such that $\varphi=(\varphi_n)_{\omega}$. If all $M_n$ are $\sigma$-finite, then we may choose the states $\varphi_n$ such that they are also faithful.
\end{corollary}
\begin{proof}
Since $\prod^{\omega}M_n$ is standard (Theorem \ref{UW2.3.2}), there exists $\xi_{\varphi}\in P_{\omega}$ such that $\varphi=\omega_{\xi_{\varphi}}$. 
By definition, $\xi_{\varphi}$ has a representative $(\xi_n)_n$ where $\xi_n\in P_n$ for all $n\in \mathbb{N}$. 
As $1=\|\varphi\|=\|\xi_{\varphi}\|^2=\lim_{n\to \omega}\|\xi_n\|^2$, we may choose each $\xi_n$ to be a unit vector. Then $\varphi_n:=\omega_{\xi_n}\in S_{{\rm{n}}}(M_n)$, and $\varphi=(\varphi_n)_{\omega}$. Now suppose each 
$M_n\ (n\in \mathbb{N})$ is $\sigma$-finite and take $\psi_n\in S_{{\rm{nf}}}(M_n)\ (n\in \mathbb{N})$.  
Let
\[\varphi_n':=\left (1-\frac{1}{n}\right )\varphi_n+\frac{1}{n}\psi_n, \ \ \ \ n\in \mathbb{N}.\]
Then $\varphi_n'$ is a normal faithful state on $M_n$ for each $n\in \mathbb{N}$, and $(\varphi_n')_{\omega}=(\varphi_n)_{\omega}=\varphi$.
\end{proof}

Recall that 
\begin{lemma}{\rm{\cite[Corollary 2.5 and Lemma 2.6]{Haagerup1}}}\label{lem: reduced standard form}
Let $(M,H,J,P)$ be a standard form, $p$ a projection in $M$, and $q=pJpJ$. Then $(qMq,q(H),qJq,q(P))$ is standard, and $pMp\ni pxp\mapsto qxq\in qMq$ is an isomorphism.
\end{lemma}

Therefore by Proposition \ref{prop: normalizer iff commutes with p}, Theorem \ref{UW2.3.2} and Lemma \ref{lem: reduced standard form}, we have
\begin{corollary}\label{cor: Ocneanu UP is qNq} Let $M^{\omega}=(M_n,\varphi_n)^{\omega}$, $\varphi^{\omega}=(\varphi_n)^{\omega}$, $N=\prod^{\omega}M_n$ $q=pJ_{\omega}pJ_{\omega}$ and $H_{\varphi^{\omega}}=L^2(M^{\omega},\varphi^{\omega})$. Then 
$(M^{\omega}, H_{\varphi^{\omega}},J_{\varphi^{\omega}},P_{\varphi^{\omega}})$ is isomorphic to $(qNq, qH_{\omega}, qJ_{\omega}q, qP_{\omega})$ as a standard form. 
\end{corollary}

\begin{corollary}\label{prop: wM^{omega}w^*=qNq and ww^*=q}
Under the same notation as in Theorem \ref{UW2.1.3}, the following holds.
\begin{itemize}
\item[{\rm{(1)}}] $ww^*=q=pJ_{\omega}pJ_{\omega}$.
\item[{\rm{(2)}}] $wM^{\omega}w^*=q(\prod^{\omega}M_n)q$.
\end{itemize}
\end{corollary}
\begin{proof}
Let $\xi_{\omega}=(\xi_{\varphi_n})_{\omega}$. Then consider the GNS representation $\pi_{\varphi^{\omega}}$ of $M^{\omega}$with respect to $\varphi^{\omega}$.  Recall also by Proposition \ref{prop: normalizer iff commutes with p} that 
$\rho_0:=\rho|_{M^{\omega}}\colon M^{\omega}\ni (x_n)^{\omega}\mapsto (x_n)_{\omega}p\in p(\prod^{\omega}M_n)p$ is a *-isomorphism, so we have another representation $\lambda$ of $M^{\omega}$ on 
$q(L^2(M_n,\varphi_n)_{\omega})$ given by $\lambda((x_n)^{\omega}):=q\rho_0((x_n)^{\omega})q=q(x_n)_{\omega}q,\ \ (x_n)^{\omega}\in M^{\omega}$. Since $\xi_{\omega}$ is cyclic for $q(\prod^{\omega}M_n)q$, it is cyclic for $\lambda(M^{\omega})$, and for $(x_n)^{\omega}\in M^{\omega}$, 
\eqa{
\nai{\lambda((x_n)^{\omega})\xi_{\omega}}{\xi_{\omega}}&=\nai{q(x_n)_{\omega}q\xi_{\omega}}{\xi_{\omega}}\\
&=\lim_{n\to \omega}\nai{x_n\xi_{\varphi_n}}{\xi_{\varphi_n}}\\
&=\varphi^{\omega}((x_n)^{\omega}).
}
Therefore by the proof of the uniqueness of GNS representation, there is a unitary $\tilde{w}\colon L^2(M^{\omega},\varphi^{\omega})\to q(L^2(M_n,\varphi_n))_{\omega}$ determined by 
\[\tilde{w}\pi_{\varphi^{\omega}}((x_n)^{\omega})\xi_{\varphi^{\omega}}=\lambda((x_n)^{\omega})\xi_{\omega},\ \ (x_n)^{\omega}\in M^{\omega},\]
which implements the unitary equivalence of $\pi_{\varphi^{\omega}}$ and $\lambda$. But by Proposition \ref{prop: normalizer iff commutes with p}, for $(x_n)^{\omega}\in M^{\omega}$, $(x_n)_{\omega}q=q(x_n)_{\omega}$ holds, whence
\eqa{
\tilde{w}\pi_{\varphi^{\omega}}((x_n)^{\omega})\xi_{\varphi^{\omega}}&=q(x_n)_{\omega}\xi_{\omega}=(x_n)_{\omega}\xi_{\omega}\\
&=w\pi_{\varphi^{\omega}}((x_n)^{\omega}),
}
and $w=\tilde{w}$ holds. Therefore $ww^*=q$. (2) By Theorem \ref{UW2.1.3}, it holds that $wM^{\omega}w^*=ww^*(\prod^{\omega}M_n)ww^*=q(\prod^{\omega}M_n)q$. 
\end{proof}

Next corollary shows that every normal faithful state on the Ocneanu ultraproduct is the ultraproduct state for some sequence of normal faithful states.
\begin{corollary}\label{newUW2.3.5}
Under the same notation as in Theorem \ref{UW2.1.3}, let $\psi$ be a normal faithful state on $M^{\omega}=(M_n,\varphi_n)^{\omega}$. Then there exists $\psi_n\in S_{{\rm{nf}}}(M_n)\ (n\in \mathbb{N})$ such that $(M_n,\psi_n)^{\omega}=M^{\omega}$ and $\psi=(\psi_n)^{\omega}$. 
\end{corollary}
\begin{proof}
Let $N:=\prod^{\omega}M_n$. Define the isometry $w\colon L^2(M^{\omega},\varphi^{\omega})\to (L^2(M_n,\varphi_n))_{\omega}$ as in Theorem \ref{UW2.1.3}. Let $\hat{\psi}$ be a normal state on $N$ given by 
\[\hat{\psi}(x):=\psi(w^*xw),\ \ \ x\in N.\]
Note that $w^*(\prod^{\omega}M_n)w=(M_n,\varphi_n)^{\omega}$ by Theorem \ref{UW2.1.3}. Then $\text{supp}(\hat{\psi})$ is $p=\text{supp}((\varphi_n)_{\omega})$. By Corollary \ref{UW2.3.4}, we may choose normal faithful states $\psi_n$ on each $M_n$ such that $\hat{\psi}=(\psi_n)_{\omega}$. Now by (the proof of) Proposition \ref{UW2.1.6}, for $(x_n)_n\in \ell^{\infty}(\mathbb{N},M_n)$ we have 
\eqa{
(x_n)_n\in \mathcal{I}_{\omega}(M_n,\varphi_n)&\Leftrightarrow (x_n)_{\omega}\in p^{\perp}Np^{\perp}\\
&\Leftrightarrow (x_n)_n\in \mathcal{I}_{\omega}(M_n,\psi_n), 
}
so $\mathcal{I}_{\omega}(M_n,\varphi_n)=\mathcal{I}_{\omega}(M_n,\psi_n)$, which implies that $(M_n,\varphi_n)^{\omega}=(M_n,\psi_n)^{\omega}$.
Recall (Corollary \ref{prop: wM^{omega}w^*=qNq and ww^*=q}) also that $\Phi\colon (M_n,\psi_n)^{\omega}\ni x \mapsto wxw^*\in qNq$ gives a *-isomorphism such that $(\psi_n)_{\omega}|_{qNq}\circ \Phi=(\psi_n)^{\omega}$. Therefore for $x\in (M_n,\psi_n)^{\omega}=(M_n,\varphi_n)^{\omega}$, we have
\[(\psi_n)^{\omega}(x)=\hat{\psi}(wxw^*)=\psi(w^*wxw^*w)=\psi(x).\]
\end{proof}
\subsection{Golodets' Asymptotic Algebra $C_M^{\omega}\cong M'\cap M^{\omega}$}\label{subsec: Golodets' asymptotic algebra} 
In this section, we describe Golodets' construction of the asymptotic algebra $C_M^{\omega}$ from our viewpoint. Let $M$ be a $\sigma$-finite von Neumann algebra, and let $\varphi\in S_{\rm{nf}}(M)$. Consider the GNS representation of $M$ associated with $\varphi$, so that $\varphi=\omega_{\xi_{\varphi}}$ with a cyclic and separating vector $\xi_{\varphi}$ on a Hilbert space $H$. Consider the following state $\overline{\varphi}$ on $\ell^{\infty}=\ell^{\infty}(\mathbb{N},M)$:
\[\overline{\varphi}((x_n)_n):=\lim_{n\to \omega}\varphi(x_n),\ \ \ \ (x_n)_n\in \ell^{\infty}(\mathbb{N},M).\]   
Let $\pi_{\rm{Gol}}\colon \ell^{\infty}\to \mathbb{B}(H_{\rm{Gol}})$ be the GNS representation of $\overline{\varphi}$ with cyclic vector $\overline{\xi}$ satisfying $\overline{\varphi}=\omega_{\overline{\xi}}$. Let $e_{\omega}$ be the projection of $H_{\rm{Gol}}$ onto $\overline{\pi_{\rm{Gol}}(\ell^{\infty})'\overline{\xi}}$. Define 
\[\mathscr{R}:=e_{\omega}\pi_{\rm{Gol}}(\ell^{\infty})''e_{\omega}\subset \mathbb{B}(e_{\omega}H_{\rm{Gol}}).\]
Let $\overline{N}$ be the set of all $\overline{x}=(x_n)_n\in \ell^{\infty}$ for which $\pi_{\rm{Gol}}(\overline{x})e_{\omega}\in \mathscr{R}$ and $\pi_{\rm{Gol}}(\overline{x}^*)e_{\omega}\in \mathscr{R}$ hold. Then \cite[Lemma 2.3.3]{Golodets}, $\overline{N}$ is a C$^*$-subalgebra of $\ell^{\infty}$. Moreover, 
\[I:=\left \{\overline{x}=(x_n)_n\in \overline{N};\ \lim_{n\to \omega}\varphi(x_n^*x_n)=0\right \}\]
is a closed two-sided ideal in $\overline{N}$, and $\pi_{\rm{Gol}}(\overline{N})e_{\omega}=\mathscr{R}\cong \overline{N}/I$. 
Let $\overline{M}_d$ be the subspace of $\ell^{\infty}(\mathbb{N},M)$ consisting of constant sequences $(x,x,\cdots )_n,\ \ x\in M$. The the {\it asymptotic algebra} $C_M^{\omega}$ of $M$ is defined by 
\[C_M^{\omega}:=\mathscr{R}\cap \pi_{\rm{Gol}}(\overline{M}_d)'\subset \mathbb{B}(e_{\omega}H_{\rm{Gol}}).\]
We show that Golodets' construction is equivalent to Ocneanu's construction.  
\begin{theorem}\label{thm: identification of Golodets algebra}
$\overline{N}=\mathcal{M}^{\omega},\ \mathscr{R}\cong M^{\omega}$ and $C_M^{\omega}\cong M'\cap M^{\omega}$ hold.
\end{theorem}
Let $H_{\omega}$ be the ultrapower of $H$, and let $\pi_{\omega}\colon \ell^{\infty}\to \mathbb{B}(H_{\omega})$ be the ultrapower map $(a_n)_n\mapsto (a_n)_{\omega}$ (we identify $(a_n)_{\omega}$ with its image in $\mathbb{B}(H_{\omega})$ as before), and let $\xi_{\omega}:=(\xi_{\varphi})_{\omega}\in H_{\omega}$. 
Let $J=J_{\varphi}$ be the modular conjugation and $J_{\omega}$ be its ultrapower. For $\overline{x}\in \ell^{\infty}$, we have
\[\overline{\varphi}(\overline{x})=\lim_{n\to \omega}\nai{x_n\xi_{\varphi}}{\xi_{\varphi}}=\nai{(x_n)_{\omega}\xi_{\omega}}{\xi_{\omega}}.\]
Therefore by the uniqueness of GNS  representation, we may identify
\[H_{\rm{Gol}}=\overline{\pi_{\omega}(\ell^{\infty})\xi_{\omega}},\ \  \overline{\xi}=\xi_{\omega},\ \  \pi_{\rm{Gol}}(\overline{x})=\pi_{\omega}(\overline{x})|_{H_{\rm{Gol}}},\ \ \ (\overline{x}\in \ell^{\infty}).\]
Recall that we defined a projection $p$ of $H_{\omega}$ onto $\overline{(\prod^{\omega}M)'\xi_{\omega}}$ (see Definition \ref{def: support projection p}). Then $J_{\omega}pJ_{\omega}$ is the projection of $H_{\omega}$ onto $\overline{(\prod^{\omega}M)\xi_{\omega}}$. We use such abbreviation as $\mathcal{M}^{\omega}, \mathcal{I}_{\omega}, \mathcal{L}_{\omega}$ given in $\S$\ref{subsec: Relation ship between Ocneanu and Rayanud}.   
\begin{lemma}\label{lem: is in normalizer iff commutes with q}
Let $\overline{x}=(x_n)_n\in \ell^{\infty}(\mathbb{N},M)$, and let $q=pJ_{\omega}pJ_{\omega}$. Then $\overline{x}\in \mathcal{M}^{\omega}$ if and only if $q(x_n)_{\omega}=(x_n)_{\omega}q$.
\end{lemma}
\begin{proof}
If $\overline{x}\in \mathcal{M}^{\omega}$, then $p(x_n)_{\omega}=(x_n)_{\omega}p$ holds by Proposition \ref{prop: normalizer iff commutes with p}. Since $p, (x_n)_{\omega}(\in \prod^{\omega}M)$ commute with $J_{\omega}pJ_{\omega}$, we have $(x_n)_{\omega}q=q(x_n)_{\omega}$.\\
Conversely, suppose that $(x_n)_{\omega}q=q(x_n)_{\omega}$ holds.  This implies that 
\begin{equation}
(x_n^*)_{\omega}p^{\perp}J_{\omega}pJ_{\omega}=J_{\omega}pJ_{\omega}p^{\perp}(x_n^*)_{\omega}.\label{eq: (x_n)_omega commutes with p^{perp}p'}
\end{equation} Let $(a_n)_{\omega}\in \mathcal{I}_{\omega}$. We have to show that $(x_na_n)_n^*\in \mathcal{L}_{\omega}$ and $(a_nx_n)_n\in \mathcal{L}_{\omega}$. By $(a_nx_n)_n\in \ell^{\infty}$, Eq. (\ref{eq: (x_n)_omega commutes with p^{perp}p'})  and $(a_n^*)_{\omega}=(a_n^*)_{\omega}p^{\perp}$, we have 
\eqa{
(a_n^*x_n^*)_{\omega}\xi_{\omega}&=J_{\omega}pJ_{\omega}(a_n^*x_n^*)_{\omega}\xi_{\omega}=J_{\omega}pJ_{\omega}(a_n^*)_{\omega}p^{\perp}(x_n)_{\omega}\xi_{\omega}\\
&=(a_n^*)_{\omega}J_{\omega}pJ_{\omega}p^{\perp}(x_n^*)_{\omega}\xi_{\omega}\\
&=(a_n^*)_{\omega}(x_n^*)_{\omega}J_{\omega}pJ_{\omega}p^{\perp}\xi_{\omega}=0,
}
hence $(a_n^*x_n^*)_n\in \mathcal{L}_{\omega}$. $(a_nx_n)_n\in \mathcal{L}_{\omega}$ is proved similarly. Therefore $(x_na_n)_n,\ (a_nx_n)_n$ belongs to $\mathcal{I}_{\omega}$ and $(x_n)_n\in \mathcal{M}^{\omega}$ holds. 
\end{proof}
\begin{proof}[Proof of Theorem \ref{thm: identification of Golodets algebra}]
Let $e'_{\rm{Gol}}$ be the projection of $H_{\omega}$ onto $H_{\rm{Gol}}$. 
By Lemma \ref{UW2.1.5} (2), we have 
\[e'_{\rm{Gol}}H_{\omega}=\overline{\pi_{\omega}(\ell^{\infty})\xi_{\omega}}=\overline{\pi_{\omega}(\ell^{\infty})''\xi_{\omega}}=J_{\omega}pJ_{\omega}H_{\omega}.\]
Therefore $e'_{\rm{Gol}}=J_{\omega}pJ_{\omega}$. Furthermore, as $\pi_{\rm{Gol}}=\pi_{\omega}|_{H_{\rm{Gol}}}$, we have 
\eqa{
e_{\omega}H_{\rm{Gol}}&=\overline{\pi_{\rm{Gol}}(\ell^{\infty})'\xi_{\omega}}=\overline{e'_{\rm{Gol}}\pi_{\omega}(\ell^{\infty})'e'_{\rm{Gol}}\xi_{\omega}}\\
&=e'_{\rm{Gol}}\overline{\pi_{\omega}(\ell^{\infty})'\xi_{\omega}}=J_{\omega}pJ_{\omega}pH_{\omega}\\
&=qH_{\omega}.
}
Therefore it holds that (cf. Proposition \ref{prop: wM^{omega}w^*=qNq and ww^*=q})
\eqa{
\mathscr{R}&=e_{\omega}\pi_{\rm{Gol}}(\ell^{\infty})''e_{\omega}=q\pi_{\omega}(\ell^{\infty})''q\\
&=q\left (\prod^{\omega}M\right )q\cong M^{\omega}.
}
We now show that $\overline{N}=\mathcal{M}^{\omega}$.\\
Suppose $\overline{x}\in \mathcal{M}^{\omega}$. Then by Lemma \ref{lem: is in normalizer iff commutes with q}, 
\[\pi_{\rm{Gol}}(\overline{x})e_{\omega}=\pi_{\omega}(\overline{x})q=q\pi_{\omega}(\overline{x})\in \mathscr{R}.\]
Similarly $\pi_{\rm{Gol}}(\overline{x^*})e_{\omega}\in \mathscr{R}$ holds, and $\overline{x}\in \overline{N}$. Conversely, suppose $\overline{x}\in \overline{N}$. Then $\pi_{\rm{Gol}}(\overline{x})e_{\omega}\in \mathscr{R}$, so 
\[\pi_{\rm{Gol}}(\overline{x})e_{\omega}=e_{\omega}\pi_{\rm{Gol}}(\overline{x}),\]
whence $\pi_{\omega}(\overline{x})q=q\pi_{\omega}(\overline{x})$. By Lemma \ref{lem: is in normalizer iff commutes with q}, $\overline{x}\in \mathcal{M}^{\omega}$ holds. Therefore $\overline{N}=\mathcal{M}^{\omega}$. Note that by Corollary \ref{cor: M^{omega} cap (L+L^*)=I}, this also shows that \[I=\overline{N}\cap \mathcal{L}_{\omega}=\mathcal{M}^{\omega}\cap \mathcal{L}_{\omega}=\mathcal{I}_{\omega}.\]
Finally, as the constant sequence $M$ in $M^{\omega}$ is mapped to $\pi_{\rm{Gol}}(\overline{M}_d)$ under the isomorphism $M^{\omega}\cong q(\prod^{\omega}M)q$, we see that $C_M^{\omega}\cong M'\cap M^{\omega}$. This finishes the proof. 
\end{proof}
\begin{remark}
Golodets \cite{Golodets} defined $C_M^{\omega}$ for a factor $M$ with separable predual. As the above proof shows, that isomorphism class of $C_M^{\omega}$ does not depend on $\varphi$, holds for any $\sigma$-finite von Neumann algebra. On the other hand, the state $\dot{\varphi}=\overline{\varphi}|_{C_M^{\omega}}$ does not depend on $\varphi$ if $M$ is a factor. This point will be clarified in $\S$\ref{subsec: Golodets state}.  
\end{remark}
\section{Tomita-Takesaki Theory for Ultraproducts}\label{sec: Tomita-Takesaki Theory for Ultraproducts}
\subsection{Modular Automorphism Group of a Ultraproduct State}
In this section we show that the ultraproduct action of the modular automorphism group on the Ocneanu ultraproduct is still continuous. This is the key result for all the subsequent analysis.  
In the case of constant algebras, similar results were obtained by Golodets \cite{Golodets} for his auxiliary algebra $\mathscr{R}$, and by Raynaud \cite{Raynaud} for the corner $p(\prod^{\omega}M)p$ which corresponds to $M^{\omega}$ (see Proposition \ref{prop: normalizer iff commutes with p}. Here, $p$ is the support projection of $(\varphi_n)_{\omega}$ as in Definition \ref{def: support projection p}). 

We prove next the corresponding result for a general sequence of $\sigma$-finite von Neumann algebras with normal faithful states. 
\begin{theorem}\label{thm: ultrapower of modular automorphism} Let $\{(M_n,\varphi_n)\}_{n=1}^{\infty}$ be a sequence of von Neumann algebras with faithful normal states. Let $M^{\omega}=(M_n,\varphi_n)^{\omega},\ \varphi^{\omega}=(\varphi_n)^{\omega}$. Then 
\[\sigma_t^{\varphi^{\omega}}((x_n)^{\omega})=(\sigma_t^{\varphi_n}(x_n))^{\omega},\ \ \ t\in \mathbb{R},\ \ (x_n)^{\omega}\in M^{\omega}.\]
In particular, $t\mapsto (\sigma_t^{\varphi_n})^{\omega}$ is a continuous flow on $(M_n,\varphi_n)^{\omega}$.
\end{theorem}

This is not an obvious result as it might look for the first sight. Indeed, it is known that the ultrapower of a continuous action of a topological group on a von Neumann algebra $M$ is often discontinuous on $M^{\omega}$. 

To prove Theorem \ref{thm: ultrapower of modular automorphism}, we need preparations. 
Consider a sequence $(H_n=L^2(M_n,\varphi_n))_n$ of standard Hilbert spaces, let $H_{\omega}=(H_n)_{\omega}$, and as before we identify $(a_n)_{\omega}$ with $\pi_{\omega}((a_n)_{\omega})\in \mathbb{B}(H_{\omega})$ for every $(a_n)_{\omega}\in (\mathbb{B}(H_n))_{\omega}$. 
\begin{lemma}\label{Uffe: g(a_n)_{omega}}
Let $(a_n)_n\in \prod_{n\in \mathbb{N}}\mathbb{B}(H_n)_{{\rm{sa}}}$ with spectra satisfying $\sigma (a_n)\subset [0,1]$ and $0,1\notin \sigma_p(a_n)$, $n\in \mathbb{N}$. Then $a=(a_n)_{\omega}\in \mathbb{B}(H_{\omega})$ satisfies $0\le a\le 1$. Moreover, if $K\subset H_{\omega}$ is a closed subspace invariant under $a$, and with 
\begin{equation}
K\cap {\rm{Ker}}(a)=\{0\}=K\cap {\rm{Ker}}(1-a),\label{eq: first claim}
\end{equation}
then for every bounded continuous function $g$ on $(0,1)$, we have:
\begin{list}{}{}
\item[{\rm{(i)}}] $K$ is invariant under $(g(a_n))_{\omega}$;
\item[{\rm{(ii)}}] $(g(a_n))_{\omega}|_K=g(a|_K)$.
\end{list}
\end{lemma}
\begin{remark}
From the assumption and Eq. (\ref{eq: first claim}), $\sigma (a|_K)\subset [0,1]$ and $0,1\notin \sigma_p(a|_K)$, so $g(a|_K)$ is well-defined.
\end{remark}
\begin{proof}
From the identification $a=\pi_{\omega}((a_n)_{\omega})$, it is straightforward that $0\le a\le 1$, also that
\begin{equation}
(f(a_n))_{\omega}=f(a)\label{eq: f(a_n)}
\end{equation}
for every polynomial $f$ on $[0,1]$, hence (by Weierstrass' Theorem) for every continuous function on $[0,1]$.

Now, let $p$ denote the projection of $H_{\omega}$ onto $K$. As $a(K)\subset K$, $p$ commutes with $a$, hence with all spectral projections of $a$. In particular, $K$ is perpendicular to both $\text{Ker}(a)$ and $\text{Ker}(1-a)$, so $K\subset 1_{(0,1)}(a)(H_{\omega})$, where $1_{X}$ denotes the indicator function of $X\subset \mathbb{R}$. Hence $K=\bigvee_{0<\varepsilon<\frac{1}{2}}K_{\varepsilon}$, where $K_{\varepsilon}=1_{(\varepsilon,1-\varepsilon)}(a)(H_{\omega})\cap K$.

Fix now $\varepsilon\in (0,\frac{1}{2})$ and fix a continuous function $f_{\varepsilon}$ on $[0,1]$ with 
\[f_{\varepsilon}([0,\textstyle \frac{\varepsilon}{2}]\cup [1-\frac{\varepsilon}{2},1])=\{0\},\ \ f_{\varepsilon}((\varepsilon,1-\varepsilon])=\{1\}.\]
Also, fix $\xi=(\xi_n)_{\omega}\in K_{\varepsilon}$ and a continuous function $g$ on $(0,1)$. Choose a continuous function $h$ on $[0,1]$ such that $g(t)=h(t)$ whenever $\frac{\varepsilon}{2}\le t\le 1-\frac{\varepsilon}{2}$. Also, let $\xi_n':=f_{\varepsilon}(a_n)\xi_n$, $n\in \mathbb{N}$. Then 
\begin{equation}
(\xi_n')_{\omega}=(f_{\varepsilon}(a_n)\xi_n)_{\omega}=f_{\varepsilon}(a)\xi=\xi,\label{eq: xi_n'}
\end{equation}
where we used Eq. (\ref{eq: f(a_n)}) in the second last equality, and in the last equality that $\xi \in 1_{(\varepsilon,1-\varepsilon)}(a)(H_{\omega})$ and $f_{\varepsilon}1_{(\varepsilon,1-\varepsilon)}=1_{(\varepsilon,1-\varepsilon)}$. Because $g=h$ on the support of $f_{\varepsilon}$, we have
\begin{equation}
g(a_n)\xi_n'=h(a_n)\xi_n',\ \ \ \ n\in \mathbb{N}.\label{eq: g(a_n)xi_n'}
\end{equation}
Also, as $p$ commutes with $1_{(\varepsilon,1-\varepsilon)}(a)$, and $g=h$ on $(\varepsilon,1-\varepsilon)$, one has 
\eqa{
g(a|_K)\xi&=g(a|_K)1_{(\varepsilon,1-\varepsilon)}(a)p\xi\\
&=g(a|_K)1_{(\varepsilon,1-\varepsilon)}(a|_K)p\xi\\
&=h(a|_K)\xi,
}
but as $h$ is continuous on $[0,1]$ (hence bounded) and $\xi \in K$, this entails 
\begin{equation}
g(a|_K)\xi=h(a)\xi. \label{eq: h(a)xi}
\end{equation}
Now we get, by Eqs. (\ref{eq: xi_n'}), (\ref{eq: g(a_n)xi_n'}), (\ref{eq: xi_n'}), (\ref{eq: f(a_n)}) and (\ref{eq: h(a)xi}) respectively:
\eqa{
(g(a_n))_{\omega}\xi&=(g(a_n)\xi_n')_{\omega}=(h(a_n)\xi_n')_{\omega}\\
&=(h(a_n))_{\omega}\xi=h(a)\xi=g(a|_K)\xi.
}
As $K=\bigvee_{0<\varepsilon<\frac{1}{2}}K_{\varepsilon}$, we now get (i) and (ii).
\end{proof}
\begin{lemma}\label{Uffe: modular operator}
For each $n\in \mathbb{N}$, let $\Delta_n$ be a positive self-adjoint (possibly unbounded) operator on $H_n$, such that $0\notin \sigma_p(\Delta_n)$. Let $a=((1+\Delta_n)^{-1})_{\omega}\in \mathbb{B}(H_{\omega})$, and let $K\subset H_{\omega}$ be a closed subspace which is invariant under $a$. If $\Delta$ is a positive self-adjoint operator on $K$ with $0\notin \sigma_p(\Delta)$ and with 
\[(1+\Delta)^{-1}=a|_K,\]
then 
\[(\Delta_n^{it})_{\omega}|_K=\Delta^{it},\ \ \ t\in \mathbb{R}.\]
\end{lemma}
\begin{proof}
Let $a_n:=(1+\Delta_n)^{-1},\ n\in \mathbb{N}$. Define 
\[g_t(x):=(x^{-1}-1)^{it},\ \ x\in (0,1),\ \ t\in \mathbb{R}.\]
As $0,1\notin \sigma_p(a|_K)=\sigma_p((1+\Delta)^{-1})$, Lemma \ref{Uffe: g(a_n)_{omega}} gives 
\[(g_t(a_n))_{\omega}|_K=g_t(a|_K),\ \ \ t\in \mathbb{R}.\]
This shows the claim, as
\[g_t(a_n)=\Delta_n^{it},\ \ \ n\in \mathbb{N},\ \ \ t\in \mathbb{R},\]
and 
\[g_t(a|_K)=g_t((1+\Delta)^{-1})=\Delta^{it},\ \ \ t\in \mathbb{R}.\]
\end{proof}

\begin{lemma}\label{lem: e wedge f}
Let $e,f$ be projections on a real Hilbert space such that $e\wedge f+e^{\perp}\wedge f=f$. Then $ef=fe$ holds.
\end{lemma}
\begin{proof}
Note that $e\wedge f= f(e\wedge f)f\le fef$ and similarly $e^{\perp}\wedge f\le fe^{\perp}f$. 
Moreover, $fef+fe^{\perp}f=f$. Hence $e\wedge f+e^{\perp}\wedge f=f$ implies that $e\wedge f=fef$ and $e^{\perp}\wedge f=fe^{\perp}f$. 
It follows that $efe$ is a positive self-adjoint operator whose square $(efe)^2=efefe=fef$ is the projection $e\wedge f$, whence $efe$ itself is the projection $e\wedge f$. 
It then holds that
\eqa{
(ef-fe)^*(ef-fe)&=fef-efef-fefe+efe\\
&=e\wedge f-e(e\wedge f)-(e\wedge f)e+e\wedge f\\
&=0,
}
whence $ef=fe$ holds.
\end{proof}

\begin{proof}[Proof of Theorem \ref{thm: ultrapower of modular automorphism}]
Consider for each $n\in \mathbb{N}$ the standard representation of $M_n$ on $H_n=L^2(M_n,\varphi_n)$, and write $N=\prod^{\omega}M_n$ for simplicity. 
Define $H_{\omega}:=(H_n)_{\omega}$ and $J_{\omega}:=(J_{\varphi_n})_{\omega}$. 
Let $\varphi_{\omega}:=(\varphi_n)_{\omega}\in ((M_n)_*)_{\omega}\cong N_*$ and let $p:=\text{supp}(\varphi_{\omega})\in \prod^{\omega}M_n$ (cf. Theorem \ref{UW2.3.3}). 
By Proposition \ref{prop: wM^{omega}w^*=qNq and ww^*=q}, we have $M^{\omega}\cong qNq=wM^{\omega}w^*$ with $q:=pJ_{\omega}pJ_{\omega}$, and with this identification, we have 
\[L^2(M^{\omega},\varphi^{\omega})\cong qH_{\omega}, J_{\varphi^{\omega}}\cong qJ_{\omega}q=pJ_{\omega}p.\]
 
Let $S_{\varphi_n}$ (resp. $F_{\varphi_n}$) be the closure of the closable (conjugate-linear) operator $S_{\varphi_n}^0$ (resp. $F_{\varphi_n}^0$) on $H_n$ defined by 
\eqa{
\dom{S_{\varphi_n}^0}=M_n\xi_{\varphi_n},\ \ \ \ \  & S_{\varphi_n}^0x\xi_{\varphi_n}:=x^*\xi_{\varphi_n}\ \ (x\in M_n),\\
\dom{F_{\varphi_n}^0}=M_n'\xi_{\varphi_n},\ \ \ \ \  & F_{\varphi_n}^0 y\xi_{\varphi_n}:=y^*\xi_{\varphi_n}\ \ (y\in M_n').
}
Since $F_{\varphi_n}$ is the adjoint of $S_{\varphi_n}$ (cf. \cite{Takesaki5} or \cite{KadisonRingrose}, Corollary 9.2.30), 
they are the adjoint of each other with respect to the real Hilbert space structure of $H_n$. 
Therefore by \cite[Theorem 13.10]{Rudin}, we have the following decomposition as a real Hilbert space:
\[H_n\oplus_{\mathbb{R}}H_n=G(S_{\varphi_n})\oplus_{\mathbb{R}}VG(F_{\varphi_n}),\ \ \ \ \ \ n\in \mathbb{N},\]
where $V=\mattwo{0}{-1}{1}{0}$ and $G(T)$ is the graph of a closed operator $T$.
Taking the ultraproduct (as a real Hilbert space), we obtain
\begin{equation}
H_{\omega}\oplus_{\mathbb{R}} H_{\omega}=(G(S_{\varphi_n}))_{\omega}\oplus_{\mathbb{R}} V_{\omega}(G(F_{\varphi_n}))_{\omega},\label{eq: H_omega plus H_omega decomposition}
\end{equation}
where $V_{\omega}=(V)_{\omega}$. Let $\tilde{\varphi}_{\omega}:=\varphi_{\omega}|_{qNq}\in (qNq)_*$ be the image of $\varphi^{\omega}$ under the isomorphism $M^{\omega}\cong qNq$. 
Let $x\in qNq$. Then by Lemma \ref{UW2.1.5} (1) and Proposition \ref{prop: normalizer iff commutes with p}, there exists $(x_n)_n\in \mathcal{M}^{\omega}$ such that $x=(x_n)_{\omega}q=q(x_n)_{\omega}$. Therefore it holds that 
\eqa{
S_{\tilde{\varphi}_{\omega}}x\xi_{\omega}&=S_{\tilde{\varphi}_{\omega}}(x_n)_{\omega}\xi_{\omega}\\
&=(x_n^*)_{\omega}(\xi_{\varphi_n})_{\omega}\\
&=(S_{\varphi_n}x_n\xi_{\varphi_n})_{\omega},
}
which shows that $(x\xi_{\omega},S_{\tilde{\varphi}_{\omega}}x\xi_{\omega})\in qH_{\omega}\oplus_{\mathbb{R}}qH_{\omega}\cap (G(S_{\varphi_n}))_{\omega}$. 
Doing similar computations for $F_{\tilde{\varphi}_{\omega}}$, we obtain
\begin{align}
G(S_{\tilde{\varphi}_{\omega}})&\subset (G(S_{\varphi_n}))_{\omega}\cap (qH_{\omega}\oplus_{\mathbb{R}} qH_{\omega}), \label{eq: G(S_tilde)}\\
V_{\omega}G(F_{\tilde{\varphi}_{\omega}})&\subset V_{\omega}(G(F_{\varphi_n}))_{\omega}\cap (qH_{\omega}\oplus_{\mathbb{R}} qH_{\omega}).\label{eq: G(F_tilde)}
\end{align}
Similarly, using $F_{\tilde{\varphi}_{\omega}}=(S_{\tilde{\varphi}_{\omega}})^*$, we have 
\begin{equation}
G(S_{\tilde{\varphi}_{\omega}})\oplus_{\mathbb{R}} V_{\omega}G(F_{\tilde{\varphi}_{\omega}})=qH_{\omega}\oplus_{\mathbb{R}} qH_{\omega}.\label{eq: decomposition of qH_omega by S_tilde and F_tilde}
\end{equation}
 Let $E$ be the real orthogonal projection of $H_{\omega}\oplus_{\mathbb{R}} H_{\omega}$ onto $(G(S_{\varphi_n}))_{\omega}$, and let $F:=q\oplus q$. 
By Eq. (\ref{eq: H_omega plus H_omega decomposition}), $E^{\perp}$ is the real orthogonal projection onto $V_{\omega}G(F_{\tilde{\varphi}_{\omega}})$.  
By Eq. (\ref{eq: G(S_tilde)}), (\ref{eq: G(F_tilde)}) and (\ref{eq: decomposition of qH_omega by S_tilde and F_tilde}), we have 
\eqa{
\text{ran}(E\wedge F)&\supset G(S_{\tilde{\varphi}_{\omega}}),\\
\text{ran}(E^{\perp}\wedge F)&\supset V_{\omega}G(F_{\tilde{\varphi}_{\omega}}),\\
\text{ran}(F)&=qH_{\omega}\oplus_{\mathbb{R}} qH_{\omega}.
}
Let $P,Q$ be real orthogonal projections from $H_{\omega}\oplus_{\mathbb{R}}H_{\omega}$ onto $G(S_{\tilde{\varphi}_{\omega}})$ and $V_{\omega}G(F_{\tilde{\varphi}_{\omega}})$, respectively. 
Then $P\le E\wedge F,\ Q\le E^{\perp}\wedge F$. On the other hand, by Eq. (\ref{eq: decomposition of qH_omega by S_tilde and F_tilde}) we have
\[P+Q=F\ge E\wedge F+E^{\perp}F.\]
Therefore it follows that $P=E\wedge F,\ Q=E^{\perp}\wedge F$ and $E\wedge F+E^{\perp}\wedge F=F$.  
Therefore by Lemma \ref{lem: e wedge f}, $E$ and $F$ commute. 
Let $E_n$ be the real orthogonal projection of $H_n\oplus_{\mathbb{R}}H_n$ onto $G(S_{\varphi_n})\ (n\in \mathbb{N})$. 
Then $E$ is the ultraproduct of $(E_n)_n$, and by (the proof of) \cite[Corollary 6.11]{HW1}, we know that 
\[E_n=\mattwo{(1+\Delta_{\varphi_n})^{-1}}{J_{\varphi_n}(\Delta_{\varphi_n}^{\frac{1}{2}}+\Delta_{\varphi_n}^{-\frac{1}{2}})^{-1}}{J_{\varphi_n}(\Delta_{\varphi_n}^{\frac{1}{2}}+\Delta_{\varphi_n}^{-\frac{1}{2}})^{-1}}{(1+\Delta_{\varphi_n})^{-1}},\ \ \ \ \ \ (n\in \mathbb{N}).\]
Let $a_n:=(1+\Delta_{\varphi_n})^{-1},\ b_n:=J_{\varphi_n}(\Delta_{\varphi_n}^{\frac{1}{2}}+\Delta_{\varphi_n}^{-\frac{1}{2}})^{-1} \in \mathbb{B}(H_n)$, and let $a_{\omega}:=(a_n)_{\omega},\ b_n:=(b_n)_{\omega}\in \mathbb{B}(H_{\omega})$ ($b_n$, $b_{\omega}$ are regarded as real linear operators). Then it holds that 
\[E=\mattwo{a_{\omega}}{b_{\omega}}{b_{\omega}}{a_{\omega}}.\]
Since $E=(E_n)_{\omega}$ commutes with $F=\mattwo{q}{0}{0}{q}$, $a_{\omega}$ commutes with $q$ and $qH_{\omega}$ is $a_{\omega}$-invariant. Therefore we see that $EF$ is the projection of $H_{\omega}\oplus_{\mathbb{R}} H_{\omega}$ onto $G(S_{\tilde{\varphi}_{\omega}})$, which is of the following form:
\[EF=E\wedge F=\mattwo{(1+\Delta_{\tilde{\varphi}_{\omega}})^{-1}q}
{J_{\tilde{\varphi}_{\omega}}(\Delta_{\tilde{\varphi}_{\omega}}^{\frac{1}{2}}+\Delta_{\tilde{\varphi}_{\omega}}^{-\frac{1}{2}})^{-1}q}
{J_{\tilde{\varphi}_{\omega}}(\Delta_{\tilde{\varphi}_{\omega}}^{\frac{1}{2}}+\Delta_{\tilde{\varphi}_{\omega}}^{-\frac{1}{2}})^{-1}q}
{(1+\Delta_{\tilde{\varphi}_{\omega}})^{-1}q}.\]
This shows that 
\begin{equation}
a_{\omega}|_{qH_{\omega}}=(1+\Delta_{\tilde{\varphi}_{\omega}})^{-1}\label{eq: UP of resolvent and resolvent of UP}
\end{equation}

Now by Lemma \ref{Uffe: modular operator}, we have
\begin{equation}
(\Delta_{\varphi_n}^{it})^{\omega}|_{qH_{\omega}}=\Delta_{\tilde{\varphi}_{\omega}}^{it},\ \ \ \ \ t\in \mathbb{R}.\label{eq: UP of Delta^it}
\end{equation}
From this equality, we have that $(\sigma_t^{\varphi_n})^{\omega}=\sigma_t^{\varphi^{\omega}}$ for all $t\in \mathbb{R}$ because $\varphi^{\omega}$ corresponds to $\tilde{\varphi}_{\omega}$ under the identification $M^{\omega}\cong qNq$.       
\end{proof}
\begin{remark}
Note that the ultrapower $t\mapsto (\sigma_t^{\varphi_n})^{\omega}$ can be defined on $\ell^{\infty}/\mathcal{I}_{\omega}$ (quotient as a Banach space), and it is indeed quite discontinuous on it. We show an example. 
\end{remark}

\begin{example}Let $0<\lambda<1$ and $R_{\lambda}=\bigotimes_{n=1}^{\infty}(M_2(\mathbb{C}),\tau_{\lambda})$ be the Powers factor of type III$_{\lambda}$, where $\tau_{\lambda}=\text{Tr}(\rho_{\lambda}\cdot ),\ \rho_{\lambda}=\text{diag}(\frac{\lambda}{1+\lambda},\frac{1}{1+\lambda})$. The modular automorphism group of $\varphi =\bigotimes_{n=1}^{\infty} \tau_{\lambda}$ is given by $\bigotimes_{n=1}^{\infty}\sigma_t^{\tau_{\lambda}}$, where
\[\sigma_t^{\tau_{\lambda}}\left (\begin{pmatrix}a & b\\ c & d\end{pmatrix}\right )=\begin{pmatrix}a & \lambda^{it}b\\ \lambda^{-it}c & d\end{pmatrix},\ \ \ \  a,b,c,d\in \mathbb{C}, t\in \mathbb{R}.\]
Define 
\[x_n:=\begin{pmatrix}0 & 1\\ 1 & 0\end{pmatrix}^{\otimes n}\otimes 1\otimes 1\cdots \in R_{\lambda} ,\ \ \ \ n\in \mathbb{N}.\]
It is clear that $(x_n)\in \ell^{\infty}(\mathbb{N},R_{\lambda})$. We see that 
\eqa{
\|\sigma_t^{\varphi}(x_n)-x_n\|_{\varphi}^2&=\left \|\begin{pmatrix}0 & \lambda^{it}\\ \lambda^{-it} & 0\end{pmatrix}^{\otimes n}-\begin{pmatrix}0 & 1 \\ 1 & 0\end{pmatrix}^{\otimes n}\right \|_{\tau_{\lambda}^{\otimes n}}^2\\
&=2-\tau_{\lambda}^{\otimes n}\left [\begin{pmatrix}\lambda^{it} & 0\\ 0 & \lambda^{-it}\end{pmatrix}^{\otimes n}+\begin{pmatrix}\lambda^{-it} & 0 \\ 0 & \lambda^{it}\end{pmatrix}^{\otimes n}\right ]\\
&=2-\left (\frac{\lambda^{it+1}+\lambda^{-it}}{1+\lambda}\right )^{n}-\left (\frac{\lambda^{-it+1}+\lambda^{it}}{1+\lambda}\right )^{n}.
}
It follows that since 
\eqa{
\left |\frac{\lambda^{it+1}+\lambda^{-it}}{1+\lambda}\right |^2&=\left |\frac{\lambda+\lambda^{-2it}}{1+\lambda}\right |^2\\
&=\frac{\lambda^2+2\lambda \cos (2t\log \lambda)+1}{\lambda^2+2\lambda+1},
}
the second term tends to zero as $n\to \infty$ whenever $|t|$ is small but nonzero, say $0<|t|<\pi/(6|\log \lambda|)$. The same happens for the third term, and we see that $\lim_{n\to \omega}\|\sigma_t^{\varphi}(x_n)-x_n\|_{\varphi}=\sqrt{2}$ for small enough $|t|\neq 0$. This shows that 
\[\lim_{t\to 0}\lim_{n\to \omega}\|\sigma_t^{\varphi}(x_n)-x_n\|_{\varphi}=\sqrt{2}\neq 0.\]
\end{example}

We state few immediate useful consequences. Similar results to (2) in the next Corollary is obtained by Golodets  \cite{Golodets} by a different proof.
\begin{corollary}\label{cor: some useful conseuences of main theorem}Let $(M_n,\varphi_n)_n$ be a sequence of pairs of $\sigma$-finite von Neumann algebras and normal faithful states. 
Let $(x_n)_n\in \mathcal{M}^{\omega}(M_n,\varphi_n)$ and put $\varphi^{\omega}=(\varphi_n)^{\omega}$.  
\begin{list}{}{}
\item[{\rm{(1)}}] $\Delta_{\varphi^{\omega}}^{it}(x_n\xi_{\varphi_n})_{\omega}=(\Delta_{\varphi_n}^{it}x_n\xi_{\varphi_n})_{\omega}$ for all $t\in \mathbb{R}$.
\item[{\rm{(2)}}] $\Delta_{\varphi^{\omega}}^{\frac{1}{2}}(x_n\xi_{\varphi_n})_{\omega}=(\Delta_{\varphi_n}^{\frac{1}
{2}}x_n\xi_{\varphi_n})_{\omega}$.
\item[{\rm{(3)}}] If $M_n=M, \varphi_n=\varphi\ (n\in \mathbb{N})$ for a fixed $M$ and $\varphi$, then 
$\sigma(\Delta_{\varphi^{\omega}})=\sigma(\Delta_{\varphi})$.
\end{list}
\end{corollary}
\begin{proof}
By Theorem \ref{thm: ultrapower of modular automorphism}, we have 
\[\Delta_{\varphi^{\omega}}^{it}(x_n)^{\omega}\xi_{\varphi^{\omega}}=\sigma_t^{\varphi^{\omega}}((x_n)^{\omega})\xi_{\varphi^{\omega}}=(\sigma_t^{\varphi_n}(x_n))^{\omega}\xi_{\varphi^{\omega}}=(\Delta_{\varphi_n}^{it}x_n\xi_{\varphi_n})_{\omega}.\]
Therefore (1) follows. For (2), by the proof of Theorem \ref{thm: ultrapower of modular automorphism}, we have $G(S_{\varphi^{\omega}})=(G(S_{\varphi_n}))_{\omega}\cap (qH_{\omega}\oplus_{\mathbb{R}} qH_{\omega})$. This implies that
\[\Delta_{\varphi^{\omega}}^{\frac{1}{2}}(x_n\xi_{\varphi})_{\omega}=J_{\varphi^{\omega}}S_{\varphi^{\omega}}(x_n\xi_{\varphi})_{\omega}=J_{\omega}(S_{\varphi_n}x_n\xi_{\varphi})_{\omega}=(\Delta_{\varphi_n}^{\frac{1}{2}}x_n\xi_{\varphi_n})_{\omega}.\]
To prove (3), note that $\Delta_{\varphi^{\omega}}|_{L^2(M,\varphi)}=\Delta_{\varphi}$, so $\sigma(\Delta_{\varphi})\subset \sigma (\Delta_{\varphi^{\omega}})$ and $\sigma((1+\Delta_{\varphi})^{-1})\subset \sigma((1+\Delta_{\varphi^{\omega}})^{-1}))$. On the other hand, by Eq. (\ref{eq: UP of resolvent and resolvent of UP}), we have 
\[\sigma((1+\Delta_{\varphi^{\omega}})^{-1})\subset \sigma(((1+\Delta_{\varphi})^{-1})^{\omega})=\sigma((1+\Delta_{\varphi})^{-1}),\]
because $(1+\Delta_{\varphi})^{-1}$ is bounded and $\sigma(a^{\omega})=\sigma(a)$ holds for a bounded operator $a$. Therefore 
$\sigma((1+\Delta_{\varphi^{\omega}})^{-1})=\sigma((1+\Delta_{\varphi})^{-1})$, whence $\sigma(\Delta_{\varphi^{\omega}})=\sigma(\Delta_{\varphi})$ holds.
\end{proof}

Therefore $\Delta_{\varphi^{\omega}}$ behaves like the ultrapower of $\Delta_{\varphi}$. Let use remark a subtle difference between the ultrapower of bounded operators and $\Delta_{\varphi^{\omega}}$. It is easy to see that for a bounded self-adjoint operator $a$, $\sigma(a^{\omega})=\sigma_p(a^{\omega})$ holds. However, the analogous result for $\Delta_{\varphi^{\omega}}$ does not hold.
\begin{proposition}\label{prop: point spectrum is smaller than the spectrum for UP state}
Let $M$ be a type {\rm{II}}$_1$ factor. There exists $\varphi\in S_{\rm{nf}}(M)$ for which $\sigma_p(\Delta_{\varphi^{\omega}})\subsetneq \sigma(\Delta_{\varphi^{\omega}})\setminus \{0\}$ holds.
\end{proposition}
\begin{proof}
Let $\tau$ be the unique tracial state on $M$ and consider the standard representation of $M$. 
Let $h\in M_+$ be such that $\sigma(h)=[\frac{1}{2},2]$ and that the distribution measure $\mu_h$ corresponding to $h$ with respect to $\tau$ has absolutely continuous spectra in $[\frac{1}{2},1]$ and purely atomic spectra in $(1,2]$. Here, $\mu_h$ is determined by moments
\[\int_{\mathbb{R}}t^pd\mu_h(t)=\tau(h^p),\ \ (p\in \mathbb{N}).\]
Define $\varphi \in S_{\rm{nf}}(M)$ by $\varphi(x):=\tau(hx)/\tau(h), x\in M$. Then $\Delta_{\varphi}=h(JhJ)^{-1}$. 
Since $M$ is a factor, $\sum_{i=1}^nx_iy_i\mapsto \sum_{i=1}^nx_i\otimes y_i\ (x_i\in M, y_i\in M')$ induces a *-isomorphism between the *-algebra generated by $M$ and $M'$ and the algebraic tensor product $M\odot M'$. Therefore C$^*(M,M')\cong M\otimes_{\alpha}M'$ for a C$^*$-tensor norm $\|\cdot \|_{\alpha}$, and since C$^*(h)$, C$^*(JhJ)$ are abelian hence nuclear, we have ${\rm{C}}^*(h,JhJ)\cong {\rm{C}}^*(h)\otimes {\rm{C}}^*(JhJ)$. Consequently, it holds that (cf. Corollary \ref{cor: some useful conseuences of main theorem} (3))
\[\sigma(\Delta_{\varphi^{\omega}})=\sigma(\Delta_{\varphi})=\left \{\frac{s}{t}; s\in \sigma(h), t\in \sigma(JhJ)\right \}=\left [\textstyle \frac{1}{4},4\right ].\]
Let $\widetilde{h}$ be the image of $h$ under the canonical embedding $M\subset M^{\omega}$. Let $\mu_{\widetilde{h}}$ be the distribution measure of $\widetilde{h}$ with respect to $\tau^{\omega}$. Since $\tau^{\omega}(\widetilde{h}^p)=\tau(h^p)$ holds for all $p\in \mathbb{N}$ and both $h, \widetilde{h}$ are bounded, $\mu_h=\mu_{\widetilde{h}}$ holds. Now we show that  
\[\sigma_p(\Delta_{\varphi^{\omega}})\cap \left ([\textstyle \frac{1}{4},\frac{1}{2}]\cup [2,4]\right )=\emptyset.\]
Suppose there were $\lambda \in [2,4]\cap \sigma_p(\Delta_{\varphi^{\omega}})$. Then by Takesaki's result \cite{Takesaki3}, there exists $u\in M^{\omega}$ such that 
$u\varphi^{\omega}=\lambda \varphi^{\omega}u$ holds. By taking the polar decomposition, we may assume that $u$ is a partial isometry. Since $\varphi^{\omega}=\tau^{\omega}(\widetilde{h}\cdot )$, this implies that $u\widetilde{h}=\lambda \widetilde{h}u$. Moreover, as $u^*u$ and $uu^*$ belong to $(M^{\omega})_{\varphi^{\omega}}$, they commute with $\widetilde{h}$. It follows that 
\[u(\widetilde{h}u^*u)u^*=(\lambda \widetilde{h})uu^*.\]
This shows that both $K=u^*uL^2(M^{\omega},\tau^{\omega})$ and $L=uu^*L^2(M^{\omega},\tau^{\omega})$ are $\widetilde{h}$-invariant subspaces, and $u$ induces an isometry of $K$ onto $L$. In particular, $\widetilde{h}|_K$ and $(\lambda \widetilde{h})|_L$ are unitarily equivalent operators, whence $\sigma(\widetilde{h}|_K)=\sigma(\lambda \widetilde{h}|_L)$ holds. On the other hand, we know that 
\[\sigma(\widetilde{h}|_K)\subset \textstyle [\frac{1}{2},2],\ \sigma(\lambda \widetilde{h}|_L)\subset [\frac{\lambda}{2},2\lambda].\]
Since $\lambda \in [2,4]$, this shows that $\sigma(\widetilde{h}|_K)=\sigma(\lambda \widetilde{h}|_L)\subset [\frac{\lambda}{2},2]$
However, $\mu_{\widetilde{h}|_K}$ restricted to $[1,2]$ is discrete, while $\mu_{\lambda \widetilde{h}|_L}$ restricted to $[\frac{\lambda}{2},2]\subset [1,2]$ is absolutely continuous, a contradiction. Therefore $\sigma_p(\Delta_{\varphi^{\omega}})\cap [2,4]=\emptyset$. $\sigma_p(\Delta_{\varphi^{\omega}})\cap [\frac{1}{4},\frac{1}{2}]=\emptyset$ can be shown similarly. This proves that $\sigma_p(\Delta_{\varphi^{\omega}})\subsetneq \sigma(\Delta_{\varphi^{\omega}})\setminus \{0\}$. 
\end{proof}
\begin{remark}
Proposition \ref{prop: point spectrum is smaller than the spectrum for UP state} states in particular that for $0<\lambda\in \sigma(\Delta_{\varphi})\setminus \sigma_p(\Delta_{\varphi^{\omega}})$, there is no bounded sequence $(x_n)_n$ of $M$ with $\|x_n\xi_{\varphi}\|=1\ (n\in \mathbb{N})$ satisfying 
\begin{equation}
\lim_{n\to \infty}\|\Delta_{\varphi}^{\frac{1}{2}}x_n\xi_{\varphi}-\lambda^{\frac{1}{2}}x_n\xi_{\varphi}\|=0.\label{eq: control of norm bound}
\end{equation}
For if there were such sequence, Corollary \ref{cor: some useful conseuences of main theorem} would imply that $(x_n)_n$ defines a nonzero element $(x_n)^{\omega}\in M^{\omega}$ satisfying $\Delta_{\varphi^{\omega}}^{\frac{1}{2}}(x_n)^{\omega}\xi_{\varphi^{\omega}}=\lambda^{\frac{1}{2}}(x_n)^{\omega}\xi_{\varphi^{\omega}}$, whence $\lambda \in \sigma_p(\Delta_{\varphi^{\omega}})$.  
On the other hand, as $M(\sigma^{\varphi},[\log \lambda-\frac{1}{n},\log \lambda+\frac{1}{n}])\neq \{0\}$ for each $n\in \mathbb{N}$, there exists a (necessarily unbounded) sequence $(x_n)_n\subset M$ with $\|x_n\xi_{\varphi}\|=1\ (n\in \mathbb{N})$ satisfying Eq. (\ref{eq: control of norm bound}). 
\end{remark}

Next, we show that elements of $\mathcal{M}^{\omega}$ are characterized by the spectral condition for $(\sigma^{\varphi_n})_n$. 
\begin{proposition}\label{prop: characterization of the normalizer}
Let $(M_n,\varphi_n)_n$ be a sequence of $\sigma$-finite von Neumann algebras and normal faithful states. Then for $(x_n)_n\in \ell^{\infty}(\mathbb{N},M_n)$, the following conditions are equivalent.
\begin{itemize}
\item[{\rm{(1)}}] $(x_n)_n\in \mathcal{M}^{\omega}(M_n,\varphi_n)$. 
\item[{\rm{(2)}}] For every $\varepsilon>0$, there exists $a>0$ and $(y_n)_n\in \mathcal{M}^{\omega}(M_n,\varphi_n)$ such that 
\eqa{
&{\rm{(i)}}\ \lim_{n\to \omega}\|x_n-y_n\|_{\varphi}^{\sharp}<\varepsilon,\\
&{\rm{(ii)}}\ y_n \in M_n(\sigma^{\varphi_n}, [-a,a]),\ \ n\in \mathbb{N}.
}
In this case, $(y_n)_n$ can be chosen to satisfy $\|(y_n)^{\omega}\|\le \|(x_n)^{\omega}\|$.
\end{itemize}
\end{proposition}

We need preparations. Recall two summability kernels on $\mathbb{R}$.
\begin{definition} 
The {\it Fej\'er kernel} $F_a\colon \mathbb{R}\to \mathbb{R}\ (a>0)$ is defined by 
\[F_a(t):=\begin{cases}\dfrac{1-\cos (at)}{\pi at^2} & (t\neq 0)\\
\ \ \ a/2\pi & (t=0) \end{cases}.\]
Its Fourier transform is
\[\widehat{F_a}(\lambda)=\begin{cases}1-\dfrac{|\lambda|}{a} & (|\lambda|\le a)\\
 \ \ \ 0 & (|\lambda|>a)\end{cases}.\]
It holds that $0\le F_a$ and $\|F_a\|_1=\widehat{F}_a(0)=1$. Moreover, we have  
\[\lim_{a\to \infty}\int_{\mathbb{R}}F_a(s)\phi (s)ds=\phi (0),\ \ \ \lim_{a\to \infty}\|F_a*f-f\|_1=0,\]
for all continuous bounded function $\phi$ on $\mathbb{R}$ and $f\in L^1(\mathbb{R})$. 
The {\it de la Vall\'ee Poussin Kernel} $D_a\colon \mathbb{R}\to \mathbb{R}$ is given by
\[
D_a(t)=2F_{2a}(t)-F_a(t)=\begin{cases}
\dfrac{\cos (at)-\cos (2at)}{\pi at^2} & (t\neq 0)\\
\ \ \ \ \ 3a/2\pi & (t=0).
\end{cases}
\]
Its Fourier transform is
\[\widehat{D_a}(\lambda)=\begin{cases}
1 & (|\lambda|\le a)\\
2-\frac{|\lambda|}{a}  & (a\le |\lambda|\le 2a)\\
0 & (|\lambda|>2a)
\end{cases}.\]
\end{definition}
For the details of summability kernels, see e.g., \cite{Katznelson}. 
\begin{lemma}\label{lem: analiticity implies normalizer}
Let $(x_n)_n\in \ell^{\infty}(\mathbb{N},M_n)$. If there exists $a>0$ such that $x_n\in M_n(\sigma^{\varphi_n},[-a,a])$ holds for all $n\in \mathbb{N}$, then $(x_n)_n\in \mathcal{M}^{\omega}(M_n,\varphi_n)$.
\end{lemma}
\begin{proof} The following argument goes back to \cite{Haagerup2}.  
We show that the map $t\mapsto \sigma_t^{\varphi_n}(x_n)$ is extended to an entire analytic $M_n$-valued function satisfying 
\[\|\sigma_z^{\varphi_n}(x_n)\|\le C_{a,z}\|x_n\|,\ \ \ \ \ \ z\in \mathbb{C},\]
where $C_{a,z}$ is a constant depending only on $a,z$.\ 
Since $x_n\in M_n(\sigma^{\varphi_n},[-a,a])$ and the de la Vall\'ee Poussin kernel satisfies $\widehat{D}_a=1$ on $[-a,a]$, we have 
$x_n=\sigma_{D_{a}}^{\varphi_n}(x_n)$. Therefore for $t\in \mathbb{R}$, we have
\eqa{
\sigma_t^{\varphi_n}(x_n)&=\int_{\mathbb{R}}D_{a}(s)\sigma_{t+s}^{\varphi_n}(x_n)ds\\
&=\int_{\mathbb{R}}D_{a}(s-t)\sigma_s^{\varphi_n}(x_n)ds.
}
By the explicit form, $D_{a}=2F_{2a}-F_a$ has an analytic continuation to $\mathbb{C}$. By (the proof of) \cite[Lemma 4.2]{Haagerup2}, we have \[\int_{\mathbb{R}}|F_a(s+it)|ds\le e^{a|t|}\ \ \ \ (t\in \mathbb{R}).\]  
Therefore for $z\in \mathbb{C}$, $s\mapsto  D_a(s-z)$ is in $L^1(\mathbb{R})$, and 
 $t\mapsto \sigma_t^{\varphi_n}(x_n)$ has an $M_n$-valued analytic extension: 
\[\sigma_z^{\varphi_n}(x_n)=\int_{\mathbb{R}}D_{a}(s-z)\sigma_s^{\varphi_n}(x_n)ds,\ \ \ \ \ \ z\in \mathbb{C}.\]
Then we have
\[
\|\sigma_z^{\varphi_n}(x_n)\|\le \int_{\mathbb{R}}|D_a(s-z)\||\sigma_s^{\varphi_n}(x_n)\|ds\le C_{a,z}\|x_n\|,\]
where $C_{a,z}:=2e^{2a|\text{Im}(z)|}+e^{a|\text{Im}z|}$. 
Let $(y_n)_n\in \mathcal{I}_{\omega}$. It follows that
\eqa{
\|(x_ny_n)^*\|_{\varphi_n}&=\|y_n^*J_{\varphi_n}\Delta_{\varphi_n}^{\frac{1}{2}}x_n\xi_{\varphi_n}\|\\
&=\|J_{\varphi_n}y_n^*J_{\varphi_n}\sigma_{-i/2}^{\varphi_n}(x_n)\xi_{\varphi_n}\|\\
&\le \|\sigma_{-i/2}^{\varphi_n}(x_n)\|\cdot \|J_{\varphi_n}y_n^*J_{\varphi_n}\xi_{\varphi_n}\|\\
&\le C_{a,-i/2}\|x_n\|\cdot \|y_n\|_{\varphi_n}\\
&\stackrel{n\to \omega}{\to}0.
}
Similarly, $\|y_nx_n\|_{\varphi}\to 0\ (n\to \omega)$. Hence $(x_n)_n\in \mathcal{M}^{\omega}$. 
\end{proof}
\begin{lemma}\label{lem: integration does not change normalizer}
Let $f\in L^1(\mathbb{R})$, and $(x_n)_n\in \mathcal{M}^{\omega}(M_n,\varphi_n)$. Then $(\sigma_f^{\varphi_n}(x_n))_n\in \mathcal{M}^{\omega}(M_n,\varphi_n)$ and $\sigma_f^{\varphi^{\omega}}((x_n)^{\omega})=(\sigma_f^{\varphi_n}(x_n))^{\omega}$ holds. 
\end{lemma}
\begin{remark}
One might think that this is a direct consequence of $\sigma_t^{\varphi^{\omega}}((x_n)^{\omega})=(\sigma_t^{\varphi_n}(x_n))^{\omega}$ (Theorem \ref{thm: ultrapower of modular automorphism}). However, we must show that \[\int_{\mathbb{R}}f(t)(\sigma_t^{\varphi_n}(x_n))^{\omega}dt=\left (\int_{\mathbb{R}}f(t)\sigma_t^{\varphi_n}(x_n)dt\right )^{\omega},\]
 i.e., the order of integration and ultralimit can be changed. 
\end{remark}
\begin{proof}[Proof of Lemma \ref{lem: integration does not change normalizer}]
We first prove\\
\textbf{Claim}. $(\sigma_{f*F_a}^{\varphi_n}(x_n))_n\in \mathcal{M}^{\omega}(M_n,\varphi_n)$ and $\sigma_{f*F_a}^{\varphi^{\omega}}((x_n)^{\omega})=(\sigma_{f*F_a}^{\varphi_n}(x_n))^{\omega}$ holds.\\ 
Since $\text{supp}(\widehat{f*F_a})\subset \text{supp}(\widehat{F_a})=[-a,a]$, we have $\sigma_{f*F_a}^{\varphi_n}(x_n)\in M_n(\sigma^{\varphi_n},[-a,a])$ for all $n\in \mathbb{N}$. Therefore by Lemma \ref{lem: analiticity implies normalizer}, we have $(\sigma_{f*F_a}^{\varphi_n}(x_n))_n\in \mathcal{M}^{\omega}(M_n,\varphi_n)$. Next, consider a bounded continuous function $Q_a\colon (0,1)\to \mathbb{C}$ given by 
\[Q_a(t):=(\widehat{f*F_a})(\log (t^{-1}-1)),\ \ \ \ \ t\in (0,1).\]
Then we have 
\[Q_a((1+t)^{-1})=(\widehat{f*F_a})(\log t),\ \ \ \ t\in \mathbb{R}.\]
By Lemma \ref{Uffe: modular operator} and (the proof of) Theorem \ref{thm: ultrapower of modular automorphism}, we have 
\[Q_a((1+\Delta_{\varphi^{\omega}})^{-1})=(Q_a((1+\Delta_{\varphi_n})^{-1}))_{\omega}|_{qH_{\omega}}\]
It then follows that
\eqa{
\sigma_{f*F_a}^{\varphi^{\omega}}((x_n)^{\omega})\xi_{\varphi^{\omega}}&=\widehat{f*F_a}(\log \Delta_{\varphi^{\omega}})(x_n)^{\omega}\xi_{\varphi^{\omega}}=Q_a((1+\Delta_{\varphi^{\omega}})^{-1})(x_n)^{\omega}\xi_{\varphi^{\omega}}\\
&=(Q_a((1+\Delta_{\varphi_n})^{-1}))_{\omega}(x_n\xi_{\varphi_n})_{\omega}=(\widehat{f*F_a}(\log \Delta_{\varphi_n})x_n\xi_{\varphi_n})_{\omega}\\
&=(\sigma_{f*F_a}^{\varphi_n}(x_n)\xi_{\varphi_n})_{\omega}=(\sigma_{f*F_a}^{\varphi_n}(x_n))^{\omega}\xi_{\varphi^{\omega}}.
}
Since $\xi_{\varphi^{\omega}}$ is separating for $(M_n,\varphi_n)^{\omega}$, we have $\sigma_{f*F_a}^{\varphi^{\omega}}((x_n)^{\omega})=(\sigma_{f*F_a}^{\varphi_n}(x_n))^{\omega}$.\\ \\
Now we prove that $(\sigma_f^{\varphi_n}(x_n))_n\in \mathcal{M}^{\omega}(M_n,\varphi_n)$ and $\sigma_f^{\varphi^{\omega}}((x_n)^{\omega})=(\sigma_f^{\varphi_n}(x_n))^{\omega}$ holds.\\
Since $\|f*F_a-f\|_1\stackrel{a\to \infty}{\to}0$, we have 
\eqa{
\sup_{n\ge 1}\|\sigma_f^{\varphi_n}(x_n)-\sigma_{f*F_a}^{\varphi_n}(x_n)\|
&\le \sup_{n\ge 1}\int_{\mathbb{R}}|f(t)-(f*F_a)(t)|\cdot \|\sigma_t^{\varphi_n}(x_n)\|dt\\
&=\sup_{n\ge 1}\|x_n\|\cdot \|f-f*F_a\|_1\stackrel{a\to \infty}{\rightarrow}0.
}
By the Claim, $(\sigma_{f*F_a}^{\varphi_n}(x_n))_n\in \mathcal{M}^{\omega}(M_n,\varphi_n)$. Therefore as  $\mathcal{M}^{\omega}(M_n,\varphi_n)$ is norm-closed, we have $(\sigma_f^{\varphi_n}(x_n))_n\in \mathcal{M}^{\omega}(M_n,\varphi_n)$. Finally, suppose $\varepsilon>0$ is given. By similar arguments to above, there exists $a>0$ such that  
\[\|\sigma_{f*F_a}^{\varphi^{\omega}}((x_n)^{\omega})-\sigma_f^{\varphi^{\omega}}((x_n)^{\omega})\|<\varepsilon,\ \ \  \left \|\left (\sigma_{f*F_a}^{\varphi_n}(x_n)\right )^{\omega}-\left (\sigma_f^{\varphi_n}(x_n)\right )^{\omega}\right \|<\varepsilon.\]
Then by the Claim, we see that
\eqa{
\left \|\sigma_f^{\varphi^{\omega}}((x_n)^{\omega})-\left (\sigma_f^{\varphi_n}(x_n)\right )^{\omega}\right \|&\le \left \|\sigma_f^{\varphi^{\omega}}((x_n)^{\omega})-\sigma_{f*F_a}^{\varphi^{\omega}}((x_n)^{\omega})\right \|\\
&\hspace{0.5cm}+\left \| \left (\sigma_{f*F_a}^{\varphi_n}(x_n)\right )^{\omega}-\left (\sigma_f^{\varphi_n}(x_n)\right )^{\omega}\right \|\\
&<2\varepsilon.
}
Since $\varepsilon>0$ is arbitrary, the Lemma is proved.    
\end{proof}

\begin{proof}[Proof of Proposition \ref{prop: characterization of the normalizer}]
(1)$\Rightarrow $(2): Let $(x_n)_n\in \mathcal{M}^{\omega}(M_n,\varphi_n)$ and put $x:=(x_n)^{\omega}$. Also, define
\[x_a:=\sigma_{F_a}^{\varphi^{\omega}}(x)\in (M_n,\varphi_n)^{\omega}\ \ (a>0).\]
Then we have $\lim_{a\to \infty}\|x_a-x\|_{\varphi^{\omega}}^{\sharp}=0$. Indeed, since $\Phi\colon  t\mapsto \|x-\sigma_t^{\varphi^{\omega}}(x)\|_{\varphi^{\omega}}^{\sharp}$ is continuous and bounded, we have
\eqa{
\|x_a-x\|_{\varphi^{\omega}}^{\sharp}&=\left \| \int_{\mathbb{R}}F_a(t)(\sigma_t^{\varphi^{\omega}}(x)-x)dt\right \|_{\varphi^{\omega}}^{\sharp}\\
&\le \int_{\mathbb{R}}F_a(t)\|x-\sigma_t^{\varphi^{\omega}}(x)\|_{\varphi^{\omega}}^{\sharp}dt\\
&\stackrel{a\to \infty}{\rightarrow }\Phi (0)=0,
}
whence the claim follows.\ 
Therefore there exists $a>0$ such that 
$y:=\sigma_{F_a}^{\varphi^{\omega}}(x)$ satisfies $\|y-x\|_{\varphi^{\omega}}<\varepsilon$. We have $\|y\|\le \|F_a\|_1\|x\|=\|x\|$, and by Lemma \ref{lem: integration does not change normalizer}, $y=(y_n)^{\omega}$, where $y_n=\sigma_{F_a}^{\varphi_n}(x_n)\ (n\in \mathbb{N})$. Therefore $(y_n)_n$ satisfies all conditions  in (2). Note that we also have $\|y_n\|\le \|x_n\|\ (n\in \mathbb{N})$.\\
(2)$\Rightarrow$(1):  
Suppose $(x_n)_n\in \ell^{\infty}(\mathbb{N},M_n)$ satisfies the conditions in (2). Let $\varepsilon>0$. Then by Lemma \ref{lem: analiticity implies normalizer} and by assumption, there is $(x_n')_n\in \mathcal{M}^{\omega}$ such that $\lim_{n\to \omega}\|x_n-x_n'\|_{\varphi_n}^{\sharp}<\varepsilon$. Let $(y_n)_n\in \mathcal{I}_{\omega}$ with $\sup_{n\ge 1}\|y_n\|\le 1$. Then we see that
\eqa{
\lim_{n\to \omega}\|(x_ny_n)^*\|_{\varphi_n}&\le \lim_{n\to \omega}\left \{\|y_n^*\|\ \|x_n^*-(x_n')^*\|_{\varphi_n}+\|y_n^*(x_n')^*\|_{\varphi_n}\right \}\\
&\le \varepsilon.
}
Since $\varepsilon>0$ is arbitrary, we have $\lim_{n\to \omega}\|(x_ny_n)^*\|_{\varphi_n}=0$. Similarly, we also have $\lim_{n\to \omega}\|y_nx_n\|_{\varphi_n}=0$. This proves that $(x_n)_n\in \mathcal{M}^{\omega}$.
\end{proof}

\subsection{Strict Homogeneity of State Spaces}\label{subsec: Strict Homogeneity of State Spaces}
As an application of the Groh-Raynaud ultraproduct, we prove that it provides examples of von Neumann algebras for which all normal faithful states are unitarily equivalent. We also prove that this property is only possible for von Neumann algebras with non-separable preduals (besides $\mathbb{C}$).

\begin{definition}\label{UW3.1.1}
Let $M$ be a $\sigma$-finite von Neumann algebra. Then $S_{\text{nf}}(M)$ is said to be
\begin{itemize}
\item {\it homogeneous}, if for any $\varphi, \psi \in S_{\text{nf}}(M)$ and any $\varepsilon>0$, there is $u\in \mathcal{U}(M)$ such that $\|u\varphi u^*-\psi\|<\varepsilon;$
\item {\it strictly homogeneous}, if for any $\varphi, \psi \in S_{\text{nf}}(M)$ there is $u\in \mathcal{U}(M)$ such that $u\varphi u^*=\psi$.
\end{itemize}
\end{definition} 
The following is immediate from results by \cite{CS}, \cite{CHS}. 
\begin{theorem}\label{UW3.1.2}
Let $M$ be a $\sigma$-finite von Neumann algebra. The following are equivalent:
\begin{itemize}
\item[{\rm{(1)}}] $M$ is a factor of type {\rm I}$_1$ or type {\rm III}$_1$.
\item[{\rm{(2)}}] $S_{{\rm nf}}(M)$ is homogeneous.
\end{itemize}
\end{theorem}
\begin{proof}
$(1)\Rightarrow (2)$ is the main result of \cite{CS}, while $(2)\Rightarrow (1)$ follows from \cite{CHS} (note that homogeneity of $S_{\rm{nf}}(M)$ implies that $M$ is a factor. Note also that the separability of $M_*$ imposed on \cite{CS} can be removed. See Remark \ref{rem: no separability assumption is necessary}).
\end{proof}
\begin{lemma}\label{lem: strict equality of states}
Let $M$ be a $\sigma$-finite factor not isomorphic to $\mathbb{C}$ with strictly homogeneous state space. Then 
\begin{itemize}
\item[{\rm{(1)}}] $M$ is a type {\rm{III}}$_1$ factor. 
\item[{\rm{(2)}}] For $\varphi, \psi \in S_{\rm{n}}(M)$, there exists a partial isometry $u\in M$ such that 
\[u^*u={\rm{supp}}(\varphi),\ uu^*={\rm{supp}}(\psi),\text{ and }\psi=u\varphi u^*.\]
\end{itemize}
\end{lemma}
\begin{proof}
(1) We have to show that $M$ has state space diameter 0. But since $S_{\rm{nf}}(M)$ is norm-dense in $S_{\rm{n}}(M)$, this is the consequence of the strict homogeneity of $S_{\rm{nf}}(M)$.\\
(2)  By (1), $M$ is a type III factor. Hence there is a partial isometry $v\in M$ such that $v^*v=\text{supp}(\varphi),\ vv^*=\text{supp}(\psi)$ holds. Put $\psi':=v^*\psi v$. We see that $\text{supp}(\psi')=v^*\text{supp}(\psi)v=\text{supp}(\varphi)$. Since $M$ is of type III, $M_{\text{supp}(\varphi)}\cong M$ has strictly homogeneous state space. Therefore regarding $\varphi, \psi'\in S_{\rm{nf}}(M_{\text{supp}(\varphi)})$ we may find $w\in M_{\text{supp}(\varphi)}$ with $w^*w=ww^*=\text{supp}(\varphi)$ such that $\psi'=w\varphi w^*$. Then $u:=vw$ satisfies
\eqa{
u^*u&=w^*\text{supp}(\varphi)w=\text{supp}(\varphi),\\
uu^*&=v\ \text{supp}(\varphi)v^*=vv^*=\text{supp}(\psi),\\
u\varphi u^*&=vw\varphi w^*v^*=v\psi' v^*=\psi.
}  
\end{proof}

By the homogeneity, the Ocneanu ultraproduct of a type III$_1$ factor does not depend on the choice of a sequence of normal faithful states.
\begin{corollary}\label{UW3.1.3}
Let $M$ be a $\sigma$-finite factor of type {\rm III}$_1$ and $(\psi_n)_n\subset S_{\rm{nf}}(M)$. Then $(M,\psi_n)^{\omega}\cong M^{\omega}$. 
\end{corollary}
\begin{proof}
Let $\psi\in S_{\text{nf}}(M)$ and choose (by Connes-St\o rmer transitivity, cf. Theorem \ref{UW3.1.2}) a sequence $(u_n)_n$ of unitaries in $M$ such that 
\[\|\psi-u_n\psi_nu_n^*\|\le \frac{1}{n},\ \ \ n\in \mathbb{N}.\]
Then $\psi_{\omega}=(u_n\psi_nu_n^*)_{\omega}$ and so $M^{\omega}=(M,\psi)^{\omega}\cong (M,\psi_n)^{\omega}$ by Theorem \ref{UW2.1.3} (cf. also the remark after Proposition \ref{UW1.2.2}).  
\end{proof}
\begin{theorem}\label{UW3.1.4}
Let $M$ be a $\sigma$-finite factor of type {\rm III}$_1$, let $M_n=M\ (n\in \mathbb{N})$, and let $N=\prod^{\omega}M_n$. Then $N$ is not $\sigma$-finite, but for any $\sigma$-finite projection $p\in N$, one has that $pNp$ has strictly homogeneous state space. In particular, $N$ and $M^{\omega}$ are factors of type {\rm III}$_1$.
\end{theorem}
Another proof that $N$ is a type III$_1$ factor will be given in $\S$\ref{subsec: UP of III_lambda}.
\begin{proof}[Proof of Theorem \ref{UW3.1.4}]
Let $\varphi, \psi$ be normal states in $N$. By Corollary \ref{UW2.3.4}, there are sequences of normal states $(\varphi_n)_n, (\psi_n)_n\subset M_*$ such that $\varphi=(\varphi_n)_{\omega}$ and $\psi=(\psi_n)_{\omega}$. By Theorem \ref{UW3.1.2}, there is $(u_n)_n\subset \mathcal{U}(M)$ such that $\|u_n\varphi_nu_n^*-\psi_n\|<\frac{1}{n}$ for all $n\in \mathbb{N}$. Now, let $u:=(u_n)_{\omega}\in \mathcal{U}(N)$. Then $u\varphi u^*=\psi$. Hence all normal states of $N$ are unitarily equivalent; in particular, $N$ is not $\sigma$-finite (there can be no faithful normal states in this situation).

If $p\in N$ is a $\sigma$-finite projection, let $\varphi, \psi$ be normal faithful states on $pNp$. Then $\tilde{\varphi}:=p\varphi p$ and $\tilde{\psi}:=p\psi p$ define normal states on $N$ with support $p$. By the above, we may choose $u\in \mathcal{U}(N)$ such that $u\tilde{\varphi}u^*=\tilde{\psi}$. Then $upu^*=p$ and hence $v:=up\in \mathcal{U}(pNp)$. Also $v\varphi v^*=u\tilde{\varphi}u^*=\tilde{\psi}=\psi$ on elements of $pNp$. Hence $S_{\text{nf}}(pNp)$ is strictly homogeneous. 
\end{proof}
\begin{remark}
The above theorem implies that  for a $\sigma$-finite type III$_1$ factor $M$, every $\alpha \in \text{Aut}(M^{\omega})$ is pointwise inner in the sense of \cite{HS2}. 
\end{remark}

We remark that no von Neumann algebra with separable predual has strictly homogeneous state space:

\begin{proposition}\label{prop: no separable M with strictly homogeneous state space} Let $M$ be a $\sigma$-finite factor not isomorphic to $\mathbb{C}$ with strictly homogeneous state space. Then $M_*$ is not separable.
\end{proposition}
\begin{lemma}\label{lem: projection in centralizer} Let $M$ be a $\sigma$-finite factor not isomorphic to $\mathbb{C}$ with strictly homogeneous state space, and let $\varphi \in S_{\rm{nf}}(M)$. Then for any $\lambda \in (0,1)$, there is a projection $p\in M_{\varphi}$ such that $\varphi(p)=\lambda$ holds. 
\end{lemma}
\begin{proof}
Put $\widetilde{M}:=M\otimes M_2(\mathbb{C})$, $\theta:=\mattwo{\lambda \varphi}{0}{0}{(1-\lambda)\varphi}$, and $q:=\mattwo{1}{0}{0}{0}$.\\
It holds that $q\in \widetilde{M}_{\theta}$, and $\theta(q)=\lambda$. Since $M$ is of type III, there is an isomorphism $\Phi\colon \widetilde{M}\stackrel{\sim}{\to} M$. Then $\psi:=\theta \circ \Phi^{-1}, p':=\Phi(q)$ satisfies $p'\in M_{\psi}$ and $\psi(p')=\lambda$. Choose, by strict homogeneity of $S_{\rm{nf}}(M)$, $u\in \mathcal{U}(M)$ such that $u\psi u^*=\varphi$. Then $p:=up'u^*$ works. 
\end{proof}
\begin{proof}[Proof of Proposition \ref{prop: no separable M with strictly homogeneous state space}]
Choose $0<\lambda<1$. By Lemma \ref{lem: projection in centralizer}, there is a projection $p\in M_{\varphi}$ such that $\varphi (p)=\lambda$. 
Put $\psi:=\frac{1}{\lambda}p\varphi$. By Lemma \ref{lem: strict equality of states} (2), there is a partial isometry $v\in M$ such that $v^*v=\text{supp}(\varphi)=1$, $vv^*=\text{supp}(\psi)=p$, and $\psi=v\varphi v^*$. We see that
\eqa{
v\varphi &=(v\varphi v^*)v=\psi v=\psi (v\cdot )\\
&=\frac{1}{\lambda}p\varphi (v\cdot )=\frac{1}{\lambda}\varphi (v\cdot p)=\frac{1}{\lambda}\varphi (pv\cdot )\\
&=\frac{1}{\lambda}\varphi v.
}
Therefore by \cite[Lemma 1.6]{Takesaki3}, $\sigma_t^{\varphi}(v)=\lambda^{it}v$ holds for all $t\in \mathbb{R}$, which is equivalent to $\lambda \in \sigma_p(\Delta_{\varphi})$. 
Since $\lambda \in (0,1)$ is arbitrary, $\Delta_{\varphi}$ has uncountably many eigenvalues. This shows that $L^2(M,\varphi)$ is not separable, whence $M_*$ is not separable. 
\end{proof}
\begin{proposition}\label{prop: centralizer of a factor with strictly homogeneous state space}Let $M$ be a $\sigma$-finite factor not isomorphic to $\mathbb{C}$ with strictly homogeneous state space. Then for any $\varphi \in S_{\rm{nf}}(M)$, $M_{\varphi}$ is a factor of type {\rm{II}}$_1$.
\end{proposition}
\begin{proof}It is clear that $M_{\varphi}$ is a finite von Neumann algebra. If $M_{\varphi}$ were not a factor, choose  a projection $p\in \mathcal{Z}(M_{\varphi})\setminus \{0,1\}$. We may assume that $0<s:=\varphi(p)\le \frac{1}{2}$. Then $\varphi(p^{\perp})=1-s\ge s=\varphi(p)$. Since $M$ is of type III, $(1-p)M(1-p)\cong M$. Hence by Lemma \ref{lem: projection in centralizer} applied to $\frac{1}{1-s}\varphi|_{M_{(1-p)}}$, there is a projection $q\in M_{\varphi}$ such that $q\le 1-p$ and $\varphi(q)=s$. Since $\frac{1}{s}p\varphi$ and $\frac{1}{s}q\varphi$ are normal states on $M$ with support $p$ and $q$, respectively. By Lemma \ref{lem: strict equality of states}, there is a partial isometry $v\in M$ such that $v^*v=p, vv^*=q$ and $v(p\varphi)v^*=q\varphi$.  Since $p,q\in M_{\varphi}$, we have
\eqa{
\varphi v&=\varphi (v\cdot )=\varphi (qv\cdot )=\varphi (v\cdot q)\\
&=q\varphi v=(vp\varphi v^*)v=vp\varphi p=vp\varphi\\
&=v\varphi,
}
whence $v\in M_{\varphi}$.  This shows that $p\sim q$ in $M_{\varphi}$.  However, as $q\le 1-p$, we know that $z_{M_{\varphi}}(q)\perp z_{M_{\varphi}}(p)=p$. Therefore $p\sim q$ in $M_{\varphi}$ cannot be the case. This shows that $M_{\varphi}$ is a factor. Then by Lemma \ref{lem: projection in centralizer}, $M_{\varphi}$ is a II$_1$ factor. 
\end{proof}
\subsection{Ultrapower of a Normal Faithful Semifinite Weight}\label{subsec: Ultarpower of a Normal Faithful Semifinite Weight}
Let $\mathcal{W}_{\text{nfs}}(M)$ be the set of all normal faithful semifinite weights on a $\sigma$-finite von Neumann algebra $M$, and let $E\colon M^{\omega}\ni (x_n)^{\omega}\mapsto \text{wot-}\lim_{n\to \omega}x_n\in M$ be the canonical normal faithful conditional expectation \cite[$\S$5]{Ocneanu}. 
\begin{definition}
We define $\varphi^{\omega}\in \mathcal{W}_{\text{nfs}}(M^{\omega})$ by 
\[\varphi^{\omega}:=\varphi\circ E,\ \ \ \varphi \in \mathcal{W}_{\text{nfs}}(M).\]
\end{definition}
Since both $\varphi$ and $E$ are normal and  faithful, and since $\varphi$ is semifinite, $\varphi^{\omega}\in \mathcal{W}_{\text{nfs}}(M^{\omega})$ holds. Note that this definition is in agreement with the definition of the ultrapower state $\varphi^{\omega}$ when $\varphi\in S_{\text{nf}}(M)$. We then have a following partial generalization of Theorem \ref{thm: ultrapower of modular automorphism}. 
\begin{lemma}\label{lem: modular automorphism of the ultrapower of weight}
Let $M$ be a $\sigma$-finite von Neumann algebra, and let $\varphi \in \mathcal{W}_{\rm nfs}(M)$. Then we have 
\[\sigma_t^{\varphi^{\omega}}((x_n)^{\omega})=(\sigma_t^{\varphi}(x_n))^{\omega},\ \ \ \ (x_n)^{\omega}\in M^{\omega},\ \ t\in \mathbb{R}.\]
\end{lemma}
\begin{proof}
Let $\psi\in S_{\text{nf}}(M)$, and let $u_t:=({\rm D}\varphi^{\omega}\colon {\rm D}\psi^{\omega})_t,\ \ (t\in \mathbb{R})$. Since $\varphi^{\omega}=\varphi\circ E$ and $\psi^{\omega}=\psi\circ E$, by Theorem \ref{thm: ultrapower of modular automorphism},  we have for $x=(x_n)^{\omega}\in M^{\omega}$ and $t\in \mathbb{R}$ that 
\eqa{
\sigma_t^{\varphi^{\omega}}((x_n)^{\omega})&=u_t\sigma_t^{\psi^{\omega}}((x_n)^{\omega})u_t^*\\
&=({\rm D}(\varphi\circ E):{\rm D}(\psi\circ E))_t(\sigma_t^{\psi}(x_n))^{\omega}({\rm D}(\varphi\circ E):{\rm D}(\psi\circ E))_t^*\\
&=(({\rm D}\varphi:{\rm D}\psi)_t\sigma_t^{\psi}(x_n)({\rm D}\varphi:{\rm D}\psi)_t^*)^{\omega}\\
&=(\sigma_t^{\varphi}(x_n))^{\omega}.
} 
This proves the lemma.
\end{proof}

 Recall that a normal faithful semifinite weight $\varphi$ on a von Neumann algebra $M$ is called {\it lacunary} if $1$ is isolated in $\sigma(\Delta_{\varphi})$. The next result will be important for the analysis of the Ocneanu ultraproduct of type III$_0$ factors. 
\begin{proposition}\label{prop: centralizer of lacunary weight}
Let $M$ be a $\sigma$-finite von Neumann algebra, and let $\varphi \in \mathcal{W}_{\rm nfs}(M)$ be lacunary. Then $(M^{\omega})_{\varphi^{\omega}}\cong (M_{\varphi})^{\omega}$ holds.  
\end{proposition}
\begin{proof}
We first prove that $(M_{\varphi})^{\omega}\subset (M^{\omega})_{\varphi^{\omega}}$. Since $\varphi$ is lacunary, it is strictly semifinite \cite{HS2} and therefore there exists a normal faithful $\varphi$-preserving conditional expectation $E\colon M\to M_{\varphi}$. Therefore we may regard $(M_{\varphi})^{\omega}\subset M^{\omega}$. Let $x=(x_n)^{\omega}\in (M_{\varphi})^{\omega}$. Then by Lemma \ref{lem: modular automorphism of the ultrapower of weight}, we have $\sigma_t^{\varphi^{\omega}}(x)=(\sigma_t^{\varphi}(x_n))^{\omega}=x\ (t\in \mathbb{R})$, whence $x\in (M^{\omega})_{\varphi^{\omega}}$ holds. Let $0<\lambda<1$ be such that $\sigma(\Delta_{\varphi})\cap (\lambda, \lambda^{-1})=\{1\}$.\\ \\
\textbf{Step 1.} We next prove $(M^{\omega})_{\varphi^{\omega}}\subset (M_{\varphi})^{\omega}$ for the case where $\varphi(1)<\infty$.  
Let $x=(x_n)^{\omega}\in (M^{\omega})_{\varphi^{\omega}}$ with $\|x\xi_{\varphi^{\omega}}\|=1$. Then by Corollary \ref{cor: some useful conseuences of main theorem} (2), we have  
\[\Delta_{\varphi^{\omega}}^{\frac{1}{2}}x\xi_{\varphi^{\omega}}=x\xi_{\varphi^{\omega}}\Leftrightarrow \lim_{n\to \omega}\|\Delta_{\varphi}^{\frac{1}{2}}x_n\xi_{\varphi}-x_n\xi_{\varphi}\|=0.\]
Let $p:=1_{\{1\}}(\Delta_{\varphi})$ be the spectral projection of $\Delta_{\varphi}$ corresponding to the eigenvalue 1. Then by assumption, we have
\eqa{
\|\Delta_{\varphi}^{\frac{1}{2}}x_n\xi_{\varphi}-x_n\xi_{\varphi}\|&=\|\Delta_{\varphi}^{\frac{1}{2}}p^{\perp}x_n\xi_{\varphi}-p^{\perp}x_n\xi_{\varphi}\|\\
&\ge \text{min}(1-\lambda^{\frac{1}{2}}, \lambda^{-\frac{1}{2}}-1)\|p^{\perp}x_n\xi_{\varphi}\|\\
&=(1-\lambda^{\frac{1}{2}})\|p^{\perp}x_n\xi_{\varphi}\|.
}
Therefore we have
\[\lim_{n\to \omega}\|x_n\xi_{\varphi}-px_n\xi_{\varphi}\|=0.\]
Let $f\in L^1(\mathbb{R})_+$ be such that $\text{supp}(\hat{f})\subset (\log \lambda,-\log \lambda)$ and $\int_{\mathbb{R}}f(t)dt=1$. Let 
\[y_n:=\sigma_f^{\varphi}(x_n)=\int_{\mathbb{R}}f(t)\sigma_t^{\varphi}(x_n)dt,\ \ \ \ n\ge 1.\]
Since
\[\sigma_f^{\varphi}(x_n)\xi_{\varphi}=\hat{f}(\log \Delta_{\varphi})x_n\xi_{\varphi},\]
we have $\text{Sp}_{\sigma^{\varphi}}(y_n)\subset \text{Sp}_{\sigma^{\varphi}}(x_n)\cap (\log \lambda,-\log \lambda)=\{0\}$ and $y_n\in M_{\varphi}$. It is clear that $\sup_{n\ge 1}\|y_n\|\le \|f\|_1\sup_{n\ge 1}\|x_n\|<\infty$.
We have
\[px_n\xi_{\varphi}=\hat{f}(\log \Delta_{\varphi})x_n\xi_{\varphi}=y_n\xi_{\varphi},\ \ \ \ \ n\ge 1.\]
This implies that $\|x_n\xi_{\varphi}-y_n\xi_{\varphi}\|\to (n\to \omega)$. Since $\Delta_{\varphi^{\omega}}x^*\xi_{\varphi^{\omega}}=x^*\xi_{\varphi^{\omega}}$ also holds, we have also $\|x_n^*\xi_{\varphi}-y_n^*\xi_{\varphi}\|\to 0$. Since $M_{\varphi}$ is a finite von Neumann algebra, $(y_n)_n$ defines an element in $(M_{\varphi})^{\omega}$, and $x=(y_n)^{\omega}$ holds. Therefore $(M^{\omega})_{\varphi^{\omega}}\subset (M_{\varphi})^{\omega}$.\\ \\
\textbf{Step 2.} Finally, we prove $(M^{\omega})_{\varphi^{\omega}}\subset (M_{\varphi})^{\omega}$ for a general lacunary $\varphi \in \mathcal{W}_{\rm{nfs}}(M)$. Take $\lambda>0$ as in Step 1. 
Since the restriction of $\varphi$ to $M_{\varphi}$ is a semifinite trace, there exists an increasing net $\{p_i\}_{i\in I}$ of projections in $M_{\varphi}$ such that $\{p_i\}_{i\in I}$ converges strongly to 1, and $\varphi(p_i)<\infty$ for all $i\in I$. 
Let $x\in (M^{\omega})_{\varphi^{\omega}}$. Fix arbitrary $i\in I$. Identifying $p_iM^{\omega}p_i$ with $(p_iMp_i)^{\omega}$ \cite[Proposition 2.10]{MasudaTomatsu}, we may regard $p_ixp_i\in (p_iMp_i)^{\omega}$. Furthermore, as $p_i\in M_{\varphi}$ and $\varphi^{\omega}(p_i)=\varphi(p_i)<\infty$, the restriction $\varphi_{p_i}^{\omega}$ of $\varphi^{\omega}$ to $(p_iMp_i)^{\omega}$ is a normal faithful positive linear functional, and $p_iM^{\omega}p_i\cap (M^{\omega})_{\varphi^{\omega}}=((p_iMp_i)^{\omega})_{\varphi^{\omega}_{p_i}}$. It also holds that $\varphi_{p_i}^{\omega}$ is the ultrapower of $\varphi_{p_i}$. Since $\Delta_{\varphi_{p_i}}=\Delta_{\varphi}|_{p_iJ_{\varphi}p_iJ_{\varphi}L^2(M,\varphi)}$, we have $\sigma(\Delta_{\varphi_{p_i}})\cap (\lambda,\lambda^{-1})\subset \sigma(\Delta_{\varphi})\cap (\lambda,\lambda^{-1})=\{1\}$, and hence $\varphi_{p_i}$ is lacunary on $p_iMp_i$. Therefore by Step 1, we have $((p_iMp_i)^{\omega})_{\varphi_{p_i}^{\omega}}=((p_iMp_i)_{\varphi_{p_i}})^{\omega}$ holds. Therefore $p_ixp_i\in ((p_iMp_i)_{\varphi_{p_i}})^{\omega}\subset (M_{\varphi})^{\omega}$. Since $i\in I$ is arbitrary, and $p_ixp_i\to x$ strongly, we have that $x\in (M_{\varphi})^{\omega}$. Therefore $(M^{\omega})_{\varphi^{\omega}}\subset (M_{\varphi})^{\omega}$.     
\end{proof}
\begin{remark}
If $\varphi$ is not lacunary, then $(M_{\varphi})^{\omega}$ can be a proper subalgebra of $(M^{\omega})_{\varphi^{\omega}}$. 
In fact, let $\varphi$ be a normal faithful state on Araki-Woods type III$_1$ factor $R_{\infty}$ with $(R_{\infty})_{\varphi}=\mathbb{C}$. Existence of such $\varphi$ was shown by Herman and Takesaki \cite{HermanTakesaki}. By Theorem \ref{UW3.1.4}, $R_{\infty}^{\omega}$ has strictly homogeneous state space. By Proposition \ref{prop: centralizer of a factor with strictly homogeneous state space}, $(R_{\infty}^{\omega})_{\varphi^{\omega}}$ is a type II$_1$ factor. Therefore 
\[(R_{\infty}^{\omega})_{\varphi^{\omega}}\supsetneq ((R_{\infty})_{\varphi})^{\omega}=\mathbb{C}.\]
\end{remark} 

\subsection{The Golodets State $\dot{\varphi}_{\omega}$ and Tensorial Absorption of Powers Factors}\label{subsec: Golodets state}
In this section, we reinterpret the main result of Golodets' work \cite[Theorem 2.5.2]{Golodets} on the asymptotic algebra from our viewpoint. Following notations in $\S$\ref{subsec: Golodets' asymptotic algebra}, let $M$ be a factor with separable predual, and consider the asymptotic algebra $C_M^{\omega}$ induced by $\varphi \in S_{\rm{nf}}(M)$. $\varphi$ naturally induces a normal faithful state $\tilde{\varphi}=\overline{\varphi}|_{\mathscr{R}}$ on $\mathscr{R}=e_{\omega}\pi_{\rm{Gol}}(\ell^{\infty})''e_{\omega}$, hence a normal faithful state $\dot{\varphi}=\tilde{\varphi}|_{C_M^{\omega}}$. The main results of Golodets' work in \cite{Golodets} were 
\begin{itemize}
\item[(1)] to generalize the central sequence algebra $M'\cap M^{\omega}$ for type III factors and give a characterization of Araki's Property $L_{\lambda}'\ (0<\lambda<1)\colon M\cong M\overline{\otimes}R_{\lambda}$ \cite{Araki2} if and only if $\lambda$ is the eigenvalue of $\Delta_{\dot{\varphi}}$. 
\item[(2)] to show that the centralizer $(C_M^{\omega})_{\dot{\varphi}}$ plays the similar role as Connes' asymptotic centralizer $M_{\omega}$ (see Definition \ref{def: asymptotic centralizer} below), namely $M$ is McDuff if and only if $(C_M^{\omega})_{\dot{\varphi}}$ is noncommutative.   
\end{itemize}
Regarding (2), Golodets and Nessonov \cite[Lemma 1.1]{GoNe} later showed that $(C_M^{\omega})_{\dot{\varphi}}$ is indeed isomorphic to $M_{\omega}$ for a factor $M$ with separable predual.  

We show that these results can be naturally interpreted by the Ocneanu's setting. 
We start from the following observation: 
\begin{proposition}\label{prop: phi_dot does not depend on phi}
Let $M$ be a $\sigma$-finite von Neumann algebra, and let $\varphi, \psi \in S_{\rm{nf}}(M)$ such that $\varphi|_{\mathcal{Z}(M)}=\psi|_{\mathcal{Z}(M)}$. Then $\dot{\varphi}_{\omega}=\dot{\psi}_{\omega}$, where $\dot{\varphi}_{\omega}:=\varphi^{\omega}|_{M'\cap M^{\omega}}, \dot{\psi}_{\omega}:=\psi^{\omega}|_{M'\cap M^{\omega}}$. In particular, if $M$ is a $\sigma$-finite factor, then $\dot{\varphi}_{\omega}$ does not depend on $\varphi$. 
\end{proposition}
\begin{remark}
We thank Yoshimichi Ueda for the discussion which improved Proposition \ref{prop: phi_dot does not depend on phi} to the current form.
\end{remark}
\begin{proof}[Proof of Proposition \ref{prop: phi_dot does not depend on phi}]
Recall that $M^{\omega}\ni (x_n)^{\omega}\mapsto {\rm{wot-}}\lim_{n\to \omega}x_n\in M$ defines a normal faithful conditional expectation $E$. It is easy to see that $E((x_n)^{\omega})\in \mathcal{Z}(M)$ if $(x_n)^{\omega}\in M'\cap M^{\omega}$. Since $\varphi^{\omega}=\varphi\circ E$, $\psi^{\omega}=\psi \circ E$, and since $\varphi$ and $\psi$ agree on $\mathcal{Z}(M)$, we have
\[
\dot{\varphi}_{\omega}=\varphi \circ E|_{M'\cap M^{\omega}}=\psi \circ E|_{M'\cap M^{\omega}}=\dot{\psi}_{\omega}. 
\]
\end{proof}
\begin{definition}\label{def: Golodets state}
Let $M$ be a $\sigma$-finite von Neumann algebra, and let $\varphi\in S_{\rm{nf}}(M)$.  We call $\dot{\varphi}_{\omega}=\varphi^{\omega}|_{M'\cap M^{\omega}}$ the {\it Golodets state} associated with $\varphi$.
\end{definition}

Next theorem corresponds to Golodets' work (1) mentioned above. 
\begin{theorem}[Golodets]\label{thm: point spectrum and property L'}
Let $M$ be a $\sigma$-finite factor of type {\rm{III}}. Then $M\cong M\overline{\otimes}R_{\lambda}$ holds if and only if $\lambda \in \sigma_p(\Delta_{\dot{\varphi}_{\omega}})$ for some (hence any) $\varphi\in S_{{\rm{nf}}}(M)$.
\end{theorem}
To prove the theorem we use the following characterization of the condition $M\cong M\overline{\otimes}R_{\lambda}$. 
\begin{lemma}{\rm{\cite[Theorem XVIII.4.1]{TakBook}}}\label{lem:  ARaki's property L'} Let $0<\lambda<1$ and let $M$ be a $\sigma$-finite factor of type {\rm{III}}. The following conditions are equivalent. 
\begin{list}{}{}
\item[{\rm{(1)}}] $M\cong M\overline{\otimes}R_{\lambda}$.
\item[{\rm{(2)}}] For any $n\in \mathbb{N}$, $\varepsilon>0$ and $\varphi_1,\cdots ,\varphi_n\in S_{{\rm{nf}}}(M)$, there exists nonzero $x\in M$ such that 
\[\|(\Delta_{\varphi_j}^{\frac{1}{2}}-\lambda^{\frac{1}{2}})x\xi_{\varphi_j}\|^2\le \varepsilon \sum_{j=1}^n\varphi_j(x^*x).\]
\end{list}
\end{lemma}

\begin{proof}[Proof of Theorem \ref{thm: point spectrum and property L'}]
(1) Assume $\lambda \in \sigma_p(\Delta_{\dot{\varphi}_{\omega}})$, and suppose $\varepsilon>0$, $n\in \mathbb{N}$ and $\varphi_1,\cdots \varphi_n\in S_{\text{nf}}(M)$ are given. Define $\psi:=\sum_{i=1}^n\varphi_i\in M_*^+$. 
 By assumption, there exists $y\in M'\cap M^{\omega}$ with $\|y\|_{\dot{\psi}}=1$ satisfying 
\[\sigma_t^{\dot{\varphi}_{\omega}}(y)=\lambda^{it}y,\ \ \ t\in \mathbb{R}.\]
Take a representative $(y_n)_n$ of $y$. By Proposition \ref{prop: phi_dot does not depend on phi}, $\dot{\varphi_j}_{\omega}=\dot{\psi}_{\omega}(=\dot{\varphi}_{\omega})$ holds for $j=1,2,\cdots ,n$. 
Note that since $\|[\sigma_t^{\varphi}(x),y]\|_{\varphi}^{\sharp}=\|[x,\sigma_{-t}^{\varphi}(y)]\|_{\varphi}^{\sharp}\ (x,y\in M)$, $M'\cap M^{\omega}$ is $\sigma_t^{\varphi^{\omega}}$-invariant thanks to Theorem \ref{thm: ultrapower of modular automorphism}. Therefore we have $\sigma_t^{\dot{\varphi_j}_{\omega}}=\sigma_t^{\varphi_j^{\omega}}|_{M'\cap M^{\omega}}\ (t\in \mathbb{R}, 1\le j\le n)$. 
This implies that 
\[\Delta_{\varphi_j^{\omega}}^{\frac{1}{2}}y\xi_{\varphi_j^{\omega}}=\lambda^{\frac{1}{2}}y\xi_{\varphi_j^{\omega}},\ \ \ \ (1\le j\le n.)\]
This means that (Corollary \ref{cor: some useful conseuences of main theorem} (2))
\[\lim_{k\to \omega}\|\Delta_{\varphi_j}^{\frac{1}{2}}y_k\xi_{\varphi_j}-\lambda^{\frac{1}{2}}y_k\xi_{\varphi_j}\|=0,\ \ \ \ (1\le j\le n).\]
Choose $k\in \mathbb{N}$ such that the following inequalities hold:
\eqa{
\|\Delta_{\varphi_j}^{\frac{1}{2}}y_k\xi_{\varphi_j}-\lambda^{\frac{1}{2}}y_k\xi_{\varphi_j}\|&\le \varepsilon (1-\varepsilon),\ \ \ \ (1\le j\le n)\\
|\|y_k\|_{\psi}^2-1|&<\varepsilon.
}
It follows that for each $1\le j\le n$ and for $x=y_k$, we have 
\eqa{
\|\Delta_{\varphi_j}^{\frac{1}{2}}x\xi_{\varphi_j}-\lambda^{\frac{1}{2}}x\xi_{\varphi_j}\|&\le \varepsilon \|x\|_{\psi}^2=\varepsilon \sum_{i=1}^n\varphi_i(x^*x).
}
By Lemma \ref{lem:  ARaki's property L'}, we have $M\cong M\overline{\otimes}R_{\lambda}$.\\  
Conversely, assume $M\cong M\overline{\otimes}R_{\lambda}$ holds. Fix $\psi \in S_{\rm{nf}}(M)$ and put $N:=M\overline{\otimes}R_{\lambda}$. Let $\varphi_{\lambda}=\bigotimes_{\mathbb{N}}\text{Tr}(\rho_{\lambda}\ \cdot\ )$ and let $x_n:=\pi^{-1}(1\otimes u_n)\in M$, where 
\[u_n:=1^{\otimes n}\otimes \begin{pmatrix}0 & 1\\ 0 & 0\end{pmatrix}\otimes 1\cdots \in R_{\lambda},\ \ \ \ n\in \mathbb{N},\]
and $\pi\colon M\stackrel{\cong}{\to}N$ is a *-isomorphism. 
Then it holds that $(x_n)_n\in \mathcal{M}^{\omega}(M)$.
 Indeed, it is clear that $\|x_n\|=1,\ n\ge 1$ and hence $(x_n)_n\in \ell^{\infty}(\mathbb{N},M)$ Let $L_{\psi}\colon M\overline{\otimes}R_{\lambda}\to R_{\lambda}$ be a left-slice map given as the extension of the map $L_{\psi}^0$ defined on the algebraic tensor product $M\odot R_{\lambda}$ by 
\[L_{\psi}^0\left (\sum_{i}a_i\otimes b_i\right ):=\sum_{i}\psi(a_i)b_i, \ \ \ \ a_i\in M,\ b_i\in R_{\lambda}.\]
$L_{\psi}$ is a normal conditional expectation (see \cite[III.2.2.6]{Blackadar}). Let $(b_n)_n\in \mathcal{I}_{\omega}(N)$. 
Using $u_n\varphi_{\lambda}=\lambda^{-1}\varphi_{\lambda}u_n$, we have
\eqa{
\|b_n\pi(x_n)\|_{\psi\otimes \varphi_{\lambda}}^2&=\psi\otimes \varphi_{\lambda}((1\otimes u_n^*)b_n^*b_n(1\otimes u_n))=\varphi_{\lambda}(L_{\psi}((1\otimes u_n^*)b_n^*b_n(1\otimes u_n)))\\
&=\varphi_{\lambda}(u_n^*L_{\psi}(b_n^*b_n)u_n)=\lambda^{-1} \varphi_{\lambda}(u_nu_n^*L_{\psi}(b_nb_n^*))\\
&\le \lambda^{-1}\|L_{\psi}(b_n^*b_n)\|_{\varphi_{\lambda}}\|\|u_nu_n^*\|_{\varphi_{\lambda}}\\
&\le \lambda^{-1}|\varphi_{\lambda} (L_{\psi}((b_n^*b_n)^2))|^{\frac{1}{2}}\ \ \ \ (L_{\psi}\text{ is a conditional expectation}),
}
and since $(b_n^*b_n)_{n=1}^{\infty}\in \mathcal{I}_{\omega}(N)$, we have 
\eqa{
\varphi_{\lambda}(L_{\psi}((b_n^*b_n)^2))&=\psi\otimes \varphi_{\lambda} ((b_n^*b_n)^2)=\|b_n^*b_n\|_{\psi\otimes \varphi_{\lambda}}^2\\
&\stackrel{n\to \omega}{\to}0.
}
Therefore $(b_n\pi(x_n))_{n=1}^{\infty}\in \mathcal{L}_{\omega}(N)$. Since $(b_n\pi(x_n))_{n=1}^{\infty}\in \mathcal{L}_{\omega}^*(N)$ automatically, we have $(b_n\pi(x_n))_{n=1}^{\infty}\in \mathcal{I}_{\omega}(N)$ Similarly, 
we have 
$(\pi(x_n)b_n)_n\in \mathcal{L}_{\omega}^*(N)$ and thus $(\pi(x_n)b_n)_n\in \mathcal{I}_{\omega}(N)$, which shows that $(\pi(x_n))_n\in \mathcal{M}^{\omega}(N)$, and hence $(x_n)_n\in \mathcal{M}^{\omega}(M)$. It is then easy to show that $(x_n)^{\omega}\in M'\cap M^{\omega}$.    
It also holds that $\sigma_t^{\psi\otimes \varphi_{\lambda}}(\pi(x_n))=\lambda^{it}\pi(x_n)$ for each $t\in \mathbb{R}$, where $\varphi_{\lambda}$ is the Powers state and $\psi$ is a normal faithful state on $M$. Therefore $\varphi:=(\psi \otimes \varphi_{\lambda})\circ \pi\in S_{{\rm{nf}}}(M)$ satisfies $\sigma_t^{\dot{\varphi}_{\omega}}((x_n)^{\omega})=\lambda^{it}(x_n)^{\omega}$.  
Therefore $\lambda \in \sigma_p (\Delta_{\dot{\varphi}_{\omega}})$ holds.\\
\end{proof}

Now recall the definition of Connes' asymptotic centralizer. 
\begin{definition}{\rm{(Connes \cite{Connes3})}}\label{def: asymptotic centralizer}
The {\it asymptotic centralizer} $M_{\omega}$ of $M$ is defined as the quotient C$^*$ algebra $\mathcal{M}_{\omega}(\mathbb{N},M)/\mathcal{I}_{\omega}(\mathbb{N},M)$, where
\[\mathcal{M}_{\omega}(\mathbb{N},M):=\left \{(x_n)_n\in \ell^{\infty}(\mathbb{N},M); \lim_{n\to \omega}\|x_n\psi-\psi x_n\|=0,\ \ \forall \psi \in M_*\right \}.\]
$M_{\omega}$ is a finite von Neumann algebra for any $M$. 
\end{definition}
Regarding Golodets' and Golodets-Nessonov's work (2) above, we prove next that $(C_M^{\omega})_{\dot{\varphi}_{\omega}}$ is nothing but $M_{\omega}$ when we identify $C_M^{\omega}$ with $M'\cap M^{\omega}$ (cf. $\S$\ref{subsec: Golodets' asymptotic algebra}, Theorem \ref{thm: identification of Golodets algebra}). 
Note that we do not need the factoriality of $M$ or the separability of the predual. 
This result will play a crucial role in answering Ueda's question in $\S$\ref{sec: Ueda's Question}.

\begin{proposition}\label{prop: centralizer of ultrapower state} Let $M$ be a $\sigma$-finite von Neumann algebra. Let $\varphi\in S_{\rm{nf}}(M)$. Then the centralizer of the Golodets state $\dot{\varphi}_{\omega}$ is $M_{\omega}$. 
\end{proposition}
 Note that Proposition \ref{prop: centralizer of ultrapower state} gives an alternative proof of the fact \cite[Theorem 2.9 (1)]{Connes3} that $M_{\omega}$ is a (finite) von Neumann algebra.
\begin{lemma}\label{lem: characterization of the centralizer1}
Let $(M_n,\varphi_n)_n$ be a sequence of pairs of $\sigma$-finite von Neumann algebras and normal faithful states. Let $(x_n)_n, (y_n)_n\in \mathcal{M}^{\omega}(M_n,\varphi_n)$. Then we have 
\[\|(x_n)^{\omega}(\varphi_n)^{\omega}-(\varphi_n)^{\omega}(y_n)^{\omega}\|=\lim_{n\to \omega}\|x_n\varphi_n-\varphi_ny_n\|\]
In particular, $(x_n)^{\omega}\in ((M_n,\varphi_n)^{\omega})_{\varphi^{\omega}}$ holds if and only if $\lim_{n\to \omega}\|x_n\varphi_n-\varphi_nx_n\|=0$ holds.
\end{lemma}
\begin{proof}We use abbreviated notations as $\ell^{\infty},\mathcal{M}^{\omega},\mathcal{L}_{\omega},\mathcal{I}_{\omega}$ as in $\S$\ref{subsec: Relation ship between Ocneanu and Rayanud}. 
Put $C_1:=\|(x_n)^{\omega}(\varphi_n)^{\omega}-(\varphi_n)^{\omega}(y_n)^{\omega}\|$ and $C_2:=\lim_{n\to \omega}\|x_n\varphi_n-\varphi_ny_n\|$. 
Let $\varepsilon>0$, and choose $a\in {\rm{Ball}}((M_n,\varphi_n)^{\omega}))$ such that \[|\nai{a}{(x_n)^\omega (\varphi_n)^{\omega}-(\varphi_n)^{\omega}(y_n)^{\omega}}|>C_1-\varepsilon.\]
Since $(M_n,\varphi_n)^{\omega}$ is a quotient of $\mathcal{M}^{\omega}$, we may find $(a_n)_n\in \mathcal{M}^{\omega}$ with $a=(a_n)^{\omega}$, such that $\|(a_n)_n\|=\sup_{n\ge 1}\|a_n\|\le 1$. 

Therefore we have
\[C_1-\varepsilon<\lim_{n\to \omega}|\nai{a_n}{x_n\varphi_n-\varphi_ny_n}|\le \lim_{n\to \omega}\|x_n\varphi_n-\varphi_ny_n\|.\]
Since $\varepsilon>0$ is arbitrary, we have $C_1\le C_2$.

To prove $C_2\le C_1$, let $a_n\in {\rm{Ball}}(M_n)\ (n\in \mathbb{N})$ be such that 
\begin{equation}
|\nai{a_n}{x_n\varphi_n-\varphi_ny_n}|>\|x_n\varphi_n-\varphi_ny_n\|-\frac{1}{n}. \label{eq: a_n almost attains the norm} 
\end{equation}
By Proposition \ref{cor: decomposition of ell^infty}, there exists $(b_n)_n\in \mathcal{M}^{\omega}$, $(c_n)_n\in \mathcal{L}_{\omega}$, and $(d_n)_n\in \mathcal{L}_{\omega}^*$ such that $a_n=b_n+c_n+d_n (n\in \mathbb{N})$ and $\|(b_n)^{\omega}\|\le \lim_{n\to \omega}\|a_n\|\le 1$. 
It follows that
\begin{align}
\nai{a_n}{x_n\varphi_n-\varphi_ny_n}&=\nai{b_n}{x_n\varphi_n-\varphi_ny_n}+\nai{c_nx_n\xi_{\varphi_n}}{\xi_{\varphi_n}}-\nai{c_n\xi_{\varphi_n}}{y_n^*\xi_{\varphi_n}}\notag \\
&\ \ \ +\nai{x_n\xi_{\varphi_n}}{d_n^*\xi_{\varphi_n}}-\nai{\xi_{\varphi_n}}{d_n^*y_n^*\xi_{\varphi_n}}\label{eq: decompose by Momega and Lomega}.
\end{align}
Since $(c_n)_n\in \mathcal{L}_{\omega}$ and $(d_n)_n\in \mathcal{L}_{\omega}^*$, the third and the fourth term in the right hand side of Eq. (\ref{eq: decompose by Momega and Lomega}) will vanish as $n\to \omega$.  
Also, as $(x_n)_n\in \mathcal{M}^{\omega}$, Corollary \ref{cor: M^{omega} cap (L+L^*)=I} (1) implies that $(c_nx_n)_n\in \mathcal{L}_{\omega}$ and $(d_n^*y_n^*)_n\in \mathcal{L}_{\omega}$, whence the second and the fifth term will vanish as $n\to \omega$. Therefore we have
\begin{equation}
\lim_{n\to \omega}|\nai{a_n}{x_n\varphi_n-\varphi_ny_n}|=\lim_{n\to \omega}|\nai{b_n}{x_n\varphi_n-\varphi_ny_n}|. \label{eq: norm is also attained by b_n}
\end{equation}
Then by Eqs. (\ref{eq: a_n almost attains the norm}) and (\ref{eq: norm is also attained by b_n}), we have 
\eqa{
\lim_{n\to \omega}\|x_n\varphi_n-\varphi_ny_n\|&\le \lim_{n\to \omega}|\nai{b_n}{x_n\varphi_n-\varphi_ny_n}|. \label{eq: norm is also attained by b_n}\\
&=\nai{(b_n)^{\omega}}{(x_n)^{\omega}(\varphi_n)^{\omega}-(\varphi_n)^{\omega}(y_n)^{\omega}}\\
&\le \|(x_n)^{\omega}(\varphi_n)^{\omega}-(\varphi_n)^{\omega}(y_n)^{\omega}\|,
}
whence $C_2\le C_1$. This finishes the proof.   
\end{proof}

\begin{proof}[Proof of Proposition \ref{prop: centralizer of ultrapower state}]
For $M_{\omega}\subset (M'\cap M^{\omega})_{\dot{\varphi}_{\omega}}$, let $(x_n)^{\omega}\in M_{\omega}$. Then for $(y_n)^{\omega}\in M^{\omega}$, we have 
\eqa{
|\varphi(y_nx_n-x_ny_n)|&=|[x_n,\varphi](y_n)|\le \|y_n\|\cdot \|[x_n,\varphi]\|\\
&\stackrel{n\to \omega}{\to}0.
}
Hence $(x_n)^{\omega}\in (M^{\omega})_{\varphi^{\omega}}\cap (M'\cap M^{\omega})\subset (M'\cap M^{\omega})_{\dot{\varphi}_{\omega}}$ holds.\\
For $M_{\omega}\supset (M'\cap M^{\omega})_{\dot{\varphi}_{\omega}}$, let $(x_n)^{\omega}\in (M'\cap M^{\omega})_{\dot{\varphi}_{\omega}}$. Since $\sigma_t^{\dot{\varphi}_{\omega}}=\sigma_t^{\varphi^{\omega}}|_{M'\cap M^{\omega}}\ (t\in \mathbb{R})$ (see the proof of Theorem \ref{thm: point spectrum and property L'}), we have 
$\sigma_t^{\varphi^{\omega}}((x_n)^{\omega})=(x_n)^{\omega}\ (t\in \mathbb{R})$. Therefore by 
By Lemma \ref{lem: characterization of the centralizer1}, we have 
\[(x_n)^{\omega}\varphi^{\omega}=\varphi^{\omega}(x_n)^{\omega}\Leftrightarrow \lim_{n\to \omega}\|x_n\varphi-\varphi x_n\|=0.\]

Then by \cite[Lemma XIV.3.4 (ii)]{TakBook},  $(x_n)^{\omega}\in M_{\omega}$ holds.  
\end{proof}
\begin{remark}
The equivalence $(x_n)^{\omega}\varphi^{\omega}=\varphi^{\omega}(x_n)^{\omega}\Leftrightarrow \lim_{n\to \omega}\|x_n\varphi-\varphi x_n\|=0$ can be seen using Corollary \ref{cor: some useful conseuences of main theorem} (2) and  \cite[Lemma 2.8 (a)]{Haagerup3} instead. 
\end{remark}
\section{Ueda's Question: $M_{\omega}=\mathbb{C}\stackrel{?}{\Rightarrow }M'\cap M^{\omega}=\mathbb{C}$}\label{sec: Ueda's Question}
Let $M$ be a $\sigma$-finite von Neumann algebra. Connes \cite{Connes3} defined the asymptotic centralizer $M_{\omega}$ (see Definition \ref{def: asymptotic centralizer}) as a generalization of $M'\cap M^{\omega}$ for the case of type II$_1$ factor. 
It is known that if $M$ is $\sigma$-finite, and if $(x_n)^{\omega}\in M'\cap M^{\omega}$ satisfies $\lim_{n\to \omega}\|x_n\varphi-\varphi x_n\|=0$ for one $\varphi \in S_{\rm{nf}}(M)$, then $(x_n)^{\omega}\in M_{\omega}$ (see \cite[Lemma XIV.3.4 (ii)]{TakBook}). Therefore the existence of a normal faithful tracial state shows that $M'\cap M^{\omega}=M_{\omega}$ for a finite von Neumann algebra. The same is true for type II$_{\infty}$ factors. However, for type III factors, it is often the case that $M_{\omega}\subsetneq M'\cap M^{\omega}$. 
\begin{example}\label{ex: Takesaki's example?}The following example has been known to experts. We add it for the reader's convenience. Let $(R_{\lambda}, \varphi)=\bigotimes_{n\in \mathbb{N}}(M_2(\mathbb{C}), \text{Tr}(\rho_{\lambda}\cdot ))$ be the Powers factor of type III$_{\lambda} (0<\lambda<1)$, where $\rho_{\lambda}=\text{diag}(\frac{\lambda}{1+\lambda},\frac{1}{1+\lambda})$. Let 
\[u_n:=1^{\otimes n}\otimes \mattwo{0}{1}{0}{0}\otimes 1\otimes \cdots\in R_{\lambda},\ \ \ \ \ \ n\ge 1.\]
Then $(u_n)_n\in \mathcal{M}^{\omega}(R_{\lambda})$ and $(u_n)^{\omega}\in R_{\lambda}'\cap R_{\lambda}^{\omega}$. On the other hand, we have 
\[\varphi u_n=\lambda u_n\varphi,\ \ \ \ \ \ \ n\in \mathbb{N}.\]
Therefore $\|u_n\varphi-\varphi u_n\|=(1-\lambda)\neq 0$ ($n\in \mathbb{N}$), and hence $(u_n)^{\omega}\notin (R_{\lambda})_{\omega}$. Moreover, $R_{\lambda}'\cap R_{\lambda}^{\omega}$ is a type III$_{\lambda}$ factor. To see this, by \cite[Proposition 1]{Takesaki4}, $(R_{\lambda})_{\omega}$ is a type II$_1$ factor. Therefore by Proposition \ref{prop: centralizer of ultrapower state}, the centralizer of the Golodets state $\dot{\varphi}_{\omega}=\varphi^{\omega}|_{R_{\lambda}'\cap (R_{\lambda})^{\omega}}$ is a factor, whence by Corollary \ref{cor: some useful conseuences of main theorem} (3), we have
\eqa{
\Gamma(\sigma^{\dot{\varphi}_{\omega}})&=\text{Sp}(\sigma^{\dot{\varphi}_{\omega}})=\log (\sigma(\Delta_{\dot{\varphi}_{\omega}})\setminus \{0\})\\
&\subset \log (\sigma(\Delta_{\varphi^{\omega}})\setminus \{0\})=\log (\sigma(\Delta_{\varphi})\setminus \{0\})\\
&=(\log \lambda)\mathbb{Z}.
}
On the other hand, we have $(\log \lambda)\mathbb{Z}\subset \text{Sp}(\sigma^{\dot{\varphi}_{\omega}})$. Therefore as $\Gamma(\sigma^{\dot{\varphi}_{\omega}})=\log (S(R_{\lambda}'\cap R_{\lambda}^{\omega})\setminus \{0\})$, we have
\[S(R_{\lambda}'\cap R_{\lambda}^{\omega})=\{\lambda^n;n\in \mathbb{Z}\}\cup \{0\}.\]
This proves that $R_{\lambda}'\cap R_{\lambda}^{\omega}$ is a type III$_{\lambda}$ factor.
\end{example}
In spite of the above example, in \cite[$\S$5.2]{Ueda1}, Ueda asked whether $M_{\omega}=\mathbb{C}$ implies $M'\cap M^{\omega}=\mathbb{C}$. We prove that the answer to his question is affirmative when $M$ has separable predual. 

\begin{theorem}\label{thm: solution to Ueda problem}Let $M$ be a von Neumann algebra with a separable predual for which $M_{\omega}=\mathbb{C}$ holds. Then $M'\cap M^{\omega}=\mathbb{C}$ holds.
\end{theorem}

The following Lemma is well-known.
\begin{lemma}\label{lem: centralizer cannot be trivial unless C or III_1}
Let $M$ be a von Neumann algebra, $\varphi$ be a normal faithful state on $M$ with $M_{\varphi}=\mathbb{C}$. Then $M$ is either $\mathbb{C}$ or a factor of type {\rm{III}}$_1$. 
\end{lemma}
\begin{proof}
Let $H$ be a Hilbert space on which $M$ acts. Since $\mathcal{Z}(M)\subset M_{\varphi}=\mathbb{C}$, $M$ is a factor. Suppose $M$ is semifinite with a normal faithful semifinite trace $\tau$. Then there exists a positive self-adjoint operator $h\in L^1(M,\tau)$ with $\tau(h)=1$ such that $\varphi=\tau(h\cdot )$ holds. It is well known that this implies   $\sigma_t^{\varphi}(x)=h^{it}xh^{-it}$ for every $x\in M$ and $t\in \mathbb{R}$. Let $A$ be the abelian von Neumann algebra generated by all spectral projections of $h$. Then for $x\in M$, $x\in M_{\varphi}$ holds if and only if $x$ commutes with $h^{it}$ for all $t\in \mathbb{R}$, which is equivalent to the condition $x\in A'$, hence $M_{\varphi}=A'\cap M=\mathbb{C}$. Since $A\subset A'\cap M=\mathbb{C}$, $h$ must be a multiple of 1 and $\tau$ is a tracial state. This implies that $\varphi=\tau$, and \[M_{\varphi}=M_{\tau}=M=\mathbb{C}.\] 
Suppose next that $M$ is of type III$_{\lambda}\ (\lambda \neq 1)$. Then by Th\'eor\`eme 4.2.1 ($0<\lambda<1$ case) and Th\'eor\`eme 5.2.1 ($\lambda=0$ case) of Connes \cite{Connes1}, there exists a maximal abelian subalgebra $A$ of $M_{\varphi}$ which is maximal abelian in $M$. This in particular means that $M_{\varphi}$ cannot be $\mathbb{C}$. This finishes the proof.    
\end{proof}
We are now ready to prove Theorem \ref{thm: solution to Ueda problem}.
\begin{proof}[Proof of Theorem \ref{thm: solution to Ueda problem}]
Put $N:=M'\cap M^{\omega}$. Take an arbitrary $\varphi \in S_{\text{nf}}(M)$. Since $\mathcal{Z}(M)\subset M_{\omega}=\mathbb{C}$, $M$ is a factor. By Proposition \ref{prop: phi_dot does not depend on phi}, the Golodets state $\dot{\varphi}_{\omega}:=\varphi^{\omega}|_N\in S_{\text{nf}}(N)$ does not depend on the choice of $\varphi$. By Proposition \ref{prop: centralizer of ultrapower state}, $N_{\dot{\varphi}_{\omega}}=\mathbb{C}$. Then by Lemma \ref{lem: centralizer cannot be trivial unless C or III_1}, $N$ is either $\mathbb{C}$ or a factor of type III$_1$. Suppose $N$ is a type III$_1$ factor and we shall get a contradiction. Fix $0<\lambda<1$. Since $N$ is of type III, there exists an automorphism $\alpha: N\to N\otimes M_2(\mathbb{C})$. Define $\psi\in S_{\text{nf}}(N)$ by 
\[\psi:=\left [\dot{\varphi}_{\omega}\otimes \text{Tr}(\rho_{\lambda}\cdot )\right ]\circ \alpha,\]
where $\rho_{\lambda}:=\text{diag}(\frac{\lambda}{1+\lambda},\frac{1}{1+\lambda})$. 
Let $\varepsilon>0$ be given. By Connes-St\o rmer transitivity \cite{CS} (note that the transitivity holds without any assumption on the predual thanks to \cite{HS}), there exists $u\in \mathcal{U}(N)$ such that 
\begin{equation}
\|\dot{\varphi}_{\omega}-u\psi u^*\|<\varepsilon. \label{eq: psi and psi_tilde}
\end{equation}
Define a $2\times 2$ matrix unit $\{f_{i,j}\}_{i,j=1}^2$ in $N$ by 
\[f_{ij}:=u^*\alpha^{-1}(1\otimes e_{i,j})u,\ \ \ \ 1\le ij\le 2,\]
where $\{e_{i,j}\}_{i,j=1}^2$ is the standard matrix unit in $M_2(\mathbb{C})$. 
For $x\in N$, write $\alpha(x)=\mattwo{x_{11}}{x_{12}}{x_{21}}{x_{22}}$, where $x_{ij}\in N$. By a straightforward computation, we have 
\eqa{
\left [\psi \alpha^{-1}(1\otimes e_{12})\right ](x)&=\psi(\alpha^{-1}(1\otimes e_{12})x)\\
&=[\dot{\varphi}_{\omega}\otimes \text{Tr}(\rho_{\lambda}\cdot )]\mattwo{x_{21}}{x_{22}}{0}{0}\\
&=\frac{\lambda}{1+\lambda}\dot{\varphi}_{\omega}(x_{21})\\
\left [\alpha^{-1}(1\otimes e_{12})\psi \right ](x)&=[\dot{\varphi}_{\omega}\otimes \text{Tr}(\rho_{\lambda}\cdot )]\mattwo{0}{x_{11}}{0}{x_{21}}\\
&=\frac{1}{1+\lambda}\dot{\varphi}_{\omega}(x_{21}).
}
Doing similar computations, we have the following equalities:
\begin{align}
\psi\alpha^{-1}(1\otimes e_{ii})&=\alpha^{-1}(1\otimes e_{ii})\psi,\ \ \ \ \ \ \ \ \ \ (i=1,2) \label{eq: twobytwo matrix(1)}\\
\psi\alpha^{-1}(1\otimes e_{12})&=\lambda\alpha^{-1}(1\otimes e_{12})\psi,\label{eq: twobytwo matrix(2)}\\
\psi\alpha^{-1}(1\otimes e_{21})&=\lambda^{-1}\alpha^{-1}(1\otimes e_{21})\psi.\label{eq: twobytwo matrix(3)}
\end{align}
Using Eq. (\ref{eq: psi and psi_tilde}) and Eqs. (\ref{eq: twobytwo matrix(1)})-(\ref{eq: twobytwo matrix(3)}), it follows that 
\eqa{
\|\dot{\varphi}_{\omega} f_{12}-\lambda f_{12}\dot{\varphi}_{\omega}\|&=\|\dot{\varphi}_{\omega} u^*\alpha^{-1}(1\otimes e_{12})u-\lambda u^*\alpha^{-1}(1\otimes e_{12})u\dot{\varphi}_{\omega}\|\\
&=\|u\dot{\varphi}_{\omega} u^*\alpha^{-1}(1\otimes e_{12})-\lambda \alpha^{-1}(1\otimes e_{12})u\dot{\varphi}_{\omega} u^*\|\\
&\le \|(u\dot{\varphi}_{\omega} u^*-\psi)\alpha^{-1}(1\otimes e_{12})\|+\|\lambda \alpha^{-1}(1\otimes e_{12})(\psi-u\dot{\varphi}_{\omega} u^*)\|\\
&\le (1+\lambda)\varepsilon.
}
Doing similar computations, we obtain
\begin{align}
\|\dot{\varphi}_{\omega} f_{ii}-f_{ii}\dot{\varphi}_{\omega} \|&\le 2\varepsilon, \ \ \ \ \ \ \ \ \ \ \ \ \ \ \ \ \ \ \ \ \ \ \ (i=1,2) \label{eq: twobytwo matrix perturbed(1)}\\
\|\dot{\varphi}_{\omega} f_{12}-\lambda f_{12}\dot{\varphi}_{\omega} \|&\le (1+\lambda)\varepsilon, \label{eq: twobytwo matrix perturbed(2)}\\
\|\dot{\varphi}_{\omega} f_{21}-\lambda^{-1}f_{21}\dot{\varphi}_{\omega} \|&\le (1+\lambda^{-1})\varepsilon. \label{eq: twobytwo matrix perturbed(3)}
\end{align}
Let $\{a_n\}_{n=1}^{\infty}$ be a $\|\cdot \|_{\varphi}^{\sharp}$-dense sequence of the unit ball of $M$.\\ \\
\textbf{Claim 1}. For each $n\in \mathbb{N}$ there exists $f_{ij}^{(n)}\in M\ (i,j=1,2)$ satisfying the following conditions:
\begin{itemize}
\item[(i)] $\|f_{ij}^{(n)}\|\le 1\ \ (i,j=1,2)$.
\item[(ii)] $\|\varphi f_{ii}^{(n)}-f_{ii}^{(n)}\varphi\|\le \frac{1}{n}\ \ (i=1,2)$. 
\item[(iii)] $\|\varphi f_{12}^{(n)}-\lambda f_{12}^{(n)}\varphi\|\le \frac{1}{n}$.
\item[(iv)] $\|\varphi f_{21}^{(n)}-\lambda^{-1}f_{21}^{(n)}\varphi \|\le \frac{1}{n}$.
\item[(v)] $\|(f_{ij}^{(n)})^*-f_{ji}^{(n)}\|_{\varphi}^{\sharp}\le \frac{1}{n}\ \ (i,j=1,2)$.
\item[(vi)] $\|f_{ij}^{(n)}a_m-a_mf_{ij}^{(n)}\|_{\varphi}^{\sharp}\le \frac{1}{n}\ \ (1\le m\le n,\ i,j=1,2)$.
\item[(vii)] $\|f_{ij}^{(n)}f_{kl}^{(n)}-\delta_{jk}f_{il}^{(n)}\|_{\varphi}^{\sharp}\le \frac{1}{n}\ \ (i,j,k,l=1,2)$
\item[(viii)] $\|f_{11}^{(n)}+f_{22}^{(n)}-1\|_{\varphi}^{\sharp}\le \frac{1}{n}$.
\end{itemize}
By Eqs (\ref{eq: twobytwo matrix perturbed(1)})-(\ref{eq: twobytwo matrix perturbed(3)}), there exists a matrix unit $(f_{ij})_{i,j=1}^{2}\in M'\cap M^{\omega}$ satisfying the following conditions.
\eqa{
\|\dot{\varphi}_{\omega} f_{ii}-f_{ii}\dot{\varphi}_{\omega} \|&\le \frac{1}{2n}.\ \ \ \ (i=1,2)\\
\|\dot{\varphi}_{\omega} f_{12}-\lambda f_{12}\dot{\varphi}_{\omega} \|&\le \frac{1}{2n}.\\
\|\dot{\varphi}_{\omega} f_{21}-\lambda^{-1}f_{21}\dot{\varphi}_{\omega} \|&\le \frac{1}{2n}.
}
Since $M'\cap M^{\omega}$ is $\sigma_t^{\varphi^{\omega}}$-invariant, by Takesaki's Theorem \cite{Takesaki2}, there exists a normal faithful conditional expectation $E\colon M^{\omega}\to M'\cap M^{\omega}$ with $\varphi^{\omega}=\dot{\varphi}_{\omega} \circ E$. Since $f_{ij}\in M'\cap M^{\omega}$, we have $E(f_{ij})=f_{ij}$. Therefore for every $a\in M^{\omega}$, we have
\eqa{
(\varphi^{\omega}f_{ii}-f_{ii}\varphi^{\omega})(a)&=\varphi^{\omega}(f_{ii}a-af_{ii})=\dot{\varphi}_{\omega} \circ E(f_{ii}a-af_{ii})\\
&=\dot{\varphi}_{\omega}(f_{ii}E(a)-E(a)f_{ii})=(\dot{\varphi}_{\omega} f_{ii}-f_{ii}\dot{\varphi}_{\omega})(E(a)),
}
and hence $\|\varphi^{\omega}f_{ii}-f_{ii}\varphi^{\omega}\|\le \|\dot{\varphi}_{\omega} f_{ii}-f_{ii}\dot{\varphi}_{\omega}\|$. Since $\|\varphi^{\omega}f_{ii}-f_{ii}\varphi^{\omega}\|\ge \|\dot{\varphi}_{\omega} f_{ii}-f_{ii}\dot{\varphi}_{\omega}\|$, we have 
\[\|\varphi^{\omega}f_{ii}-f_{ii}\varphi^{\omega}\|=\|\dot{\varphi}_{\omega} f_{ii}-f_{ii}\dot{\varphi}_{\omega}\|\le \frac{1}{2n}\hspace{1.0cm} (i=1,2).\]
Similarly, we have
\eqa{
\|\varphi^{\omega}f_{12}-\lambda f_{12}\varphi^{\omega}\|\le \frac{1}{2n},\\
\|\varphi^{\omega}f_{21}-\lambda^{-1}f_{21}\varphi^{\omega}\|\le \frac{1}{2n}.
}
Choose $(f_{ij}^{(k)})_k\in \mathcal{M}^{\omega}\ (i,j=1,2)$ such that $f_{ij}=(f_{ij}^{(k)})^{\omega}$. They can be chosen to satisfy $\|f_{ij}^{(k)}\|\le 1$. By the definition of $f_{ij}$ and the matrix unit property, together with Lemma \ref{lem: characterization of the centralizer1}, we have 
\begin{list}{}{}
\item[(ii)$^*$] $\displaystyle \lim_{k\to \omega}\|\varphi f_{ii}^{(k)}-f_{ii}^{(k)}\varphi \|\le \frac{1}{2n}\ (i=1,2)$.
\item[(iii)$^*$] $\displaystyle \lim_{k\to \omega}\|\varphi f_{12}^{(k)}-\lambda f_{12}^{(k)}\varphi\|\le \frac{1}{2n}$.
\item[(iv)$^*$] $\displaystyle \lim_{k\to \omega}\|\varphi f_{21}^{(k)}-\lambda^{-1} f_{21}^{(k)}\varphi\|\le \frac{1}{2n}$.
\item[(v)$^*$] $\displaystyle \lim_{k\to \omega}\|(f_{ij}^{(k)})^*-f_{ji}^{(k)}\|_{\varphi}^{\sharp}=0\ (i,j=1,2)$.
\item[(vi)$^*$] $\displaystyle \lim_{k\to \omega}\|f_{ij}^{(k)}a_m-a_mf_{ij}^{(k)}\|_{\varphi}^{\sharp}=0\ (m\ge 1, i,j=1,2)$.
\item[(vii)$^*$] $\displaystyle \lim_{k\to \omega}\|f_{ij}^{(k)}f_{lm}^{(k)}-\delta_{jl}f_{im}^{(k)}\|_{\varphi}^{\sharp}=0\ (i,j,l,m=1,2)$.
\item[(viii)$^*$] $\displaystyle \lim_{k\to \omega}\|f_{11}^{(k)}+f_{22}^{(k)}-1\|_{\varphi}^{\sharp}=0$. 
\end{list}  
For fixed $n$, there are only finitely many conditions. Therefore there exists $k=k(n)\in \mathbb{N}$ such that $f_{ij}^{(k(n))}$ satisfies all the conditions (i)-(viii) in the claim.\\ \\
\textbf{Claim 2}. If $(f_{ij}^{(n)})_{i,j=1}^2\in M$ satisfies conditions (i)-(viii) in Claim 1 for all $n\ge 1$, then $(f_{ij}^{(n)})_n\in \mathcal{M}^{\omega}$ holds for $i,j=1,2$.\\
By (i), $(f_{ij}^{(n)})_n\in \ell^{\infty}(\mathbb{N},M)$ holds.  
Let $(b_n)_{n=1}^{\infty}\in \mathcal{I}_{\omega}$ with $\sup_{n\ge 1}\|b_n\|\le 1$. Then $(f_{ij}^{(n)}b_n) \in \mathcal{L}_{\omega}$ holds automatically. On the other hand we have
\eqa{
\varphi(f_{ij}^{(n)}b_n(f_{ij}^{(n)}b_n)^*)&\le |\varphi (f_{ij}^{(n)}b_nb_n^*\{(f_{ij}^{(n)})^*-f_{ji}^{(n)}\})|+|\varphi(f_{ij}^{(n)}b_nb_n^*f_{ji}^{(n)})|\\
&\le \|b_nb_n^*(f_{ij}^{(n)})^*\|_{\varphi}\|(f_{ij}^{(n)})^*-f_{ji}^{(n)}\|_{\varphi}+|(\varphi f_{ij}^{(n)}-c(i,j)f_{ij}^{(n)}\varphi)(b_nb_n^*f_{ji}^{(n)})|\\
&\hspace{2.0cm}+c(i,j)|\varphi(b_nb_n^*f_{ji}^{n}f_{ij}^{(n)})|\\
&\le \|b_n\|^2\cdot \frac{1}{n}+\frac{1}{2n}\cdot \|b_n\|^2+c(i,j)\|b_nb_n^*\|_{\varphi}\|f_{ji}^{(n)}f_{ij}^{(n)}\|_{\varphi}\\
&\le \frac{3}{2n}+c(i,j)\|b_nb_n^*\|_{\varphi}\stackrel{n\to \omega}{\to}0,
}
where
\[c(i,j):=\begin{cases}
1 & (i=j)\\
\lambda & (i=1,j=2)\\
\lambda^{-1} & (i=2,j=1).
\end{cases}\]
This shows that $(f_{ij}^{(n)}b_n)_n\in \mathcal{L}_{\omega}^*$, and hence $(f_{ij}^{(n)}b_n)\in \mathcal{I}^{\omega}=\mathcal{L}_{\omega}\cap \mathcal{L}_{\omega}^*$. Similarly, we have $(b_nf_{ij}^{(n)})_n\in \mathcal{I}^{\omega}$. This proves that $(f_{ij}^{(n)})_n\in \mathcal{M}^{\omega}$ for $i,j=1,2$.\\ \\
Therefore by Claim 1 and Claim 2, we see that $(g_{ij})_{i,j=1}^2$, where $g_{ij}:=(f_{ij}^{(n)})^{\omega}$ is a well-defined matrix unit in $M'\cap M^{\omega}$, and using conditions (i)-(viii) in Claim 1, we have 
\eqa{
g_{ii}\dot{\varphi}_{\omega}&=\dot{\varphi}_{\omega} g_{ii},\ \ \ \ \ \ \ \ \ \ \ \ \ \ \ (i=1,2)\\
\dot{\varphi}_{\omega} g_{12}&=\lambda g_{12}\dot{\varphi}_{\omega},\\
\dot{\varphi}_{\omega} g_{21}&=\lambda^{-1} g_{21}\dot{\varphi}_{\omega}.
}
In particular, $g_{ii}\in N_{\dot{\varphi}_{\omega}}\ (i=1,2)$ holds and $\lambda$ is in the point spectrum of $\Delta_{\dot{\varphi}_{\omega}}$. Then we have 
\eqa{
\dot{\varphi}_{\omega}(g_{11})&=\dot{\varphi}_{\omega}(g_{12}g_{21})=(g_{21}\dot{\varphi}_{\omega})(g_{12})\\
&=\lambda (\dot{\varphi}_{\omega} g_{21})(g_{12})=\lambda \dot{\varphi}_{\omega}(g_{21}g_{12})\\
&=\lambda \dot{\varphi}_{\omega}(g_{22}),
}
and hence $\dot{\varphi}_{\omega}(g_{11})=\frac{\lambda}{1+\lambda}$, which is neither $0$ nor $1$. Therefore $g_{11}\in N_{\dot{\varphi}_{\omega}}$ is a nontrivial projection. This implies $\dim(N_{\dot{\varphi}_{\omega}})\ge 2$, a contradiction. Hence $N$ must be $\mathbb{C}$. 
\end{proof}

Finally, we remark that there is no difference between $M_{\omega}$ and $M'\cap M^{\omega}$ when $M$ is of type III$_0$. 
\begin{proposition}\label{prop: no difference between M_omega and M'cap Momega}
If $M$ is a $\sigma$-finite type {\rm III}$_0$ factor, then $M'\cap M^{\omega}$ is a finite von Neumann algebra and $M'\cap M^{\omega}=M_{\omega}$ holds. 
\end{proposition}
\begin{proof}
Let $\varphi \in S_{\text{nf}}(M)$. By Proposition \ref{prop: phi_dot does not depend on phi}, the Golodets state $\dot{\varphi}_{\omega}=\varphi^{\omega}|_{M'\cap M^{\omega}}$ does not depend on $\varphi$. Hence by Corollary \ref{cor: some useful conseuences of main theorem} (3), we have
\eqa{
\sigma(\Delta_{\dot{\varphi}_{\omega}})&=\bigcap_{\psi \in S_{\text{nf}}(M)}\sigma(\Delta_{\dot{\psi}_{\omega}})\subset \bigcap_{\psi \in S_{\text{nf}}(M)}\sigma(\Delta_{\psi^{\omega}})\\
&= \bigcap_{\psi \in S_{\text{nf}}(M)}\sigma(\Delta_{\psi})=S(M)=\{0,1\},
}
whence $\sigma (\Delta_{\dot{\varphi}_{\omega}})=\{1\}$ because $0\notin \sigma_p(\Delta_{\dot{\varphi}_{\omega}})$. This shows that $\dot{\varphi}_{\omega}$ is a normal faithful trace on $M'\cap M^{\omega}$. Since $M_{\omega}$ is the centralizer of $\dot{\varphi}_{\omega}$ by Proposition \ref{prop: centralizer of ultrapower state}, we see that $M'\cap M^{\omega}=M_{\omega}$ holds.
\end{proof}

\section{Factoriality and Type of Ultraproducts}\label{sec: Factoriality and Type of Ultraproducts}
In this section, we study the factoriality and Murray-von Neumann-Connes type of the ultraproduct of factors. 
\subsection{Ultraproduct of Semifinite Factors}
The answers to factoriality/type questions for the Ocneanu ultrapower $M^{\omega}$ of a semifinite factor $M$ has been known. In fact, it has been known to experts that for a von Neumann algebra $M$ with separable predual, $(M\overline{\otimes}\mathbb{B}(H))^{\omega}\cong M^{\omega}\overline{\otimes}\mathbb{B}(H)$ and $(M\overline{\otimes}\mathbb{B}(H))_{\omega}\cong M_{\omega}\otimes \mathbb{C}$ holds, where $H$ is a separable Hilbert space. The proof can be found e.g., in \cite[Lemma 2.8]{MasudaTomatsu}. 
On the other hand, it is well-known that $M^{\omega}$ is a type II$_1$ factor if so is $M$. This shows the following folklore result: 
\begin{proposition}\label{prop: Masuda-Tomatsu for semifinite factor} Let $M$ be a semifinite factor with separable predual. Then $M^{\omega}$ is a factor. If $M$ is of type {\rm{I}}$_n\ (n\in \mathbb{N}\cup \{\infty\})$, {\rm{II}}$_1$ or {\rm{II}}$_{\infty}$, so is $M^{\omega}$.
\end{proposition}

On the other hand, the situation for the factoriality of the Groh-Raynaud ultraproduct is very different. Based on the local reflexivity principle for Banach spaces and the fact that $\mathbb{B}(H)^{**}$ is not semifinite, Raynaud \cite{Raynaud} showed that $\prod^{\mathcal{U}}\mathbb{B}(H)$ is not semifinite (for a free ultrafilter $\mathcal{U}$ on a suitable index set $I$ and infinite-dimensional $H$). We prove that $\prod^{\omega}R$ is not semifinite, where $R$ is the hyperfinite type II$_1$ factor. 
For a fixed $\lambda \in (0,1)$, put $\rho_{\lambda}={\rm{diag}}(\frac{\lambda}{1+\lambda},\frac{1}{1+\lambda})\in M_2(\mathbb{C})_+$, and let $R_{\lambda}=\bigotimes_{\mathbb{N}}(M_2(\mathbb{C}),\text{Tr}(\rho_{\lambda}\ \cdot ))$ be the Powers factor of type III$_{\lambda}$. Define $\varphi_n\in S_{\rm{nf}}(R)$ by 
\[\varphi_n:=\bigotimes_{k=1}^n\text{Tr}(\rho_{\lambda}\cdot )\otimes \bigotimes_{k=n+1}^{\infty}\frac{1}{2}\text{Tr},\ \ \ \ n\ge 1.\]
\begin{proposition}\label{prop: hiro embed Powers}
There exists a normal injective *-homomorphism $\pi\colon R_{\lambda}\to (R,\varphi_n)^{\omega}$ whose range is a normal faithful conditional expectation $\varepsilon\colon (R,\varphi_n)^{\omega}\to \pi(R_{\lambda})$.  
\end{proposition}

This shows that
\begin{proposition}\label{prop: (R,varphi_n)^{omega} not semifinite}
$(R,\varphi_n)^{\omega}$ is not semifinite.
\end{proposition}
\begin{proof}
By Proposition \ref{prop: hiro embed Powers}, the type III$_{\lambda}$ factor $\pi(R_{\lambda})$ is the range of a normal faithful conditional expectation from $(R,\varphi_n)^{\omega}$, hence \cite[Theorem 3]{Tomiyama} shows that $(R,\varphi_n)^{\omega}$ is not semifinite. 
\end{proof}
Therefore, we have
\begin{theorem}\label{thm: prod^{omega}R is not a factor}
$\prod^{\omega}R$ is not semifinite, and not a factor.
\end{theorem}
\begin{proof}
By Proposition \ref{prop: normalizer iff commutes with p}, the first claim is a corollary to Proposition \ref{prop: (R,varphi_n)^{omega} not semifinite}. Also by Proposition \ref{prop: normalizer iff commutes with p}, the type II$_1$ factor $R^{\omega}$ is a corner of $\prod^{\omega}R$. Therefore $\prod^{\omega}R$ has a nontrivial center. 
\end{proof}
Now we proceed to prove Proposition \ref{prop: hiro embed Powers}.
\begin{proof}[Proof of Proposition  \ref{prop: hiro embed Powers}]
Put $A_m:=\bigotimes_{k=1}^mM_2(\mathbb{C})\otimes \mathbb{C}\otimes \mathbb{C}\otimes \cdots$ considered as a subalgebra of $R_{\lambda}$, and let $\widehat{A}_m$ be the same algebra now considered as a subalgebra of $R$. Moreover, put
\[A:=\bigcup_{m=1}^{\infty}A_m\subset R_{\lambda},\ \ \ \ \widehat{A}:=\bigcup_{m=1}^{\infty}\widehat{A}_m\subset R.\]
For $x\in A$, let $\hat{x}$ denote the corresponding element in $\widehat{A}$. Define now a *-monomorphism $\pi_0\colon A\to \ell^{\infty}(\mathbb{N},R)$ by 
$\pi_0(x)=(\hat{x})_n,\ x\in A$ (constant sequence). Note that for $x\in A_m\ (m\in \mathbb{N}{\rm{\ fixed}})$, we have 
\[\varphi_n(\hat{x})=\varphi_{\lambda}(x),\ \ \ \sigma_t^{\varphi_n}(\hat{x})=\widehat{\sigma_t^{\varphi_{\lambda}}(x)},\ \ \ \ \ n\ge m.\]
Since $\sigma(\Delta_{\text{Tr}(\rho_{\lambda}\cdot )})=\sigma(\rho_{\lambda})\cdot \sigma(\rho_{\lambda}^{-1})=\{\lambda,1,\lambda^{-1}\}$, we have 
\[\sigma\left (\bigotimes_{k=1}^{m}\Delta_{\text{Tr}(\rho_{\lambda}\cdot )}\right )=\{\lambda^m,\lambda^{m-1},\cdots,\lambda^{-m}\}.\]
Therefore it holds that
\[\widehat{A}_m\subset R(\sigma^{\varphi_n},[m\log \lambda,-m\log \lambda])\]
for all $n\ge m$. Thus by Lemma \ref{lem: analiticity implies normalizer}, $\pi_0(A_m)\subset \mathcal{M}^{\omega}(R,\varphi_n)$ holds for all $m\in \mathbb{N}$ and hence also $\pi_0(A)\subset \mathcal{M}^{\omega}(R,\varphi_n)$ holds. Let $\pi_1\colon A\to M^{\omega}:=(R,\varphi_n)^{\omega}$ be $\pi_0$ composed with the quotient map from $\mathcal{M}^{\omega}(R,\varphi_n)$ onto $M^{\omega}=\mathcal{M}^{\omega}(R,\varphi_n)/\mathcal{I}_{\omega}(R,\varphi_n)$. Then it is elementary to check that
\[\varphi^{\omega}(\pi_1(x))=\varphi_{\lambda}(x),\ \ \ \ x\in A,\]
where $\varphi^{\omega}:=(\varphi_n)^{\omega}$. Using Theorem \ref{thm: ultrapower of modular automorphism}, we also have
\[\sigma_t^{\varphi^{\omega}}(\pi_1(x))=\pi_1 (\sigma_t^{\varphi_{\lambda}}(x)),\ \ \ x\in A,\ t\in \mathbb{R}.\]
Therefore $\pi_1$ extends to a normal *-monomorphism $\pi$ of $R_{\lambda}=\overline{A}^{\text{sot}}$ onto a von Neumann subalgebra $\pi(R_{\lambda})$ of $M^{\omega}$, which is invariant under $\sigma_t^{\varphi^{\omega}}\ (t\in \mathbb{R})$, whence by \cite{Takesaki2}, there is a normal faithful conditional $\varphi^{\omega}$-preserving expectation of $M^{\omega}$ onto $\pi(R_{\lambda})$.  
\end{proof}

\subsection{Ultraproduct of Type III$_{\lambda}$ ($\lambda \neq 0$) Factors}\label{subsec: UP of III_lambda}
Let $M$ be a $\sigma$-finite type III$_{\lambda}\ (0<\lambda \le 1)$ factor. We show that $\prod^{\omega}M$, as well as $(M,\varphi_n)^{\omega}$, is again a type III$_{\lambda}$ factor, and the isomorphism class of $(M,\varphi_n)^{\omega}$ does not depend on the choice of $(\varphi_n)_n\subset S_{\rm{nf}}(M)$. To do this, we first recall the state space diameter of factors. 
Let $M$ be a von Neumann algebra. Then an equivalence relation $\sim$ on $S_{\rm{n}}(M)$ is defined by $\varphi \sim \psi$ if they are approximately unitarily equivalent, i.e., there is a sequence of unitaries $(u_n)_n\subset \mathcal{U}(M)$ such that $\lim_{n\to \infty}\|\varphi -u_n\psi u_n^*\|=0$. Denote by $[\varphi]$ the equivalence class in $S_{\rm{n}}(M)$ represented by $\varphi \in S_{\rm{n}}(M)$. Then $S_{\rm{n}}(M)/\sim$ is a metric space by 
\[d([\varphi],[\psi]):=\inf_{u\in \mathcal{U}(M)}\|\varphi -u\psi u^*\|,\ \ \ [\varphi],[\psi]\in S_{\rm{n}}(M)/\sim.\]  
\begin{definition}
The {\it state space diameter} of $M$, denoted as $d(M)$ is defined by 
\[d(M):=\sup_{\varphi, \psi \in S_{\rm{n}}(M)}d([\varphi],[\psi]).\]
\end{definition}
It holds that $d(M)\le 2$, and $d(M)=2$ if $M$ is not a factor. By the result of Connes-St\o rmer \cite{CS}, Connes-Haagerup-St\o rmer \cite{CHS}, and Haagerup-St\o rmer \cite{HS}, the explicit form of $d(M)$ is given as follows. 
\begin{theorem}\label{thm: state space diameter theorem}
Let $M$ be a factor. Then the $d(M)$ is 
\begin{itemize}
\item[{\rm{(1)}}] $2(1-\frac{1}{n})$ if $M$ is of type {\rm{I}}$_n,\ (n\in \mathbb{N}\cup \{\infty\})$.
\item[{\rm{(2)}}] $2$ if $M$ is of type {\rm{II}}.
\item[{\rm{(3)}}] $2\frac{1-\lambda^{\frac{1}{2}}}{1+\lambda^{\frac{1}{2}}}$ if $M$ is of type {\rm{III}}$_{\lambda}\ (0\le \lambda\le 1)$. 
\end{itemize}
\end{theorem} 
\begin{remark}\label{rem: no separability assumption is necessary}
Although the above theorem in \cite{CS}, \cite{CHS} were stated for the separable predual case, the conclusion holds in full generality thanks to Martingale arguments given in \cite{HS}.  
\end{remark}
Let $(M_n,H_n)_n$ be a sequence of standard von Neumann algebras, and define the Groh-Raynaud ultraproduct $N=\prod^{\omega}M_n$. We will show the diameter formula $d(N)=\lim_{n\to \omega}d(M_n)$. 
\begin{lemma}\label{UW3.3.1} Let $(\varphi_n)_n, (\psi_n)_n\in \prod_{n\in \mathbb{N}}S_{\text{n}}(M_n)$ and let $\varphi=(\varphi_n)_{\omega}$ and $\psi=(\psi_n)_{\omega}$ be the corresponding normal states on $N$ (cf. Theorem \ref{UW2.3.3}). Then 
\[d([\varphi],[\psi])=\lim_{n\to \omega}d([\varphi_n],[\psi_n]).\]
\end{lemma}
\begin{remark}
The above limit clearly exists since the distances are bounded by 2. 
\end{remark}
\begin{proof}
For each $n\in \mathbb{N}$, choose a unitary $u_n\in M_n$ such that 
\[\|\varphi_n-u_n\psi_nu_n^*\|\le d([\varphi_n],[\psi_n])+\frac{1}{n},\]
then with $u:=(u_n)_{\omega}\in N$ we have
\[\|\varphi-u\psi u^*\|=\lim_{n\to \omega}\|\varphi_n-u_n\psi_n u_n^*\|\le \lim_{n\to \omega}d([\varphi_n],[\psi_n]).\]
Hence $d([\varphi],[\psi])\le \lim_{n\to \omega}d([\varphi_n],[\psi_n]).$

For the converse inequality, we use that the unitary group of $\pi_{\omega}((M_n)_{\omega})$ is strongly$^*$-dense in the unitary group of $N$ by Kaplansky density Theorem (cf. Definition \ref{UW2.1.2}). Hence given $\varepsilon>0$, we may choose a unitary $u_n\in M_n$ for each $n\in \mathbb{N}$, such that with $u:=(u_n)_{\omega}$, we have 
\[\|\varphi-u\psi u^*\|\le d([\varphi],[\psi])+\varepsilon.\]
But then
\[\lim_{n\to \omega}d([\varphi_n],[\psi_n])\le \lim_{n\to \omega}\|\varphi_n-u_n\psi_n u_n^*\|=\|\varphi-u\psi u^*\|\le d([\varphi]],[\psi])+\varepsilon.\]
Since $\varepsilon>0$ was arbitrary, we obtain $d([\varphi],[\psi])\ge \lim_{n\to \omega}d([\varphi_n],[\psi_n])$.
\end{proof}
\begin{lemma}\label{UW3.3.2} With the above notations, $\displaystyle d(N)=\lim_{n\to \omega}d(M_n)$.
\end{lemma}
\begin{proof}
For all $\varphi, \psi\in S_{\text{n}}(N)$ we may, by Corollary \ref{UW2.3.4}, choose normal states $(\varphi_n), (\psi_n)\in \prod_{n\in \mathbb{N}}S_{\text{n}}(M_n)$ such that $\varphi=(\varphi_n)_{\omega}$ and $\psi=(\psi_n)_{\omega}$. By Lemma \ref{UW3.3.1}, 
\[d([\varphi],[\psi])=\lim_{n\to \omega}d([\varphi_n],[\psi_n])\le \lim_{n\to \omega}d(M_n).\]
Hence $d(N)\le \lim_{n\to \omega}d(M_n)$.

Conversely, we may for each $n\in \mathbb{N}$ chose $\varphi_n, \psi_n\in S_{\text{n}}(M_n)$ such that 
\[d([\varphi_n],[\psi_n])\ge d(M_n)-\frac{1}{n}, \ \ \ n\in \mathbb{N}.\]
Let $\varphi:=(\varphi_n)_{\omega}$ and $\psi:=(\psi_n)_{\omega}$. By Lemma \ref{UW3.3.1}, we get (taking the limit of the inequalities above): $d([\varphi],[\psi])\ge \lim_{n\to \omega}d(M_n)$. Hence $d(N)\ge \lim_{n\to \omega}d(M_n)$.   
\end{proof}
\begin{theorem}\label{thm: UP of type III nonzero factor is a factor}
Let $M$ be a $\sigma$-finite factor of type {\rm{III}}$_{\lambda} (\lambda\neq 0)$. Then $\prod^{\omega}M$ is a type {\rm{III}}$_{\lambda}$ factor. Moreover, for any sequence $(\varphi_n)_n\subset S_{\rm{nf}}(M)$, $(M,\varphi_n)^{\omega}\cong M^{\omega}$ is also a factor of type {\rm{III}}$_{\lambda}$. 
\end{theorem}
\begin{proof}
Let $p:=\text{supp}(\varphi_{\omega})$, where $\varphi_{\omega}=(\varphi_n)_{\omega}\in (M_*)_{\omega}$. Then by Proposition \ref{prop: normalizer iff commutes with p}, we have 
\[(M,\varphi_n)^{\omega}\cong pNp,\ \ N:=\prod^{\omega}M.\]
By Theorem \ref{thm: state space diameter theorem}, the state space diameter of $N$ is 
\[d(N)=\lim_{n\to \omega}d(M)=2\frac{1-\lambda^{\frac{1}{2}}}{1+\lambda^{\frac{1}{2}}}.\]
Hence $N$ is a type III$_{\lambda}$ factor, so is its corner $pNp$. Since all $\sigma$-finite projections in a type III factor are equivalent, all $(M,\varphi_n)^{\omega}$'s are mutually isomorphic. 
\end{proof}
\begin{remark}\label{rem: another proof of factoriality}
Let $M$ be a $\sigma$-finite factor of type {\rm{III}}$_{\lambda}$\ ($0<\lambda<1$). Then the factoriality of $M^{\omega}$ can be shown using Theorem \ref{thm: ultrapower of modular automorphism}. 
\end{remark}
 
\begin{proof}[Proof of Remark \ref{rem: another proof of factoriality}]
Let $x\in \mathcal{Z}(M^{\omega})$. Let $\varphi \in S_{\text{nf}}(M)$ be such that $\sigma_T^{\varphi}=\text{id}$, where $T=-2\pi/\log \lambda$. By Proposition \ref{prop: centralizer of lacunary weight}, we have 
\[x\in \mathcal{Z}(M^{\omega})\subset (M^{\omega})_{\varphi^{\omega}}=(M_{\varphi})^{\omega}.\]
Then by Takesaki's Theorem for periodic state \cite{Takesaki3}, $M_{\varphi}$ is a type II$_1$ factor and $\sigma(\Delta_{\varphi})=\{\lambda^n;\ n\in \mathbb{Z}\}\cup \{0\}$, whence $(M_{\varphi})^{\omega}$ is also a type II$_1$ factor by a standard argument. This shows that 
\[x\in (M_{\varphi})^{\omega}\cap (M^{\omega})'\subset \mathcal{Z}((M_{\varphi})^{\omega})=\mathbb{C}.\]
Therefore $M^{\omega}$ is a factor, and since $(M^{\omega})_{\varphi^{\omega}}$ is a factor, we have (cf. Corollary \ref{cor: some useful conseuences of main theorem} (3))
\[S(M^{\omega})=\sigma (\Delta_{\varphi^{\omega}})=\sigma(\Delta_{\varphi})=\{\lambda^n;\ n\in \mathbb{Z}\}\cup \{0\}.\]
This shows that $M^{\omega}$ is a type III$_{\lambda}$ factor.    
\end{proof}

\subsection{Ultraproduct of Type III$_0$ Factors}
As we have seen, in the case of type III$_{\lambda}\ (\lambda \neq 0)$ factor, the Ocneanu ultraproduct $(M,\varphi_n)^{\omega}$ does not depend on the choice of $(\varphi_n)_n$. In this section we see that the situation is different for the case of type III$_0$ factors. Moreover, we will show that $M^{\omega}$ is not a factor. 

\begin{theorem}\label{thm: UP of III_0 can be of finite type}
Let $M$ be a $\sigma$-finite type {\rm{III}}$_0$ factor. Then there exists a sequence $\{\varphi_n\}_{n=1}^{\infty}$ of normal faithful states on $M$ such that $(M,\varphi_n)^{\omega}$ is isomorphic to the finite von Neumann algebra $(M_{\varphi_n},\tau_n)^{\omega}$ where $\tau_n:=\varphi_n|_{M_{\varphi_n}}$.
\end{theorem}
We need lemmata.
\begin{lemma}{\rm{\cite[Lemma 2.3.4]{Connes2}}}\label{lem: intersection is Gamma} Let $\alpha$ be a continuous action of a locally compact abelian group $G$ on a factor $M$. Denote by $\widehat{G}$ the Pontrjagin dual of $G$. Then the family $\mathcal{F}$ of subsets of $\widehat{G}$ of the form $\{{\rm{Sp}}(\alpha^e)+K\}$, where $e$ is a non-zero projection in $M^{\alpha}$ and $K$ is a compact neighborhood of 0 in $\widehat{G}$, forms a directed set with intersection $\Gamma (\alpha)$.  
\end{lemma}

\begin{lemma}{\rm{\cite[Lemma 5.2.3]{Connes2}}}\label{lem: spectral gap lemma} Let $\sigma$ be a continuous action of $\mathbb{R}$ on a factor $M$, and assume there is $c>0$ such that 
\[{\rm{Sp}}(\sigma)\cap \{[-2c,-c]\cup [c,2c]\}=\emptyset,\]
where we identify $\widehat{\mathbb{R}}=\mathbb{R}$. Then there exists $h\in M^{\sigma}$, $-c/2\le h\le c/2$ such that the action $\sigma'$ of $\mathbb{R}$ defined by 
\[\sigma'_t(x):=e^{-ith}\sigma_t(x)e^{ith},\ \ \ \ x\in M, t\in \mathbb{R},\]
satisfies ${\rm{Sp}}(\sigma')\cap (-c,c)=\{0\}$. 
\end{lemma}

\begin{lemma}\label{lem: states with large spectral gaps}
Let $M$ be a $\sigma$-finite factor of type {\rm{III}}$_0$. Then for each $n\in \mathbb{N}$, there exists $\varphi_n\in S_{\rm{nf}}(M)$ such that ${\rm{Sp}}(\sigma^{\varphi_n})\cap (-\log n,\log n)=\{0\}$.
\end{lemma}
\begin{proof}
For $n\in \mathbb{N}$ and $\varepsilon>0$, define $I_n:=[-2\log n,-\log n]\cup [\log n,2\log n]$ and $K_{\varepsilon}:=[-\varepsilon,\varepsilon]$. 
Assume that there is $n\ge 2$ such that $\text{Sp}(\sigma^{\psi})\cap I_n\neq \emptyset$ for every $\psi \in S_{\rm{nf}}(M)$. Fix $\psi \in S_{\rm{nf}}(M)$. Let $e\in \text{Proj}(M_{\psi})\setminus \{0\}$. Since $M\cong eMe$, the assumption implies that $\text{Sp}(\sigma^{\psi_e})\cap I_n\neq \emptyset$. Now given finitely many $e_1,\cdots, e_N\in \text{Proj}(M_{\psi})\setminus \{0\}$ and $\varepsilon_1,\cdots ,\varepsilon_N>0$. By Lemma \ref{lem: intersection is Gamma}, there is $e\in \text{Proj}(M_{\psi})\setminus \{0\}$ and $\varepsilon>0$ such that 
\[\emptyset \neq I_n\cap \{\text{Sp}(\sigma^{\psi_e})+K_{\varepsilon}\}\subset I_n\cap \bigcap_{i=1}^N\{\text{Sp}(\sigma^{\psi_{e_i}})+K_{\varepsilon_i}\}.\]
Therefore by the compactness of $I_n$ and by Lemma \ref{lem: intersection is Gamma}, we have 
\eqa{
\emptyset &\neq I_n\cap \bigcap_{0\neq e\in M_{\psi}, \varepsilon>0}\{\text{Sp}(\sigma^{\psi_e})+K_{\varepsilon}\}\\
&=I_n\cap \Gamma(\sigma^{\psi})=I_n\cap \log (S(M)\setminus \{0\})\\
&=I_n\cap \{0\}=\emptyset,
}
which is a contradiction. Therefore for each $n\in \mathbb{N}$, there is $\psi_n\in S_{\rm{nf}}(M)$ such that $\text{Sp}(\sigma^{\psi_n})\cap I_n=\emptyset$. Then choose $h_n\in M_{\psi_n},\ -\frac{1}{2}\log n\le h_n\le \frac{1}{2}\log n$ as in Lemma \ref{lem: spectral gap lemma} for $\psi_n$. Then set $\varphi_n:=\psi_n(h'_n\cdot )$, $h'_n:=(h_n+\frac{1}{2}\log n+1)^{-1}$. Then we have $\varphi_n\in S_{\rm{nf}}(M)$ and $\text{Sp}(\sigma^{\varphi_n})\cap (-\log n,\log n)=\emptyset$.     
\end{proof}
\begin{proof}[Proof of Theorem \ref{thm: UP of III_0 can be of finite type}]
By Lemma \ref{lem: states with large spectral gaps}, for each $n\in \mathbb{N}$ there exists $\varphi_n\in S_{\text{nf}}(M)$ such that $\text{Sp}(\sigma^{\varphi_n})\cap (-\log n,\log n)=\{0\}$. Let $x=(x_n)^{\omega}\in (M,\varphi_n)^{\omega}$. By Proposition \ref{prop: characterization of the normalizer}, $x$ can be approximated strongly by elements of the form $(y_n)^{\omega}$, where $(y_n)_n$ satisfies $y_n\in M(\sigma^{\varphi_n},[-a,a])$ for each $n$ for a fixed $a>0$. Fix one such $(y_n)$ and $a>0$. Let $n_0\in \mathbb{N}$ be such that $a\le \log n_0$. Then by $\text{Sp}(\sigma^{\varphi_n})\cap (-\log n,\log n)=\{0\}$, for $n>n_0$ we have
\[M(\sigma^{\varphi_n},[-a,a])\subset M(\sigma^{\varphi_n},[-\log n_0,\log n_0])=M_{\varphi_n},\]
whence $(y_n)^{\omega}\in (M_{\varphi_n},\tau_n)^{\omega}$, where $\tau_n:=\varphi_n|_{(M_n)_{\varphi_n}}$. Since $(x_n)^{\omega}$ is approximated by these elements, $(x_n)^{\omega}\in (M_{\varphi_n},\tau_n)^{\omega}$ holds too. This finishes the proof. 
\end{proof}
\begin{remark}\label{rem: nonfactoriality of Raynaud UP of III_0}
Together with Proposition \ref{prop: normalizer iff commutes with p}, Theorem \ref{thm: UP of III_0 can be of finite type} shows that the Groh-Raynaud ultraproduct $\prod^{\omega}M$ of a type III$_0$ factor $M$ has a finite projection, and hence it is not a factor. Note that $M$ does not embed into the above $(M,\varphi_n)^{\omega}$. 

\end{remark}
Next we show that the Ocneanu ultraproduct of a type III$_0$ factor is not a factor. 
\begin{theorem}\label{thm: non-factoriality of UP(III_0)}
Let $M$ be a $\sigma$-finite factor of type {\rm III}$_0$. Then $M^{\omega}$ is not a factor. 
\end{theorem}

\begin{lemma}\label{lem: UP does not preserve ergodicity}
Let $A=L^{\infty}(X,\mu)$ be a (possibly non-separable) diffuse abelian von Neumann algebra, where $(X,\mu)$ is a probability space without atoms. Let $T$ be an ergodic transformation on $(X,\mu)$. Let $\alpha(f)(\omega):=f(T^{-1}\omega)$ be the corresponding automorphism of $A$. Then $\alpha^{\omega}\in \text{Aut}(A^{\omega})$ is not ergodic. 
\end{lemma}

\begin{remark}
Schmidt showed \cite[Proposition 2.2]{Schmidt} that if $(X,\mu)$ is a standard non-atomic probability space, there exists a measurable sets $\{B_n\}_{n=1}^{\infty}\subset X$ which are {\it non-trivial asymptotically $T$-invariant sets}. That is, it satisfies 
\[\lim_{n\to \infty}\mu(TB_n\bigtriangleup B_n)=0,\ \ \ \liminf_{n\to \infty}\mu(B_n)\mu(1-B_n)>0.\]
Therefore $p:=(1_{B_n})^{\omega}$ is a non-trivial projection in $(A^{\omega})^{\alpha^{\omega}}$, and Lemma \ref{lem: UP does not preserve ergodicity} follows. Since we could not check if his proof works for non-separable space $(X,\mu)$, we add a proof of Schmidt's result for non-separable space below (Lemma \ref{lem: almost invariant set}).  
\end{remark}

We need a slight modification of Rokhlin's Theorem from \cite{TakBook} due to Kawahigashi-Sutherland-Takesaki. We include a proof for reader's convenience. 
\begin{lemma}{\rm{\cite[Lemma 10]{KST}}}\label{lem: modified Rokhlin} Let $(\Omega,\mu)$ be a non-atomic probability space, $T\colon X\to X$ be a non-singular ergodic transformation. Then for each $n\in \mathbb{N}$ and $\varepsilon>0$ there exists a measurable subset $E\subset X$ such that 
\begin{itemize}
\item[{\rm{(1)}}] $E,T(E),\cdots, T^{n-1}(E)$ are mutually disjoint.
\item[{\rm{(2)}}] $\mu \left (X-\bigcup_{j=0}^{n-1}T^jE\right )<\varepsilon$.
\item[{\rm{(3)}}] $\mu(E)\le \frac{1}{n}$.
\end{itemize}  
\end{lemma}
\begin{proof}
Let $\nu_n:=\sum_{j=0}^{n-1}\mu\circ T^j$. Then $\nu$ is absolutely continuous with respect to $\mu$. Therefore given $\varepsilon>0$, there is $\delta=\delta(n,\varepsilon)>0$ such that 
\[\mu(F)<\delta\Rightarrow \nu_n(F)<\varepsilon\]
holds. This implies that $\mu(T^jF)<\varepsilon,\ 0\le j\le n-1$. By Rokhlin tower Theorem (see e.g., \cite[Lemma XVIII.3.2]{TakBook}), there exists a measurable set $F\subset X$ such that $F,TF,\cdots ,T^{n-1}F$ are mutually disjoint, and $G:=X-\bigcup_{j=0}^{n-1}T^jF$ has $\mu(G)<\delta(n,\varepsilon)$. In particular, we have $\mu(T^jG)<\varepsilon, 0\le j\le n-1$. Since $\sum_{j=0}^{n-1}\mu(T^jF)\le 1$, we may choose $k\in \{0,1,\cdots ,n-1\}$ with $\mu(T^kF)\le \frac{1}{n}$. Put $E:=T^kF$. Then $E,TE,\cdots, T^{n-1}E$ are mutually disjoint, and 
\[X-\bigcup_{j=0}^{n-1}T^jE=T^k\left (X-\bigcup_{j=0}^{n-1}T^jF\right )=T^kG,\]
whence $\mu(X-\bigcup_{j=0}^{n-1}T^jE)<\varepsilon$, and $\mu(E)\le \frac{1}{n}$.
\end{proof}
\begin{lemma}\label{lem: almost invariant set}
 Let $(X,\mu)$ be a non-atomic probability space, $T\colon X\to X$ be a non-singular ergodic transformation. For each $n\ge 2$, there exists a measurable set $B_n\subset X$ with $\mu(B_n)=\frac{1}{2}$ such that $\mu(TB_n\bigtriangleup B_n)\le \frac{2}{n}$ holds.
\end{lemma}
\begin{proof}
Put $\varepsilon=\frac{1}{n}$ and choose $E\subset X$ as in Lemma \ref{lem: modified Rokhlin}. Since $\mu$ has no atoms, there exists a family $\{G(t)\}_{t\in [0,1]}$ of measurable subsets of $X$ with the following properties:
\begin{itemize}
\item[(i)] $G(0)=\emptyset, G(1)=E$.
\item[(ii)] $0\le t\le s\le 1\Rightarrow G(t)\subset G(s)$.
\item[(iii)] $\mu(G(t))=t\mu(E),\ \ 0\le t\le 1$. 
\end{itemize}
Put 
\[B(t):=\bigcup_{j=0}^{n-1}T^jG(t),\ \ 0\le t\le 1.\]
We see that $B(0)=\emptyset, B(1)=\bigcup_{j=0}^{n-1}T^jE$, so that $\mu(B(1))>1-\frac{1}{n}\ge \frac{1}{2}$. Since $t\mapsto \mu(B(t))=\sum_{j=0}^{n-1}\mu(T^jG(t))$ is continuous by the choice of $\{G(t)\}_{t\in [0,1]}$, we can find $t_0\in [0,1]$ with $\mu(B(t_0))=\frac{1}{2}$. Then put $B_n:=B(t_0)$. Since $E,TE,\cdots, T^{n-1}E$ are disjoint, we see that 
\[B_n\bigtriangleup TB_n\subset T^nE\cup E\subset G\cup E,\]
where $G:=X-\bigcup_{j=0}^{n-1}T^jE$ (the last inclusion is true modulo null sets, since $\mu(T^nE\cap T^jE)=0$ for $1\le j\le n-1$). Therefore we have 
\[\mu(B_n\bigtriangleup TB_n)\le \mu(G)+\mu(E)\le \frac{2}{n}.\] 
\end{proof}

\begin{proof}[Proof of Lemma \ref{lem: UP does not preserve ergodicity}]
By Lemma \ref{lem: almost invariant set}, we can find measurable sets $B_n\subset X\ (n\in \mathbb{N})$ with $\mu (TB_n\triangle B_n)\le \frac{2}{n}$. Then put $p:=(1_{B_n})^{\omega}\in A^{\omega}$. By assumption, $p$ is a $\alpha^{\omega}$-invariant projection in $A^{\omega}\setminus \{0,1\}$. Hence $\alpha^{\omega}$ is not ergodic.
\end{proof}
We next show that for type {\rm{III}}$_0$ factors, discrete decomposition is preserved under the Ocneanu ultrapower.

Let $M$ be a type III$_0$ factor. There is a normal faithful lacunary weight $\varphi$ on $M$ such that $M_{\varphi}$ is of type II$_{\infty}$ with diffuse center, and let $\tau:=\varphi|_{M_{\varphi}}$. There is $0<\lambda_0<1$ and $U\in M(\sigma^{\varphi},(-\infty,\log \lambda_0])$ (the spectral condition follows from the proof of \cite[Th\'eor\`eme 5.3.1]{Connes1}) such that $\theta=\text{Ad}(U)|_{M_{\varphi}}\in \text{Aut}(M_{\varphi})$ is a centrally ergodic automorphism satisfying $\tau\circ \theta\le \lambda_0 \tau$. In this setting, we have $M\cong M_{\varphi}\rtimes_{\theta}\mathbb{Z}$ and $\varphi=\hat{\tau}$ (dual weight of $\tau$) under this isomorphism. We call this a discrete decomposition of $M$. Similar decompositions are possible for type III$_{\lambda} (0<\lambda<1)$ factors, in which case we have $\tau \circ \theta=\lambda \tau$ and $U\in M(\sigma^{\varphi},\{\log \lambda\})$ (see \cite{Connes1}). 

\begin{proposition}\label{prop: discrete decompositon is preserved}Let $M$ be a $\sigma$-finite factor of type {\rm{III}}$_0$ with discrete decomposition $M=M_{\varphi}\rtimes_{\theta}\mathbb{Z}$ ($\varphi$ is chosen as above). Then $M^{\omega}\cong (M_{\varphi})^{\omega}\rtimes_{\theta^{\omega}}\mathbb{Z}$. 
\end{proposition}

\begin{remark}
The discrete decomposition for type III$_{\lambda} (0<\lambda<1)$ factor is also preserved under the Ocneanu ultraproduct. The proof is exactly the same.  
\end{remark}
Recall that for a von Neumann algebra $M$ and $\theta\in \text{Aut}(M)$, $p(\theta)$ is the greatest projection $e\in M^{\theta}$ for which $\theta|_{M_e}$ is an inner automorphism. 
\begin{lemma}\label{lem: isomorphic to crossed product}
Let $N$ be a von Neumann subalgebra of $M$, and let $\theta\in \text{Aut}(N)$ be such that $p(\theta^k)=0$ for all $k\neq 0$. Suppose $N$ satisfies 
\begin{itemize}
\item[{\rm{(a)}}] $N'\cap M\subset N$.
\item[{\rm{(b)}}] There is a normal faithful conditional expectation $E$ from $M$ onto $N$.
\item[{\rm{(c)}}] There is $U\in \mathcal{U}(M)$ such that $UxU^*=\theta (x)$ for all $x\in N$.
\item[{\rm{(d)}}] $M$ is generated by $\{U\}\cup N$ as a von Neumann algebra.
\end{itemize}
Then there is a *-isomorphism $\Phi\colon M\to N\rtimes_{\theta}\mathbb{Z}$ sending $U$ (resp. $N$) to the canonical implementing unitary of $\theta$ (resp. the canonical image of $N$) in the crossed product.
\end{lemma}
\begin{proof} See \cite[Proposition 4.1.2]{Connes1}. Note that the central ergodicity assumption for $\theta$ in \cite{Connes1} is necessary only for the factoriality of $M$. 
\end{proof}

\begin{proof}[Proof of Proposition \ref{prop: discrete decompositon is preserved}]
We have to verify (a)-(d) in Lemma \ref{lem: isomorphic to crossed product} for $(M_{\varphi})^{\omega}\subset M^{\omega}$ and $\theta^{\omega}$. 
Let $U$ be the implementing unitary of $\theta$ in the discrete decomposition. Then we have $U\in M(\sigma^{\varphi},(-\infty,\log \lambda_0])$ for some $0<\lambda_0<1$, and $\varphi=\hat{\tau} \in \mathcal{W}_{\rm{nfs}}(M)$ is lacunary.  Also, $\tau \circ \theta \le \lambda_0 \tau$.

Since $\varphi$ is strictly semifinite, there is a normal faithful conditional expectation $E\colon M\to M_{\varphi}$. By Proposition \ref{prop: centralizer of lacunary weight}, we have $(M^{\omega})_{\varphi^{\omega}}=(M_{\varphi})^{\omega}$, so that the normal faithful $\varphi^{\omega}$-preserving conditional expectation coincides with $E^{\omega}\colon M^{\omega}\to (M_{\varphi})^{\omega}$. So (b), (c) is clearly satisfied. 
Regarding $p((\theta^{\omega})^k)$, note that 
$\tau^{\omega}:=\varphi^{\omega}|_{(M_{\varphi})^{\omega}}$ is a normal faithful semifinite trace satisfying $\tau^{\omega}\circ \theta^{\omega}\le \lambda_0\tau^{\omega}$. This implies, by the proof of \cite[Proposition 5.1.1]{Connes1}, that $p((\theta^{\omega})^k)=0$ for all $k\neq 0$.
 
Next, we show (d): $M^{\omega}$ is generated by $(M_{\varphi})^{\omega}$ and $U$ (canonical image of $U$ in $M^{\omega}$).
 Let $\{p'_i\}_{i\in I}$ be a net of projections in $M_{\varphi}$ such that $\tau(p'_i)<\infty\ (i\in I)$ and $p'_i\nearrow 1$ strongly. Then put $p_i:=\bigvee_{n=1}^{\infty}\theta^n(p'_i)$. Then it holds that
\[\tau(p_i)\le \sum_{n=1}^{\infty}\tau(\theta^n(p'_i))\le \sum_{n=1}\lambda^n\tau(p'_i)<\infty,\]
 and $\theta(p_i)\le p_i$, $p_i\nearrow 1$ strongly. 
Now fix one of such finite projection $p=p_i$ in $M_{\varphi}$, and we prove that $p(M^{\omega})p$ is generated by $pQp$, where  
 \[Q:=\text{span}\left ((M_{\varphi})^{\omega}\cup \bigcup_{k=1}^{\infty}(M_{\varphi})^{\omega}
U^k \cup \bigcup_{k=1}^{\infty}(U^*)^{k}(M_{\varphi})^{\omega} \right ).\]
By construction, each $x\in M=M_{\varphi}\rtimes_{\theta}\mathbb{Z}$ has a formal expansion 
\[x\sim x(0)+\sum_{k=1}^{\infty}\{x(k)U^k+(U^*)^kx(-k)\},\]
where $x(k)\in M_{\varphi}\ (k\in \mathbb{Z})$ is uniquely determined by 
\[x(k)=E(x(U^*)^k),\ \ \  x(-k)=E(U^kx)\ \ \ \ \ \  (k\ge 0)\] (the order of $U^{k}$ and $x(k)$ is a matter of convention). 
Let $x\in M^{\omega}$, and put $y:=pxp\in (M^{\omega})_p$ (here we used $p(M^{\omega})p=(pMp)^{\omega}$. See \cite[Proposition 2.10]{MasudaTomatsu}).  
Since $p\in (M^{\omega})_{\varphi^{\omega}}=(M_{\varphi})^{\omega}$, we may consider $\varphi^{\omega}_p$ as a faithful normal positive functional on $(M^{\omega})_p$. Let $\varepsilon>0$ be given. By Proposition \ref{prop: characterization of the normalizer}, we may find $a>0$ and $z=(z_n)^{\omega}\in (M^{\omega})_p$ with $z_n\in M_p(\sigma^{\varphi_p},[-a,a])\ (n\in \mathbb{N})$ such that $\|y-z\|_{\varphi^{\omega}_p}<\varepsilon$, and $\|z\|\le \|y\|$.
Consider the expansion of $z_n$ (in $M$):
\[z_n\sim z_n(0)+\sum_{k=1}^{\infty}\{z_n(k)U^{k}+(U^*)^kz_n(-k)\},\ \ \ \ \ \ \ n\in \mathbb{N}.\]
Let $V:=Up$. Then by $UpU^*=\theta(p)\le p$, we have
\eqa{
V^2&=UpUp=Up(UpU^*)U=U\theta (p)U\\
&=U^2p.
}
Similar computations show that 
\[V^k=U^kp,\ \ \ \ (V^*)^k=p(U^*)^k,\ \ \ \ k\ge 1.\]

We see that for $k\ge 0$,
\eqa{
z_n(k)&=E(z_n(U^*)^k)=E(z_np(U^*)^k)=E(z_n(V^*)^k),\\
z_n(-k)&=E(U^kpz_n)=E(V^kz_n).
}
In particular, we have $z_n(k)\in pM_{\varphi}\theta^k(p), z_n(-k)\in \theta^k(p)M_{\varphi}p$ and 
\[z_n(k)U^k=z_n(k)V^k,\ \ \ \  (U^*)^kz_n(-k)=(V^*)^kz_n(-k).\]
Therefore the expansion of $z_n$ can be rewritten as 
\[z_n\sim z_n(0)+\sum_{k=1}^{\infty}\{z_n(k)V^{k}+(V^*)^kz_n(-k)\},\ \ \ \ \ \ \ n\in \mathbb{N}.\] 
Since $U\in M(\sigma^{\varphi},(-\infty,\log \lambda_0])$,  we have 
\[V^kz_n\in M_p(\sigma^{\varphi_p},(-\infty,k\log \lambda_0+a]),\ \ \ z_n(V^*)^k\in M_p(\sigma^{\varphi_p},[-k\log \lambda_0-a,\infty))\]
for each $k\ge 1$. Let $K:=[a/(-\log \lambda_0)]+1\in \mathbb{N}$, and consider the GNS representation of $(\varphi_p, (M_{\varphi})_p)$. 
$E$ induces $\varphi_p$-preserving conditional expectation $E_p\colon M_p\to (M_{\varphi})_p$. Then for $k\ge K$, we have
\[k\log \lambda_0+a<0<-k\log \lambda_0-a\]
Hence 
\eqa{
z_n(k)\xi_{\varphi_p}&=E_p(z_n(V^*)^k)\xi_{\varphi_p}=1_{\{1\}}(\Delta_{\varphi_p})(z_n(V^*)^k\xi_{\varphi_p})=0,\\
z_n(-k)\xi_{\varphi_p}&=E_p(V^kz_n)\xi_{\varphi_p}=1_{\{1\}}(\Delta_{\varphi_p})(V^kz_n\xi_{\varphi_p})=0.
}
Since $\xi_{\varphi_p}$ is separating for $M_p$, we have $z_n(k)=0,\ |k|\ge K$.  
Therefore we have  
\[z_n=z_n(0)+\sum_{k=1}^{K-1}\{z_n(k)V^k+(V^*)^kz_n(-k)\},\ \ \ \ \ \ \ n\in \mathbb{N}.\]
Now, since $(M_{\varphi})_p$ is a finite von Neumann algebra, each $(z_n(k))_n\ (|k|\le K-1)$ is in $\mathcal{M}^{\omega}(\mathbb{N},M_{\varphi_p})$, and we have
\eqa{z=(z_n)^{\omega}&=(z_n(0))^{\omega}+\sum_{k=1}^{K-1}\{(z_n(k))^{\omega}U^k+(U^*)^{k}(z_n(-k))^{\omega}\}\\
&\in pQp.
}
Since $\varepsilon>0$ is arbitrary, $y=pxp$ can be approximated strongly by elements from $pQp$. Hence $pM^{\omega}p=\overline{pQp}^{\text{sot}}$. 
Since $i$ is arbitrary (recall that $p=p_i$), this implies that $x$ is in $\overline{Q}^{\text{sot}}$ as well. This proves the claim. 
 
Finally, by the proof of \cite[Proposition 4.1.1]{Connes1}, (a) is proved. Namely, let $x\in ((M_{\varphi})^{\omega})'\cap M^{\omega}$. Then by the above, $x$ has a formal expansion by $x\sim x(0)+\sum_{k=1}^{\infty}(x(k)U^k+(U^*)^kx(-k)\ (x(k)\in (M_{\varphi})^{\omega})$. Since $ax=xa$ for $a\in (M_{\varphi})^{\omega}$, this implies that $ax(k)=x(k)\theta^k(a)$, $\theta^k(a)x(-k)=x(-k)a$ for all $a\in (M_{\varphi})^{\omega}$ and $k\ge 0$. Then by $p((\theta^{\omega})^k)=0\ (k\neq 0)$ and by \cite[Remarques 1.5.3 (a)]{Connes1}, we have $x(k)=0, k\neq 0$ and hence $x=x(0)\in (M_{\varphi})^{\omega}$. This proves that $((M_{\varphi})^{\omega})'\cap M^{\omega}\subset (M_{\varphi})^{\omega}$   
\end{proof}

Now we are ready to prove the non-factoriality of the Ocneanu ultrapower for type III$_0$ factors.

\begin{proof}[Proof of Theorem \ref{thm: non-factoriality of UP(III_0)}]

Now we show that $M^{\omega}$ is not a factor. By Claim 1, $M^{\omega}$ is generated by $(M_{\varphi})^{\omega}$ and $U^{\omega}$, which implements $\theta^{\omega}$. Representing the center of $M_{\varphi}$ as $L^{\infty}(X,\mu)$ where $(X,\mu)$ is a diffuse probability space, \cite[Lemma 2.8]{MasudaTomatsu}, together with \cite[Corollary 4.2]{FHS1} implies that the center of $(M_{\varphi})^{\omega}$ is $L^{\infty}(X,\mu)^{\omega}$. By Lemma \ref{lem: UP does not preserve ergodicity}, $\theta^{\omega}$ is not centrally ergodic. This implies that there is a nontrivial element $x\in (L^{\infty}(X,\mu)^{\omega})^{\theta^{\omega}}$, whence a nontrivial element in $\mathcal{Z}(M^{\omega})$. Therefore $M^{\omega}$ is not a factor. 
\end{proof}

\begin{remark}The above result reproduces Connes's result \cite{Connes3} that $\sigma$-finite type III$_0$ factor $M$ cannot be full. For if such $M$ satisfies $M_{\omega}=\mathbb{C}$, then $M'\cap M^{\omega}=\mathbb{C}$ holds by Proposition \ref{prop: no difference between M_omega and M'cap Momega}. Since $\mathcal{Z}(M^{\omega})\subset M'\cap M^{\omega}$, this shows that $M^{\omega}$ is a factor, a contradiction. Therefore $M_{\omega}\neq \mathbb{C}$.
\end{remark}

\section*{Acknowledgement}
The authors thank Yoshimichi Ueda for discussions regrading his problem on full factors, Toshihiko Masuda and Reiji Tomatsu for their useful comments on the paper. HA is supported by EPDI/JSPS/IH\'ES Fellowship. Most of the work was done during his visits to University of Copenhagen since August 2010. He thanks for their hospitality. He also thanks to Kyoto University GCOE program for financial support during his first visit to Copenhagen. UH is supported by ERC Advanced Grant No. OAFPG 27731 and the Danish National Research Foundation through the Center for Symmetry and Deformation. Last but not least, the authors would like to thank anonymous referee for numerous suggestions which improved the presentation of the paper.\\ \\
{\it Notes added in Proof} After the paper has been submitted, the authors were informed from Eberhard Kirchberg of his ultraproduct \cite{Kirchberg} for a given sequence $(A_n,\varphi_n)_n$ of pairs of C$^*$-algebras and faithful states, which are (after some identification) in fact a generalization of Ocneanu's construction to C$^*$-algebras. In particular, he proved a result \cite[Proposition 2.1 (iv)]{Kirchberg} which is essentially the same as Theorem 4.1 in the case of von Neumann algebras. Our proof is different from his method. 

Hiroshi Ando\\
Institut des Hautes \'Etudes Scientifiques,\\
Le Bois-Marie 35, route de Chartres,\\
 91440 Bures-sur-Yvette, France\\
ando@ihes.fr 
\\ \\
Uffe Haagerup\\
Department of Mathematical Sciences,\\
University of Copenhagen\\
Universitetsparken 5,\\
2100 K\o benhavn \O, Denmark\\
haagerup@math.ku.dk
\end{document}